\documentclass{article}
\usepackage{graphicx} 
\usepackage{subfigure}
\usepackage{amsmath}
\usepackage{amsfonts}
\usepackage{amsthm}
\usepackage{amssymb}
\usepackage{mathrsfs}
\usepackage {color}
\usepackage {xcolor}
\usepackage{indentfirst}
\usepackage{float}
\graphicspath{{images/}} 
\usepackage[all]{xy}
\usepackage{hyperref}
\usepackage[linkcolor=red,anchorcolor=blue, colorlinks=red]{adjustbox} 

\providecommand{\subjclass}[2][]{%
  \textit{Mathematics Subject Classification (2020):} #2%
}

\providecommand{\keywords}[1]{%
  \textbf{Keywords and phrases:} #1%
}

\title{Generalized Hyperbolic Conical Circle Packings associated with Finite Polygonal Decompositions of Surfaces with Boundary
}
\date{}
\author{Guangming Hu, Jun Hu, Lishan Li and Yi Qi}

\begin{document}

\maketitle

\newtheorem{definition}{Definition}[section]
\newtheorem{theorem}[definition]{Theorem}
\newtheorem{proposition}[definition]{Proposition}
\newtheorem{lemma}[definition]{Lemma}
\newtheorem{remark}[definition]{Remark}
\newtheorem{corollary}[definition]{Corollary}

\begin{abstract}
Let $\mathcal{S}$ be a compact topological surface with finitely many genus and finitely many holes and let $\mathcal{D}$ be a polygonal decomposition of $\mathcal{S}$. In this paper, we consider the generalized hyperbolic conical circle packings
associated with $\mathcal{D}$. We first show that the boundary value problem has a unique solution $\textbf{k}$ by prescribing total geodesic curvatures of generalized hyperbolic conical circles centered at interior vertices and geodesic curvatures of generalized hyperbolic conical circles centered at boundary vertices.
Then we show that such a solution $\textbf{k}$ can be obtained by taking a limit of the packings inductively modified by Thurston's algorithm via an arbitrarily chosen initial generalized hyperbolic conical circle packing associated with $\mathcal{D}$ and with given boundary values. Thirdly, we develop the so-called discrete Schwarz-Pick lemma for the solution packing $\textbf{k}$ on $\mathcal{D}$.
\end{abstract}



\section{Introduction}
Discretizing the conformal structures on smooth surfaces has been a robust area of research in differential geometry and hyperbolic geometry, which has applications to the theory of general relativity. Circle packings and their combinatorial patterns serve as a collection of bricks of different shapes and a system of the combinatorial relationships among the bricks that characterize how the bricks are used to construct or approximate a conformal structure on a surface with or without boundary. 
 

In his work on the study of hyperbolic structures on 3-manifolds, Thurston \cite{thurston2022geometry} investigated circle packings on triangulated closed surfaces, revealing profound connections between circle packings and conformal structures on triangulated surfaces. In 1985, he presented a method to approximate a Riemann mapping from a simply connected hyperbolic domain to the open unit disk by discrete conformal mappings constructed through the use of circle packings. Approximating maps utilize regular hexagonal circle packings of the domain and corresponding packings on the unit disk, which are called Thurston finite Riemann mappings. The convergence of Thurston finite Riemann mappings to the Riemann mapping was proved by 
Rodin and Sullivan two years later in \cite{MR906396}. Since 1985, research has been successful in identifying the types of graph and associated data that ensure the existence and uniqueness of corresponding circle packings or circle patterns in Euclidean, spherical, and hyperbolic geometry. 
The classical Perron method or Colin de Verdière's variational principle in \cite{MR1106755} plays an important role in establishing the existence and uniqueness of circle patterns in hyperbolic or Euclidean background geometry. Under the geometry of the spherical metric, the uniqueness may fail if the prescribed data are not imposed on appropriate objects, similar to the McOwen-Troyanov theorem \cite{MR1224820, MR1005085} if prescribed data are varied under M\"{o}bius transformations. Applying Colin de Verdière's variational method, Nie obtained in \cite{MR4683863} the uniqueness of a circle pattern in the spherical background geometry by prescribing the total geodesic curvatures of the circles. A different proof of Nie's result, using the classical Perron method, is given in \cite{li2026perronsmethodsphericalideal}.

In\cite{MR2496046}, Guo-Luo investigated generalized circle packings of ten generalized hyperbolic triangles, proving the metric is uniquely determined by discrete curvature for six symmetric cases, which extends Thurston, Penner and Bobenko-Springborn’s rigidity results. However, for the generalized circle packing of types (1,1,-1) (triangles with 2 circles and 1 hypercycle) and (1,1,0) (triangles with 2 circles and 1 horocycle), the rigidity theorem fails to hold, leaving these cases unsolved. In \cite{MR4962537}, Ba-Hu-Sun applied  the total geodesic curvature  introduced by Nie \cite{MR4683863} to prove rigidity for all ten triangle types, settling the open problems in \cite{MR2496046}. Then the so-called discrete boundary value problem and discrete Schwarz-Pick lemma are considered in \cite{2024arXiv241106274H} for such generalized hyperbolic conical circle packings associated with triangulated oriented compact surfaces with boundaries. With respect to finite polygonal decompositions of compact surfaces with boundaries, the existence and uniqueness are obtained in \cite{MR4842734} for generalized circle packings by prescribing total geodesic curvatures of circles associated with all vertices. 

In this paper, we continue to consider generalized hyperbolic conical circle packings associated with polygonal decompositions of compact surfaces with boundaries, but the prescribed data that we use are total geodesic curvatures of the circles associated with interior vertices and geodesic curvatures of the circles associated with boundary vertices. We briefly call such a problem {\em a discrete boundary value problem for generalized hyperbolic conical circle packings}. 

The work in this paper is three-fold: (1) Apply the Perron method to prove the existence and uniqueness of the solution for the discrete boundary value problem for generalized hyperbolic conical circle packings associated with polygonal decompositions of compact surfaces with boundaries, which generalizes the same result obtained in \cite{2024arXiv241106274H} for a special case when graphs are induced by triangulations of the surfaces and employs a method different from Colin de Verdière's variational principle used in \cite{2024arXiv241106274H}. (2) Apply Thurston's algorithm to approximate the solution of the discrete boundary value problem and show the convergence of the approximations to the solution. (3) Establish a discrete Schwarz-Pick lemma for the solution of the discrete boundary value problem.

\subsection{Discrete boundary value problem}


In their seminal work in 1991, Beardon and Stephenson \cite{MR1115988} investigated the boundary value problem for circle packings associated with triangulations on a disk by prescribing the radii of hyperbolic circles centered at boundary vertices. Techniques involving normal families are used in their approach. Lately, the discrete boundary value problem is resolved in \cite{2024arXiv241106274H} for generalized hyperbolic conical circle packings associated with triangulations of compact surfaces with finitely many boundaries by applying Colin de Verdière's variational principle. The first result of this paper is to generalize the result of \cite{2024arXiv241106274H} for generalized hyperbolic conical circle packings associated with triangulations of compact surfaces with boundaries to hyperbolic circle packings associated with polygonal decompositions of compact surfaces. 

By a \textbf{\em generalized hyperbolic circle} on the hyperbolic plane $\mathbb{H}^2$ we mean a hyperbolic circle, a horocycle or a hypercycle (see Definition \ref{def of three types of curve} and Figure \ref{Fig1}). Each of them has a constant curvature, which we denote by $k$. For any hyperbolic circle, $k>1$; for any horocycle, $k=1$; for any hypercycle,  $0<k<1$. Any segment on a generalized hyperbolic circle is viewed/called a \textbf{generalized hyperbolic conical circle} (see Definition \ref{generalized conical circle} and Figure \ref{Fig2}).

In this paper, we consider generalized hyperbolic conical circle packings, associated with graphs induced by polygonal decompositions of surfaces (with or without boundaries), satisfying that the tangent points among the conical circles centered at the vertices of a face lie on a hyperbolic circle for each face. This condition is automatically satisfied when the polygonal decomposition is a triangulation. 

Now we are ready to state the first main result. We apply the Perron method, instead of Colin de Verdière's variational principle, to prove the following theorem.






\begin{theorem}\label{theorem-1.1} 
Let $S_{g,n}$ be a compact oriented topological surface with finitely many genus $g\geq 0$ and finitely many holes $n\geq 1$. Let $\mathcal{D}$ be a polygonal cellular decomposition of $S_{g,n}$ with the set of vertex $V$, a disjoint union of the set of boundary vertices $V^{\partial}$ and the set of interior vertices $V^{\circ}$. 

Given $\hat{\textbf{k}}:V^{\partial}\rightarrow \mathbb{R}_{+}$ and $\hat{\textbf{T}}:V^{\circ}\rightarrow \mathbb{R}_{+}$, the following two statements are equivalent:

\begin{enumerate}
    \item The discrete boundary value problem has a unique solution in the sense that there exists a unique generalized hyperbolic conical circle packing up to isometry (see Figure \ref{Fig5} for an illustration) associated with the polygonal decomposition $\mathcal{D}$ 
    with geodesic curvatures $\textbf{k}:V\rightarrow \mathbb{R}_{+}$ satisfying:
    \begin{itemize}
    \item[(i)] For each vertex $w\in V^{\partial}$, the geodesic curvature $\textbf{k}(w)$ of the corresponding generalized circle is equal to $\hat{\textbf{k}}(w)$;
    \item [(ii)] For each vertex $v_0\in V^{\circ}$, the total geodesic curvature $\textbf{T}_{\textbf{k}}(v_0)$ of the corresponding generalized circle is equal to $\hat{\textbf{T}}(v_0)$.
    \end{itemize}

    \item{$\hat{\textbf{T}}\in \mathcal{T}$ with
    \begin{equation}\label{equa-1} 
\mathcal{T}=\left\{\textbf{T}: V^{\circ}\rightarrow\mathbb{R}_{+} \begin{array}{|l} \sum_{v \in Q} T_{v}<\sum_{f \in F_Q} \pi \min \{N(f, Q), N(f)-2\} \\
\text{for any nonempty set } Q\subset V^{\circ}  \end{array}\right\}, 
\end{equation} 
where
$F_Q=\left\{f \in F: \exists \;v \in Q \text { such that } v \text { is a vertex of } f\right\}$, $V(f)$ is the vertex set of $f$,
$N(f)=\sharp V(f)$, and 
$
N(f,Q)=\sharp (V(f)\cap Q)$.}
\end{enumerate}

\end{theorem}

\begin{remark}
One may apply Colin de Verdière's variational principle to give an alternative proof of Theorem \ref{theorem-1.1} since that method is successfully used in \cite{2024arXiv241106274H} to prove the special case when $\mathcal{D}$ is a  triangulation of a surface $\mathcal{S}$. Our paper deals with the general case that $\mathcal{D}$ is a polygonal decomposition of $\mathcal{S}$ and emphasizes an application of the Perron method, which is different from the variational method. 
\end{remark}

\subsection{Thurston's algorithm}\label{Thurston algorithm} In one of the appendices of \cite{MR906396},
Rodin and Sullivan described Thurston's algorithm for finding the Andreev circle packing for any triangulation of the extended complex plane. They also pointed out it converges quite rapidly.  Subsequently, Carter and Rodin \cite{MR1081937} developed a computational algorithm based on a modified Perron method to solve the boundary value problem in the Euclidean setting. Lately, the Thurston algorithm is applied in \cite{li2026perronsmethodsphericalideal} to approximate ideal spherical circle patterns associated with polygonal cellular decompositions of compact oriented surfaces. In this paper, 
we show that the Thurston algorithm continues to be valid in the course of finding the solution of the discrete boundary value problem in the setting of this paper. 

We continue to use the notation $S_{g,n}$, $\mathcal{D}$, $V^{\partial}$, and $V^{\circ}$ introduced in the previous subsection. Let $\hat{\textbf{k}}:V^{\partial}\rightarrow \mathbb{R}_{+}$ and $\hat{\textbf{T}}:V^{\circ}\rightarrow \mathbb{R}_{+}$, $\hat{\textbf{T}}\in \mathcal{T}$.
We call the circle packing from Theorem \ref{theorem-1.1} the target circle packing and denote it by $\hat{\mathcal{P}}$.

Given any initial vector 
$$\textbf{k}^0:V\longrightarrow\mathbb{R}_{+}: v \longmapsto k_v^0$$
with $k^0_w=\hat{\textbf{k}}(w)$ for any $w\in V^{\partial}$, there is a generalized circle packing $\mathcal{P}^0$ such that the geodesic curvature of the generalized circle centered at each vertex $v\in V$ is equal to $\textbf{k}^0(v)=k^0_v$. 

For any $v_i\in V^{\circ}=\{v_1,\cdots,v_N\}$, we modify $k_{v_i}^0$ to $k_{v_i}^1\in \mathbb{R}_{+}$, without changing the value of $\textbf{k}^0$ at any other vertex, so that the total geodesic curvature of the generalized circle centered at $v_i$ is equal to the prescribed value $\hat{\textbf{T}}(v_i)$; that is, the curvature value at each boundary point $w\in V^\partial$ remains at $\hat{\textbf{k}}(w)$ and 
the value of $k_{v_i}^1$ is determined by the equation
$$
\textbf{T}_{(k^0_{v_1},\cdots,k^0_{v_{i-1}},k^1_{v_{i}},k^0_{v_{i+1}},\cdots,k^0_{v_{N}})}(v_i)=\hat{\textbf{T}}(v_i).
$$

We define a new vector of geodesic curvatures as
$$\textbf{k}^1:V\longrightarrow\mathbb{R}_{+}: v \longmapsto k_v^1,$$
with $k^1_v=k^0_v=\hat{\textbf{k}}(v)$ for any $v\in V^{\partial}$. Then there is a generalized circle packing $\mathcal{P}^1$ such that the geodesic curvature of the generalized circle centered at each vertex $v\in V$ is equal to $\textbf{k}^1(v)=k^1_v$, which we call the $1^{st}$ approximation of the target circle packing $\hat{\mathcal{P}}$.

By replacing $\textbf{k}^0$ by $\textbf{k}^1$ and repeating the above process, we obtain $\textbf{k}^2$ and the $2^{nd}$ approximation $\mathcal{P}^2$ of $\hat{\mathcal{P}}$. Inductively, given $\textbf{k}^m$, define $\textbf{k}^{m+1}(v_i)$	to be the unique number $k^{m+1}_{v_i}$ such that $$
\textbf{T}_{(k^m_{v_1},\cdots,k^m_{v_{i-1}},k^{m+1}_{v_{i}},k^m_{v_{i+1}},\cdots,k^m_{v_{N}})}(v_i)=\hat{\textbf{T}}(v_i).
$$

All $\textbf{k}^{m+1}({v_i})$ are solved independently with the other interior coordinates fixed at their old values $\textbf{k}^m$; after all scalar equations have been solved, we set $\textbf{k}^{m+1}({v_i})=k^{m+1}_{v_i}$ for $i=1,\cdots,N$. Therefore, we obtain $\textbf{k}^{m+1}$ and the ${(m+1)}^{th}$ approximation $\mathcal{P}^{m+1}$ of $\hat{\mathcal{P}}$. We define ${\textbf{k}}^{*}:V\rightarrow\mathbb{R}_{+}$ as the vector of geodesic curvature of the generalized circles in  $\hat{\mathcal{P}}$.

The second part of this paper is to prove the convergence of this algorithm. 
\begin{theorem}\label{theorem-1.3}

Let \(\hat {\textbf{k}}:V^\partial\to\mathbb R_+\) and let
\(\hat {\textbf{T}}\in\mathcal T\), where \(\mathcal T\) is defined by $(\ref{equa-1})$ in Theorem \ref{theorem-1.1}. Then for any initial generalized circle packing $\mathcal{P}^0$, determined by an arbitrary vector 
$$\textbf{k}^0:V\longrightarrow\mathbb{R}_{+}: v \longmapsto k_v^0$$
with $\textbf{k}^0(w)=\hat{\textbf{k}}(w)$ at each $w\in V^{\partial}$, the sequence $\textbf{k}^m$ converges to ${\textbf{k}}^{*}$ where ${\textbf{k}}^{*}$ is the geodesic curvature vector
of the unique solution given by Theorem \ref{theorem-1.1}; equivalently,

$$
{\textbf{k}}^{*}(w)=\hat{\textbf{k}}(w), w\in V^{\partial},\quad \textbf{T}_{\textbf{k}^*}(v)=\hat{\textbf{T}}(v), v\in V^{\circ}.
$$

Consequently, the associated packings $\mathcal{P}^m$ converge, up to isometry,
to the target packing $\hat{\mathcal{P}}$, equivalently in the sense that their
geodesic curvature vectors converge to \(\textbf{k}^*\).
\end{theorem}

\begin{remark}
The other method for the existence of circle packings with prescribed pattern/data is to study systems of ordinary differential equations that govern certain types of combinatorial curvature flows and investigate if the asymptotic limits of the solutions as the time parameter going to $\infty$ realize the desired circle packings. Here are some works in this direction. Parallel to the Ricci flow in differential geometry, Chow and Luo \cite{MR2015261} introduced the combinatorial Ricci flow to determine circle patterns with prescribed curvatures in both Euclidean and hyperbolic geometries. To address constant-curvature problems, Luo \cite{MR2100762} developed a combinatorial Yamabe flow on triangulated surfaces. Inspired by the Calabi flow in differential geometry, Ge \cite{GePhDthesis},\cite{MR3729504},\cite{MR3818085} established combinatorial Calabi flow to construct circle patterns with prescribed curvatures in Euclidean geometry. Then numerical methods to approximate solutions to these systems of differential equations provide different algorithms to find the circle packings with prescribed pattern/data, see \cite{MR4466650},\cite{MR4498159},\cite{MR4127872},\cite{MR3566936},\cite{MR4078690},\cite{MR3269185},\cite{MR3613434}.
\end{remark}

\subsection{Discrete Schwarz-Pick lemma}
The Schwarz-Pick lemma, a fundamental result in complex analysis, states that any holomorphic map $f$ from the open unit disk $\mathbb{D}$ into itself does not increase the hyperbolic distance between points on $\mathbb{D}$. 

In the context of hyperbolic circle packings, Beardon and Stephenson \cite{MR1115988} established a discrete version of the Schwarz-Pick lemma by comparing lengths of geodesics between vertices and areas of interstices when the boundary values are changed. Lately, this result has been extended in \cite{2024arXiv241106274H} to generalized hyperbolic conical circle packings associated with triangular decompositions of surfaces. In this paper, we continue to extend this result to generalized hyperbolic conical circle packings associated with finite polygonal cellular decompositions of surfaces, and furthermore we find two more properties. 

Consider two generalized hyperbolic conical circle packings on $S_{g,n}$ underlying the graph $G$ and having the same total geodesic curvature values at the interior vertices. Assume that the geodesic curvature values at the boundary vertices for one packing uniformly dominate those for the other. Our goal, for a version of the discrete Schwarz-Pick lemma, is to establish comparison relationships in the following four aspects.
\begin{itemize}
    \item [(1)] The geodesic curvatures of all generalized circles;

    \item [(2)] The lengths of geodesics between vertices (circle centers);

    \item [(3)] The areas of the interstices;

    \item [(4)] The lengths of the arcs that bound interstices. 
\end{itemize}
More precisely, we show the following theorem. 

\begin{theorem}[Discrete Schwarz-Pick lemma]\label{theorem-1.4}
Let $S_{g,n}$ be a compact oriented topological surface of genus $g\geq 0$ with $n\geq 1$ boundary components, and let $\mathcal{D}$ be a polygonal cellular decomposition of $S_{g,n}$. Let $\hat{\textbf{k}}_{i}:V^{\partial}\rightarrow \mathbb{R}_{+}$ and $\hat{\textbf{T}}_{i}:V^{\circ}\rightarrow \mathbb{R}_{+}$  $,\;i=1,2$, where $\hat{T}_i\in\mathcal{T}$ and $\mathcal{T}$ is defined by $(\ref{equa-1})$ in Theorem \ref{theorem-1.1}. Let $\mathcal{P}_i$ be the generalized circle packing on $(S_{g,n},\mathcal{D})$ such that $\hat{\textbf{k}}_i(w)$ gives the geodesic curvature of the generalized circle centered at each boundary vertex $w\in V^{\partial}$, and $\hat{\textbf{T}}_i(v)$ gives the total geodesic curvature at each interior vertex $v\in V^{\circ}$. Denote by $\textbf{k}_i:V\rightarrow\mathbb{R}_{+}$ the geodesic curvature vector of $\mathcal{P}_i$, where $\textbf{k}_i(w)=\hat{\textbf{k}}_i(w)$ for all $w\in V^{\partial}$ and $\textbf{T}_{\textbf{k}_i}(v)=\hat{\textbf{T}}_i(v)$ for all $v\in V^{\circ}$. Assume that $\hat{\textbf{k}}_{1}(w) \leq \hat{\textbf{k}}_{2}(w)$ for all $w \in V^{\partial}$, and $\hat{\textbf{T}}_{1}(v) = \hat{\textbf{T}}_{2}(v)$ for all $v \in V^{\circ}$. Then the following statements hold:

\begin{itemize}
    \item [(1)] Vertex-wise Comparison (Sub harmonicity)$:$  

\({\textbf{k}}_{1}(v) \leq {\textbf{k}}_{2}(v), \forall \;v\in V\).

    \item [(2)] Area Comparison$:$  

   \(Area_{\textbf{k}_{1}}(\Omega_f) \geq Area_{\textbf{k}_{2}}(\Omega_f)\), $\forall$ \(f\in F\), where \(Area_{\textbf{k}_{i}}(\Omega_f)\) denotes the hyperbolic area of the interstice $\Omega_f$ (see Figure \ref{Fig8}) corresponding to the face \(f\)  on \(\tilde{S}(\textbf{k}_{i})\), $i=1, 2$.

    \item [(3)] Arc Length Comparison$:$

    \(\ell_{\textbf{k}_{1}}(v, f) \geq \ell_{\textbf{k}_{2}}(v, f)\), $\forall$ $f\in F$, where \(\ell_{\textbf{k}_i}(v, f)\) denotes the length of the edge (see Figure \ref{Fig19}) of the interstice $\Omega_f$ (corresponding to the face \(f\))) that lies on the generalized circle centered at $v$, $i=1, 2$.

    \item [(4)] Distance Comparison$:$ 

   \(\rho_{\textbf{k}_{1}}(v_1, v_2) \geq \rho_{\textbf{k}_{2}}(v_1, v_2)\), $\forall$ \(v_1, v_2 \in \tilde{V} = \{v \in V \mid \textbf{k}_{1}(v) > 1\}\), where \(\rho_{\textbf{k}_{i}}\) denotes the distance function on the hyperbolic surface \(\tilde{S}(\textbf{k}_{i})\) (see Definition \ref{def-2.5}), $i=1, 2$. 

    \item [(5)] Uniqueness Criteria$:$ 
    
    The generalized circle packing $\mathcal{P}_1$ and $\mathcal{P}_2$ are the same (that is; \(\textbf{k}_{1} = \textbf{k}_{2}\)) provided that the equality holds in any of the following cases:
\begin{itemize}
    \item [(i)] At some interior vertex in (1). 
      
    \item [(ii)] At some some interstice $\Omega_f$ (which has at least one interior vertex) in (2). 
    \item  [(iii)] At some edge $\ell_{\boldsymbol{k}_1}(v,f)$ (with $v\in V^{\circ}$) of the interstice $\Omega_f$ in (3).
    \item [(iv)] At two vertices \(u\) and \(v\) (with at least one of them in $V^{\circ}$) in (4).
\end{itemize}
\end{itemize}

\end{theorem}
\begin{remark} For the special case where $\mathcal{D}$ is a triangular decomposition of $S_{g, n}$, properties (1), (2), and (3) of the previous theorem have been obtained in \cite{2024arXiv241106274H}. 
\end{remark}
\bigskip

The remainder of the paper is organized as follows. Section~2 provides background information for generalized hyperbolic conical circle packings associated with polygonal decompositions $(S_{g,n},\mathcal D)$, including generalized
circle configurations, induced hyperbolic polygons and packing metrics, and
generalized stars. Section~3 gives a summary on local variation formulas and monotonic properties of some variables involved in the formulas.
In Section~4, we show how geodesic segments on a generalized circle configuration have their lengths varied when the curvatures of the circles are changed.  
In Section~5, we apply the Perron method, through the construction of generalized subpackings and superpackings, to prove the existence and uniqueness result for
the discrete boundary value problem (Theorem~\ref{theorem-1.1}). In Section~6, we analyze
Thurston's iterative algorithm and prove its convergence to the unique target packing (Theorem~\ref{theorem-1.3}). In Section~7, we establish the discrete
Schwarz--Pick lemma (Theorem~\ref{theorem-1.4}) by deriving vertex-wise,
area, arc-length, and distance comparison results, and a maximum
modulus principle for generalized circle packings in the hyperbolic metric setting (Theorem~\ref{theorem-6.5}).

\section{Preliminary}
\subsection{Generalized hyperbolic conical circle configuration}\label{sec-2.1} 
Let us first give the definition of generalized hyperbolic circle.
\begin{figure}[ht]
{
\centering
\includegraphics[height=4cm]{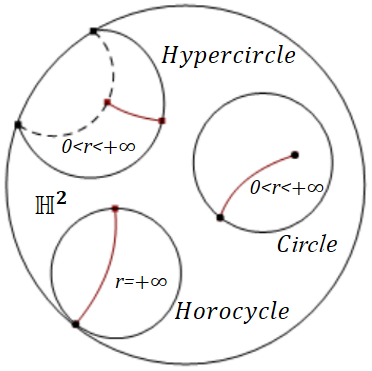} 
\caption{Generalized hyperbolic circles in $\mathbb{H}^2$.}
\label{Fig1}
}
\end{figure}

\begin{definition}[Circle, horocycle and hypercycle in $\mathbb{H}^2$]\label{def of three types of curve} Denote by $\mathbb{H}^2$ the two-dimensional hyperbolic plane, modeled by the upper half plane or the open unit disk, and denote by $\partial\mathbb{H}^2$ the boundary of $\mathbb{H}^2$.
\begin{itemize}
\item [(1)] A hyperbolic \textbf{circle} \( C(z_0, r) \) centered on \( z_0 \in \mathbb{H}^2 \) and of radius \( r \) is defined as the locus

\[
C(z_0, r) = \{ z \in \mathbb{H}^2 : d(z, z_0) = r \},
\]

where \( d(z, z_0) \) denotes the hyperbolic distance between \( z \) and \( z_0 \). 

    \item [(2)] By a \textbf{horocycle} we mean a circle on $\mathbb{H}^2$ that is tangent to $\partial\mathbb{H}^2$, which is viewed as a hyperbolic circle  centered at the tangent point and of radius $\infty$. 

    \item [(3)] Given a geodesic \( \gamma \) on \( \mathbb{H}^2 \) and a non-negative real number $r$, a \textbf{hypercycle} with axis \( \gamma \) and of distance \( r \) is defined as a connected component of the set
\[
C(\gamma, r) = \{ z \in \mathbb{H}^2 : d(z, \gamma) = r \},
\]

where \( d(z, \gamma) \) represents the minimal hyperbolic distance from \( z \) to \( \gamma \). We view $\gamma $ as the ``center" of the hypercycle and $r$ as the radius of the hypercycle. 
\end{itemize}
\end{definition}
Note that (i) the center of a hyperbolic circle is a point on \( \mathbb{H}^2 \); (ii) the center of a horocycle is an ideal point on the boundary \( \partial\mathbb{H}^2 \); (iii) the center of a hypercycle is defined as the geodesic in \( \mathbb{H}^2 \) connecting the boundary points of the hypercycle. Each of them has a constant geodesic curvature $k$, and $0<k<1$ for a hypercycle, $k=1$ for a horocycle, and $k>1$ for a hyperbolic circle. Furthermore, $r$ is expressed as a function of $k$ as follows:

\begin{equation}\label{equa-4}
 r(k)= \begin{cases}\operatorname{arctanh} k & \text { if } 0<k<1 \\ +\infty & \text { if } k=1 \\ \operatorname{\text{arccot}h} k & \text { if } k>1\end{cases}
 \end{equation}

 When $k\in[1,+\infty)$, denote by $\alpha$ the cone/central angle of an arc $l$ on a circle or horocycle of geodesic curvature $k$. When $k\in(0,1)$, given an arc $l$ in a hypercycle $C$ of geodesic curvature $k$, let $\gamma_1$ and $\gamma_2$ be the geodesics that pass through the end points of $l$ and are perpendicular to the axis $\gamma $ of $C$, and we define $\beta$ as the length of the arc on $\gamma $ between $\gamma_1$ and $\gamma_2$. We call $\alpha $ or $\beta$ the \textbf{\em inner angle} of $l$, which is illustrated on Figure \ref{Fig2}.  
\begin{figure}[ht]
{
\centering
\includegraphics[height=4cm]{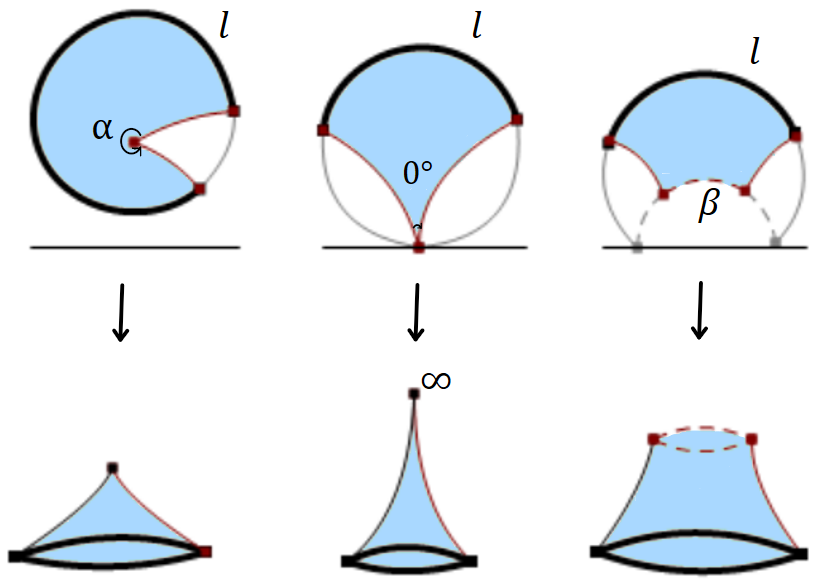} 
\caption{Generalized hyperbolic conical circles.}
\label{Fig2}
}
\end{figure}

\begin{definition}[Generalized circle]\label{generalized conical circle}
The black arc $l$ on the left (resp. middle) drawing of Figure \ref{Fig2} is called a conical hyperbolic circle (resp. conical horocycle) by viewing it as the boundary of a cone obtained by gluing the two radial sides of the shaded region, where $k\in[1,+\infty)$. The black arc $l$ on the right drawing of Figure \ref{Fig2} is called a hypercircle by viewing it as one of the boundaries of the surface obtained by gluing the two sides perpendicular to the axis $\gamma$, where $k\in (0, 1)$. In this paper, we simply call $l$ in any of the three cases a \textbf{generalized circle}. 
\end{definition}

Given a generalized circle $l$, Table \ref{tab-1} shows the relationships among radius $r$, geodesic curvature $k$, generalized inner angle $\alpha$ (or $\beta$) and arc length $\ell$ of $l$. The total geodesic curvature $T$ of $l$ is given by $T=k\ell$, where $k, \ell\in (0, +\infty )$.

\begin{table}[h!]
 \centering
 
 \label{tab-1}
\begin{tabular}{|c|c|c|c|}\hline
 & Circle & Horocycle & hypercycle \\ \hline
Radius & $r\in(0,+\infty)$ & $r=+\infty$ & $r\in(0,+\infty)$\\ \hline
Geodesic curvature& $k=\coth r$  & k=1 & $k=\tanh r$  \\ \hline
Inner angle & $\alpha\in(0,+\infty)$  & $\alpha=0$ & $\beta\in(0,+\infty)$ \\ \hline
Arc length& $\ell=\alpha\sinh r$ & $\ell$ & $\ell=\beta\cosh r$  \\ \hline
Total geodesic curvature& $T=\alpha\cosh r$ & $T=\ell$ & $T=\beta\sinh r$ \\ \hline
\end{tabular}
\caption{Relationships among the measures of a generalized hyperbolic conical circle.}
\end{table}

\begin{figure}[ht]
{
\centering
\includegraphics[height=4cm]{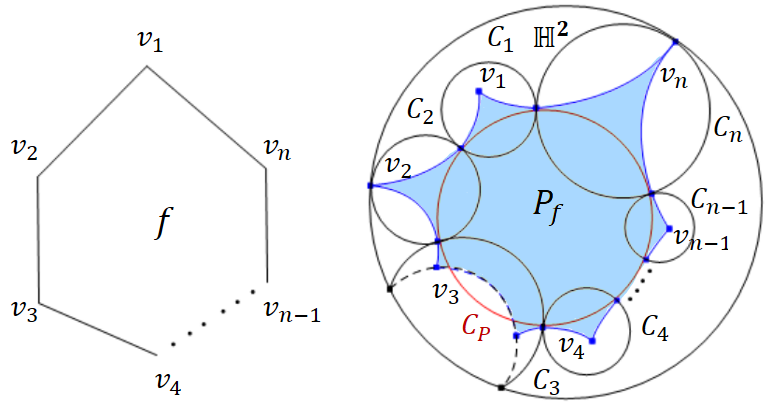} 
\caption{Generalized circle configuration and induced hyperbolic polygon.}
\label{Fig3}
}
\end{figure}
\begin{definition}[Generalized circle configuration]\label{def-2.3}
Let $n \geq 3$ be an integer. A collection $\mathcal{P}^0 = \{C_1^0, C_2^0, \ldots, C_n^0\}$ of $n$ generalized circles on $\mathbb{H}^2$ is called a generalized circle configuration if it satisfies the following conditions:

\begin{itemize}
\item [(H1)] Tangency: For each $i \in \mathbb{Z}/n\mathbb{Z}$, the circles $C_i^0$ and $C_{i+1}^0$ are externally tangent at a point $T_i^0$, and $C_i^0 \cap C_{i+1}^0 = \{T_i^0\}$.

\item [(H2)] Simplicity: The configuration $\mathcal{P}^0$ is simple, i.e., $C_i^0 \cap C_j^0 = \emptyset$ for $|i-j| > 1$ (modulo $n$), and no three circles meet at a common point.

\item [(H3)] Convexity: The configuration $\mathcal{P}^0$ bounds a compact  convex region $P_f^0 \subset \mathbb{H}^2$ with non-empty interior.

\item [(H4)] There exists a (\textbf{regular}) hyperbolic circle $C_P$ that contains the tangent points between any two of $C_i$'s and is perpendicular to all $C_i$'s. The hyperbolic circle $C_P$ is called the dual circle of $\{C_i\}_{i=1}^{n}$. 

\end{itemize}
\end{definition}
We emphasize that $C_P$ in the previous definition is a regular hyperbolic circle, for which there is no singularity at the center. Figure \ref{Fig3} shows a generalized circle configuration. 

\begin{definition}[Induced hyperbolic polygon by a generalized circle configuration]\label{def-2.4}
Let $\{C_i\}_{i=1}^{n}$ be a generalized circle configuration. The blue region on Figure \ref{Fig3} is a hyperbolic polygon induced by this configuration, which we denote by $P_f$. More precisely, the vertices and edges of $P_f$ are obtained as follows.

\begin{itemize}
    \item [(1)] Given $i \in\{1,\cdots,n\}$, if $C_i$ is a conical hyperbolic circle or horocycle, the center of $C_i$ is a vertex of $P_f$; if $C_i$ is a hypercycle, both endpoints of the axis of $C_i$ are vertices of $P_f$ and both are viewed as the centers of this hypercycle.
    \item [(2)] The edges of $P_f$ are comprised of the hyperbolic geodesics connecting the centers of any two tangent circles and the hyperbolic geodesics connecting both centers of each hypercycle. 
\end{itemize} 
\end{definition}

\begin{lemma}[\cite{MR4842734}]\label{lem2.5}
Given a vector $\textbf{k}=\left(k_1, \cdots, k_n\right) \in \left(\mathbb{R}_{+}\right)^n$, there exists a unique generalized circle configuration (up to isometry under hyperbolic metric) such that the geodesic curvature of $C_i$ is equal to $k_i$ for $i=1,\cdots,n$.
\end{lemma}
Therefore, we denote by $\textbf{k}=\left(k_1,\cdots, k_n\right)$ the generalized circle configuration $\{C_i\}_{i=1}^{n}$ such that the geodesic curvature of $C_i$ is $k_i$ for each $i=1,\cdots,n$.

\subsection{Generalized circle packing associated with polygonal decomposition of surface}
Let $S_{g,n}$ be a compact oriented topological surface of genus $g$ having $n$ holes. 

\begin{definition}[Finite polygonal cellular decomposition]\label{def of poly decomposition} Let $V$, $E$ and $F$ be finite collections of 0-cells, 1-cells and 2-cells on $S_{g,n}$ respectively.
We call $\mathcal{D}=\{V,E,F\}$ a \textbf{finite polygonal cellular decomposition} of $S_{g,n}$ if the following conditions are satisfied:
\begin{enumerate}
\item[(I)] Each 0-cell of $V$ is an endpoint of at least three 1-cells.
\item[(II)]  Any two 1-cells of $E$ do not intersect, except possibly at endpoints.
\item[(III)]  The boundary of each 2-cell of $F$ is composed of at least three 1-cells of $E$; whenever two 2-cells of $F$ intersect, the intersection is a 1-cell in $E$; the union of all 2-cells is equal to $S_{g,n}$.
\item[(IV)] Every boundary of $S_{g,n}$ contains at least three 0-cells of $V$.
\end{enumerate}
\end{definition}

Note that Condition (IV) is required for any finite polygonal cellular decompositions considered in this paper. It is clear that when  $\mathcal{D}=\{V,E,F\}$ is a finite polygonal cellular decomposition of $S_{g,n}$, the union of all 1-cells of $E$, called the 1-skeleton of $\mathcal{D}$ and denoted by $G_{\mathcal{D}}$, is a simple graph on  $S_{g,n}$, which means that the graph does not contain a loop or two edges between any two vertices. 

\begin{definition}[Generalized hyperbolic conical circle packing]\label{def-circle packing}
With the same notation as introduced in Definition \ref{def of poly decomposition}, by a generalized hyperbolic conical circle packing associated with $\mathcal{D}$ we mean a collection $\{C_v\}_{v\in V}$ of generalized hyperbolic conical circles such that the circles centered at the vertices of each face form a generalized circle configuration. 
\end{definition}

Given a geodesic curvature vector $\textbf{k}=(k_1, k_2, \cdots, k_{|V|})\in (\mathbb{R}_{+})^{|V|}$, one may construct a generalized circle packing 
through the following two steps. 
\begin{enumerate}
\item For each face  $f\in F$ with $m$ vertices $i_{1},i_{2},\cdots, i_{m}\in V$, we apply Lemma \ref{lem2.5} to obtain a generalized circle configuration with geodesic curvatures of circles equal to $k_{i_1}, k_{i_2}, \cdots, \text{ and }k_{i_m}$ respectively. We denote the corresponding hyperbolic polygon by $P_f$.
\item For each vertex $v\in V$, we take the angle at $v$ to be the sum of the measures of the angles of those $P_f$'s with $v$ as a vertex. 
\end{enumerate}

    

\begin{figure}[ht]
{
\centering
\includegraphics[height=6cm]{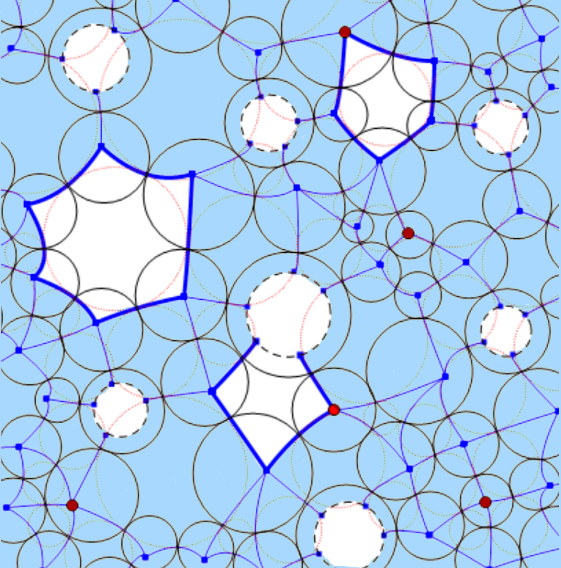} 
\caption{A portion of a generalized hyperbolic conical circle packing on a surface with boundary.}
\label{Fig4}

}
\end{figure}

Figure \ref{Fig4} shows a portion of a generalized hyperbolic conical circle packing on a surface with boundary. Pairs of concentric circles (one solid circle and one dotted circle) represent hypercycles. The red points on the figure represent the centers of horocycles and hence the circles with red points inside represent horocycles. The blue points represent the centers of hyperbolic circles. By gluing the hyperbolic polygons along their adjacent edges, a hyperbolic structure associated with the graph $G_{\mathcal{D}}$ is introduced on $S_{g, n}$, which is uniquely determined by $\textbf{k}$ up to isometry. There are possible singularities at some vertices. 

\begin{definition}[Generalized hyperbolic conical circle packing metric]\label{def-2.5}
Let $\{C_v\}_{v\in V}$ be a generalized hyperbolic conical circle packing associated with $\mathcal{D}$. One can construct a hyperbolic surface $\tilde{S}=\tilde{S}(\textbf{k})$, with possible singularities at vertices, by gluing the hyperbolic polygons of the generalized circle configurations associated with the faces along their adjacent edges. The corresponding metric is called a generalized hyperbolic conical circle packing metric. 
\end{definition}

\subsection{Generalized star}
\begin{figure}[ht]
{
\centering
\includegraphics[height=4cm]{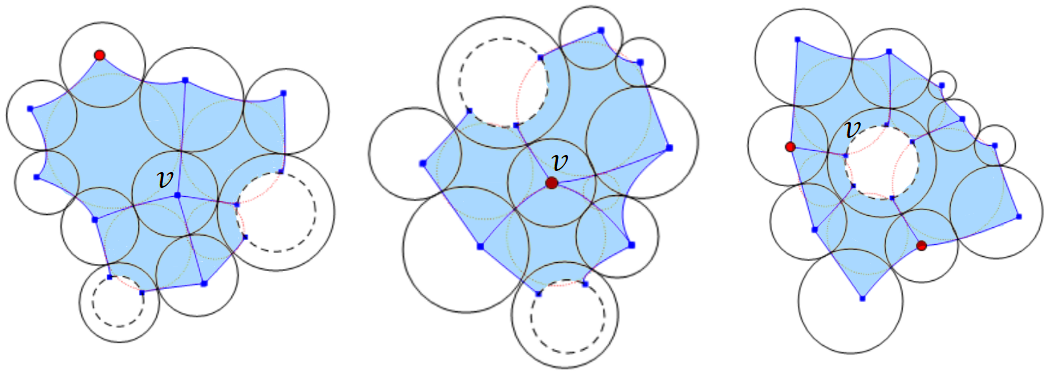} 
\caption{The generalized star at a vertex v.}
\label{Fig5}

}
\end{figure}
Given an interior vertex $v$, there is an ordered chain of faces $\{f_1,\cdots,f_n\}$ surrounding it. Let $f_{n+1}=f_1$. All faces in the chain have $v$ as a vertex and $f_i$ is contiguous to $f_{i+1}$ for each $1\le i\le n$. Associated with each chain of faces, there is a corresponding chain of hyperbolic polygons $\{P_1,\cdots,P_n\}$ for a generalized hyperbolic conical circle packing underlying the graph $G$. The union of the hyperbolic polygons of the chain around a vertex $v$ is called the generalized star at the vertex $v$. The vertex $v$ is possibly the center of a hyperbolic circle, horocycle or hypercycle. Keep in mind that in the case of a hypercycle, we view its axis as its center. Three different examples of generalized stars are shown in Figure \ref{Fig5}.

\section{Local variation formulas and monotonicity}

In this section, we develop several lemmas which are applied to prove our theorems. These lemmas show how the hyperbolic structures are affected when the geodesic curvatures are changed. 

\begin{figure}[ht]
{
\centering
\includegraphics[height=2cm]{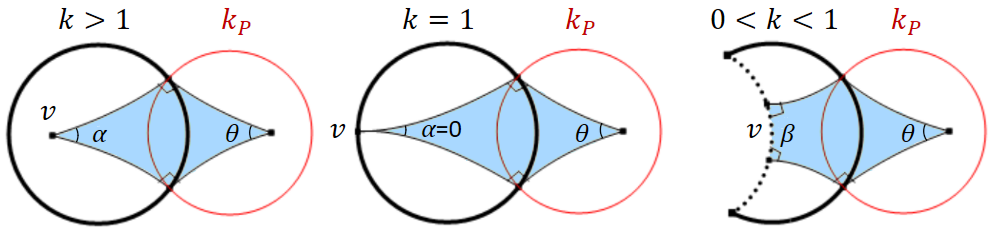} 
\caption{The quadrilateral.}
\label{Fig6}
}
\end{figure}

\subsection{Dual circle and partial derivatives}
\begin{lemma}\label{lem-3.1}
Let $k$ and  $k_P$ be the geodesic curvatures of two generalized circles $C$ and $C_P$ that intersect perpendicularly, $k\in (0,+\infty)$ and $k_P\in(1,+\infty)$. Connecting the centers of the generalized circles with the intersection points by geodesics, we obtain a generalized quadrilateral (as shown in Figure \ref{Fig6}). Denote by $\theta$ the angle of the intersection part viewed from the center of $C_P$, by $\alpha$ the angle of the intersection part viewed from the center of $C$ when $C$ is a circle or horocycle, and by $\beta$ the portion on the axis of $C$ as shown in Figure \ref{Fig6} when $C$ is a hypercycle.

Then $\theta(k,k_P)$ is a differentiable function of $k\in(0,+\infty)$ and $k_P\in(1,+\infty)$; $\alpha(k,k_P)$ is a differentiable function of $k\in(1,+\infty)$ and $k_P\in(1,+\infty)$; $\beta(k,k_P)$ is a differentiable function of $k\in(0,1)$ and $k_P\in(1,+\infty)$. Moreover,
\begin{equation}\label{equa-5}
\begin{aligned}
(a) \;\;\;\;\;\;\;\; & \frac{\partial \theta}{\partial k_P} =\frac{2 k_P k}{\sqrt{k_P^2-1}\left(k^2+k_P^2-1\right)}>0, \\
(b) \;\;\;\;\;\;\;\; &\frac{\partial \theta}{\partial k} 
=-\frac{2 \sqrt{k_P^2-1}}{k^2+k_P^2-1}<0,\\
(c) \;\;\;\;\;\;\;\; &\frac{\partial \alpha}{\partial k_P} =
-\frac{2 \sqrt{k^2-1}}{k_P^2+k^2-1}< 0 \text{ when }k\in(1,+\infty),\\
(d) \;\;\;\;\;\;\;\; &\frac{\partial \beta}{\partial k_P} =
-\frac{2 \sqrt{1-k^2}}{k_P^2+k^2-1}< 0 \text{ when }k\in (0,1).
\end{aligned}
\end{equation}

\end{lemma}

\begin{proof}
    This lemma is deduced from the four-part formula on a quadrilateral. The first two expressions are obtained in Lemma 2.1 of \cite{MR4842734}. Since the partial differences are continuous, the function $\theta(k,k_P)$ is differentiable. Here, we focus on proving the last two expressions.

\noindent\textbf{Case one:} $k\in(1,+\infty)$. Then $k=\coth r$.

By the law of the hyperbolic triangle, we obtain
\begin{equation}\label{equa-6}
\cot \frac{\alpha}{2}=\coth r_P \sinh r=\frac{k_P}{\sqrt{k^2-1}}
\end{equation}
Therefore,
$$
\alpha(k,k_P)=2\text{arccot}\frac{k_P}{\sqrt{k^2-1}}, \text{ where } k\in(1,+\infty).
$$
Differentiating the left and right sides of the equation (\ref{equa-6}) with respect to $k_P$, we obtain
$$
 -\frac{1}{2 \sin ^2 \frac{\alpha}{2}} \frac{\partial \alpha}{\partial k_P}=\frac{1}{\sqrt{k^2-1}}.
$$
Thus,
$$
\frac{\partial\alpha}{\partial k_P}=-\frac{2 \sqrt{k^2-1}}{k_P^2+k^2-1}<0,\text{ where } k\in (1,+\infty).
$$

\noindent\textbf{Case two:} $k\in(0,1)$. Then $k=\tanh r$.

By the formula of a hyperbolic pentagon,
\begin{equation}\label{equa-7}
\coth \frac{\beta}{2}=\coth{r_P}\cdot \cosh r=\frac{k_P}{\sqrt{1-k^2}}.
\end{equation}
Therefore,
$$
\beta(k,k_P)=2\text{arccot}h\frac{k_P}{\sqrt{1-k^2}},\text{ where } k\in(0,1).
$$
Differentiating the left and right sides of the equation (\ref{equa-7}) with respect to $k_P$, we obtain
$$
 -\frac{1}{2 \sinh ^2 \frac{\beta}{2}} \frac{\partial \beta}{\partial k_P}=\frac{1}{\sqrt{1-k^2}}.
$$
Thus,
$$
\frac{\partial \beta}{\partial k_P} =
-\frac{2\sqrt{1-k^2}}{k_P^2+k^2-1}<0,\text{ where } k\in(0,1).
$$
\end{proof}

\subsection{Facewise monotonicity of interstice curvature and area}

\begin{lemma}[\cite{MR4842734}]\label{lem-3.2} Let $\left\{C_i\right\}_{i=1}^n$ be a generalized circle configuration and denote by $k_i$ the geodesic curvature of $C_i$, where $k_i \in \mathbb{R}_{+}, i=1, \cdots, n$. Denote by $k_P$ ($>1$) the geodesic curvature of the circle $C_P$ perpendicular to all $C_i$. Then $k_P$ is a $C^1$ continuous function of $\left(k_1, \cdots, k_n\right)$, and $k_P$ is strictly increasing on $k_i \text{ for each }i=1,\cdots,n$.
\end{lemma}
\begin{proof} Denote by $C_P^i$ the part of $C_P$ that lies in the circle $C_i$ and by $\theta_i$ the angle of $C_P^i$ viewed from the center of $C_P$, $i=1, 2, \cdots, n$ (see Figure \ref{Fig7}). Then 
\begin{equation}\label{sum of angles viewed from the center of C_P}
\sum_{i=1}^n \theta_i=2\pi.
    \end{equation}
By Lemma \ref{lem-3.1}, $\theta_i$ is a smooth function of $k_i$ and $k_P$ for each $1\le i\le n$. 
Applying the implicit differentiation and Lemma \ref{lem-3.1} (a) and (b), we obtain 
\begin{equation}\label{equa-10}
\frac{\partial k_P}{\partial k_i}=-\frac{\frac{\partial \theta_i}{\partial k_i}}{\sum_{j=1}^n\frac{\partial \theta_j}{\partial k_P}}=\frac{ \frac{\sqrt{k_P^2-1}}{k_P^2+k_i^2-1}}{\sum_{j=1}^n \frac{k_P k_j}{\left(k_P^2+k_j^2-1\right) \sqrt{k_P^2-1}}}>0.
\end{equation} 
 Thus, $k_P$ increases strictly on $k_i \text{ for each }i=1,\cdots,n$.
\end{proof}
\begin{figure}[ht]
{
\centering
\includegraphics[height=4cm]{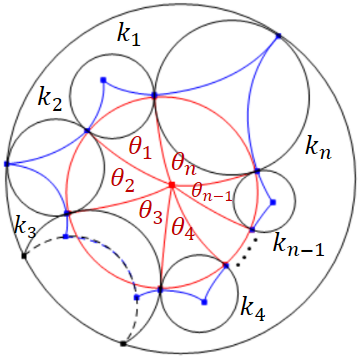} 
\caption{Dual circle $k_P$ and the angles $\theta_i$, $i=1, 2, \cdots, n$, of the edges of $\Omega_f$ viewed from the center of $k_P$.}
\label{Fig7}

}
\end{figure}
\begin{remark}
Using the equation (\ref{sum of angles viewed from the center of C_P}) and the expression of $\theta_i$ in terms of $k_i$ and $k_P$ for each $1\le i\le n$, one can derive that $k_P$ is a positive real root of the following equation: 
\begin{equation}\label{equ-6}
\sum_{m=0}^{[\frac{n-1}2{}]}(-1)^mS_{n-1-2m}(k_P^{2}-1)^{m}=0,\quad n\geq 3,
\end{equation}
where 
$$S_0=1 \text{ and }S_{t}=\sum_{1\leq i_1<\cdots <i_{t}\leq n} k_{i_1}\cdots k_{i_{t}} \text{ for } t=1,\cdots,n.$$
 We show how the equation (\ref{equ-6}) is obtained in Appendix \ref{sec-7.1}.
\end{remark}

Let $\mathcal{P}$ be a generalized hyperbolic conical circle packing on the surface $\mathcal{S}$ with boundary. Then $\mathcal{P}$ induces a finite polygonal decomposition $\mathcal{D}$ on $\mathcal{S}$, which induces a simple graph $G_{\mathcal{D}}=\{V_{\mathcal{D}},E_{\mathcal{D}},F_{\mathcal{D}}\}$ on $\mathcal{S}$, where $V_{\mathcal{D}},E_{\mathcal{D}},F_{\mathcal{D}}$ represent the vertices, edges and faces of the graph, respectively.

\begin{figure}[ht]
{
\centering
\includegraphics[height=4cm]{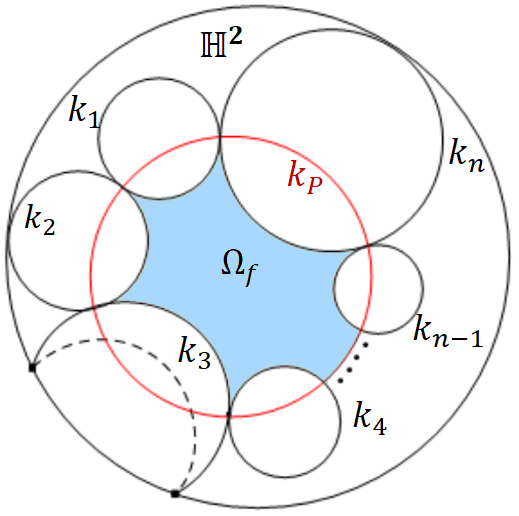} 
\caption{The interstice $\Omega_f$ corresponding to the face $f$.}
\label{Fig8}

}
\end{figure}

\begin{lemma}[Partial derivatives, \cite{MR4825314}]\label{lem-3.3-1}
Let $\{C_1, \cdots, C_n\}$ be a generalized circle configuration with geodesic curvatures equal to $k_1, k_2,\cdots, k_n \in \mathbb{R}_{+}$ respectively, and let $\Omega_f$ be the interstice bounded by this circle configuration (which is a hyperbolic polygon as shown in Figure \ref{Fig8}). Denote by $T_{i, P}$ the total geodesic curvature of the edge of $\Omega_f$ lying on $C_i$, where $1\le i\le n$. Then 
\begin{equation}\label{partial derivatives}
\begin{aligned}
(a) \;\;\;\;\;\;\;\; &
\frac{\partial\left(\sum_{j=1}^n T_{j, P}\right)}{\partial k_i}>0, \quad 1 \leq i \leq n.\\
(b) \;\;\;\;\;\;\;\; & \frac{\partial T_{j, P}}{\partial k_i}<0, \quad \forall \;i \neq j. \\
(c) \;\;\;\;\;\;\;\; & \frac{\partial T_{i, P}}{\partial k_i}>0, 
\quad 1 \leq i \leq n.\\
(d) \;\;\;\;\;\;\;\; &
\frac{\partial \operatorname{Area}\left(\Omega_f\right)}{\partial k_i}<0, \quad 1 \leq i \leq n .
\end{aligned}
\end{equation}
\end{lemma}

 For a vertex $v\in V^{\circ}$, denote by $P_i,i=1,\cdots,n$ the induced hyperbolic polygons by the packing $\mathcal{P}$ that take $v$ as a vertex, and denote by $C_{P_i}$ the corresponding dual circles contained in $P_i$ for $i=1,\cdots,n$. All the vertices of all $P_i$'s except $v$ are called the neighboring vertices of $v$. Note that in the case of a hypercycle, all vertices of $P_i$'s on its axis are viewed as the center of that hypercycle or the vertex of the graph corresponding to that hypercycle.  
 
 Figure \ref{Fig9} shows an example for the neighboring vertices of a vertex and the dual circles surrounding this vertex.  
\begin{figure}[ht]
{
\centering
\includegraphics[height=4cm]{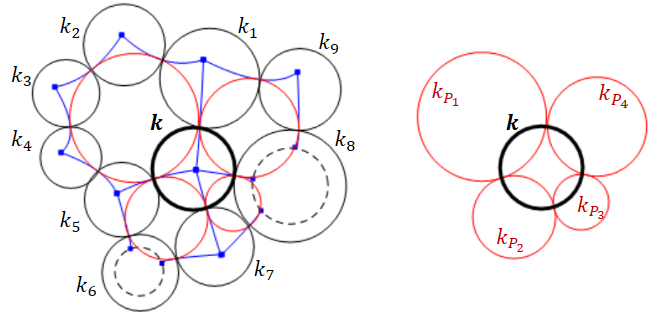} 
\caption{Neighboring vertices and surrounding dual circles for a vertex.}
\label{Fig9}

}
\end{figure}

From the previous lemma, we obtain the following monotonicity lemma.
\begin{lemma}\label{lem-3.3}
Let $v$ be an interior vertex of a generalized circle packing $\mathcal{P}$ with neighboring faces $f_1, \ldots f_n$. 
Let $k$ be the geodesic curvature of the generalized circle centered at $v$ and let $k_j$, $j=1,\cdots,q$, be the geodesic curvatures of the neighboring vertices $v_j$ of $v$, where by a neighboring vertex of $v$ we mean a vertex, except $v$, of a face with $v$ as a vertex. Denote by $\Omega_{f_i}$ the interstice corresponding to the face $f_i$. The total geodesic curvature $T$ of the circle centered at $v$ is a continuous function of $(k,k_1 \cdots k_q)$, which we denote by 
$$
T=T\left(k,k_1 \cdots k_q\right).
$$
Then

(a) $T$ is strictly increasing with $k$;

(b) $T$ is strictly decreasing with each of $k_{j}, j=1 ,\cdots,q$;

(c) $\sum_{i=1}^n Area(\Omega_{f_i})$ is strictly decreasing with each of $k$ and $k_{j},j= 1,\cdots,q$. 
\end{lemma}

\subsection{Edge-angle monotonicity for generalized convex hyperbolic polygons}

\begin{lemma}\label{lem-3.6}
Let $P$  be a convex hyperbolic polygon on the hyperbolic plane with its edges labeled by $e_1, e_2, \cdots, e_n$ in the counterclockwise order satisfying that any two adjacent edges $e_i$ and $e_{i+1}$ meet at a point $v_{i+1}$ at angle $\alpha_{i+1}$ or have a common perpendicular geodesic arc $\beta_{i+1}$ between them (called a \textbf{hyper-edge}), where $i=1, 2, \cdots, n$ and $n+1=1$. We call $v_i$ or $\beta_i$ the vertices of $P$, which we collectively denote by $I_1, I_2, \ldots, I_n$ in the counterclockwise order. We call $\alpha_i$ or $\beta_i$ the  \textbf{generalized interior angles} of $P$, and also let $\alpha_i$ (resp. $\beta_i$) denote its measure. See Figure \ref{Fig24}) for an example. 

Denote the length of $e_i$ by $\ell_i$ and assume that $\alpha_i\in (0, \pi)$ and $\beta_i>0$, where $i=1, 2, \cdots, n$. The length $l_n$ of the edge $e_n$ is uniquely determined by the lengths $l_1, \ldots, l_{n-1}$ of the other edges and the measures of the generalized interior angles $I_2, I_3, \ldots, I_{n-1}$. Consider the deformations of $P$ that preserve the lengths of the edges $e_1, e_2, \cdots, e_{n-1}$ and the nature of the generalized interior angles (that is, each vertex changes to a vertex and each hyper-edge changes a hyper-edge). Then $\ell_n$, regarded as a function of the measures of $I_2, I_3, \ldots, I_{n-1}$, is strictly increasing with respect to the measure of each $I_i$, where $i = 2, \ldots, n-1$.
\end{lemma}

\begin{figure}[ht]
{
\centering
\includegraphics[height=3cm]{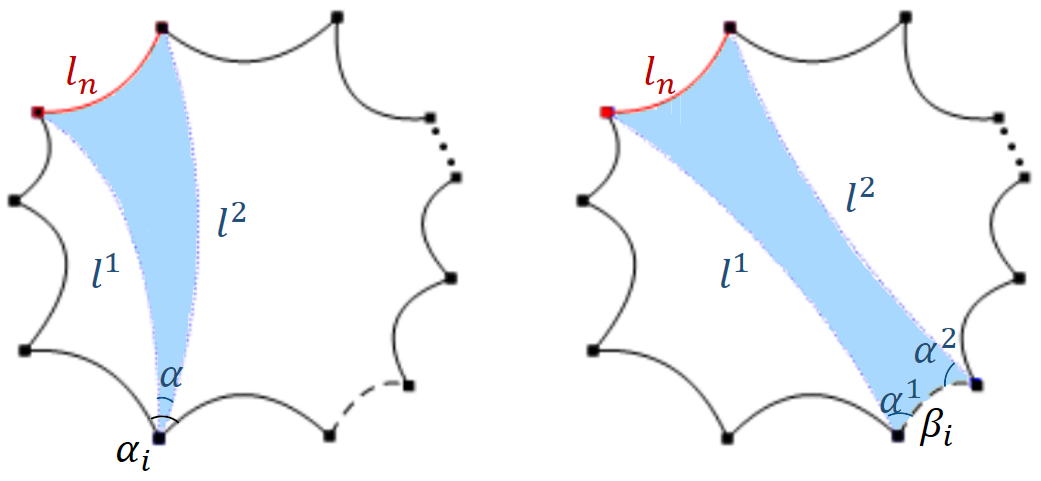} 
\caption{Decompositions of a hyperbolic polygon.}
\label{Figuu10}

}
\end{figure}
\begin{proof}

From the viewpoint of geometry, we see that for the deformations of $P$ considered in this lemma, $\ell_n$ is a function of the measures of the edges $e_1, e_2, \cdots, e_{n-1}$ and the measures of $I_2, I_3, \ldots, I_{n-1}$. The explicit relationship under which $\ell_n$ is determined by $\ell_1, \ell_2, \cdots, \ell_{n-1}$ and the measures of $I_2, I_3, \ldots, I_{n-1}$ is given in the appendix \ref{sec-7.2}.

To prove monotonicity, we draw several auxiliary geodesic lines to decompose the polygon $P$ into a union of triangles or quadrilaterals. By analyzing how the side lengths of each triangle or quadrilateral vary with \( I_i \)'s, we see how the length of \( l_n \) changes with respect to the measure of each \( I_i \).  

There are two cases to consider. 

In one case, we assume that some $I_i=\alpha_i$ is increased and the others $I_j$'s, $j\neq i$ and $2\leq j\leq n-1$, remain the same. We consider the auxiliary triangle shown on the left figure on Figure \ref{Figuu10}. By the cosine law for hyperbolic triangles, 

\begin{equation}\label{equ-10}
\cosh l_n=\cosh l^1\cosh l^2-\sinh l^1\sinh l^2\cos\alpha,
\end{equation}

Differentiating with respect to \(\alpha\), we obtain
\begin{equation}\label{equa-11}
\frac{\partial \ell_n}{\partial\alpha}
        =
    \frac{\sinh l^1\sinh l^2\sin\alpha}{\sinh\ell_n}>0.        
\end{equation}
 
 We know that $l_n$ increases with $\alpha$. Since the other $I_j$'s, $j\neq i$, are unchanged, using Lemma \ref{lem-3.6} in the two polygons on the left and right sides of the auxiliary triangle, the lengths of the edges $l^1$ and $l^2$ remain the same and $\alpha_i-\alpha$ remains the same as well. Thus, $\alpha$ increases with $\alpha_i$. Therefore, $l_n$ increases with $\alpha_i$.

In the other case, we suppose that some $I_i=\beta_i$ is increased and the others $I_j$'s, $j\neq i$ and $2\leq j\leq n-1$, are unchanged. We consider the quadrilateral shown on the right figure on Figure \ref{Figuu10}. Since the other $I_j$'s, $j\neq i$, are not changed, the lengths of $l^1$ and $l^2$ and the angles $\alpha^1$ and $\alpha^2$ are not changed. By the hyperbolic distance formula (\ref{equa-28}), we obtain
\begin{equation}\label{equ-11}
    \begin{aligned}
\cosh l_n= & \cosh \beta_i \cosh l^1 \cosh l^2 \\
& -\sinh \beta_i\left(\sinh l^1 \cosh l^2 \cos \alpha^1+\cosh l^1 \sinh l^2 \cos \alpha^2\right) \\
& +\sinh l^1 \sinh l^2\left(\cosh \beta_i \cos \alpha^1 \cos \alpha^2-\sin \alpha^1 \sin \alpha^2\right)
\end{aligned}
\end{equation}


Differentiating equation (\ref{equ-11}) with respect to \(\beta_i\), we obtain
\begin{equation}
  \begin{aligned}
 \frac{\partial l_n}{\partial\beta_i}
        =&
        \frac{\sinh\beta_i \cosh l^1 \cosh l^2  +\sinh l^1 \sinh l^2\sinh\beta_i \cos \alpha^1 \cos \alpha^2}{\sinh l_n}\\
        &-\frac{ \cosh \beta_i\left(\sinh l^1 \cosh l^2 \cos \alpha^1+\cosh l^1 \sinh l^2 \cos \alpha^2\right) }{\sinh l_n}. 
         \end{aligned} 
\end{equation}

Set $\tanh u_1=\tanh l^1\cos\alpha^1$, $\tanh u_2=\tanh l^2\cos\alpha^2$. Then
\begin{equation}
\frac{\partial l_n}{\partial\beta_i}
    =
    \frac{\cosh l^1\cosh l^2}
         {\sinh l_n\cosh u_1\cosh u_2}
    \sinh\bigl(\beta_i-u_1-u_2\bigr).
\end{equation}

Next, we prove that $\beta_i-u_1-u_2>0$ using the geometric meaning of $u_j,j=1,2$.

\begin{figure}[ht]
{
\centering
\includegraphics[height=5cm]{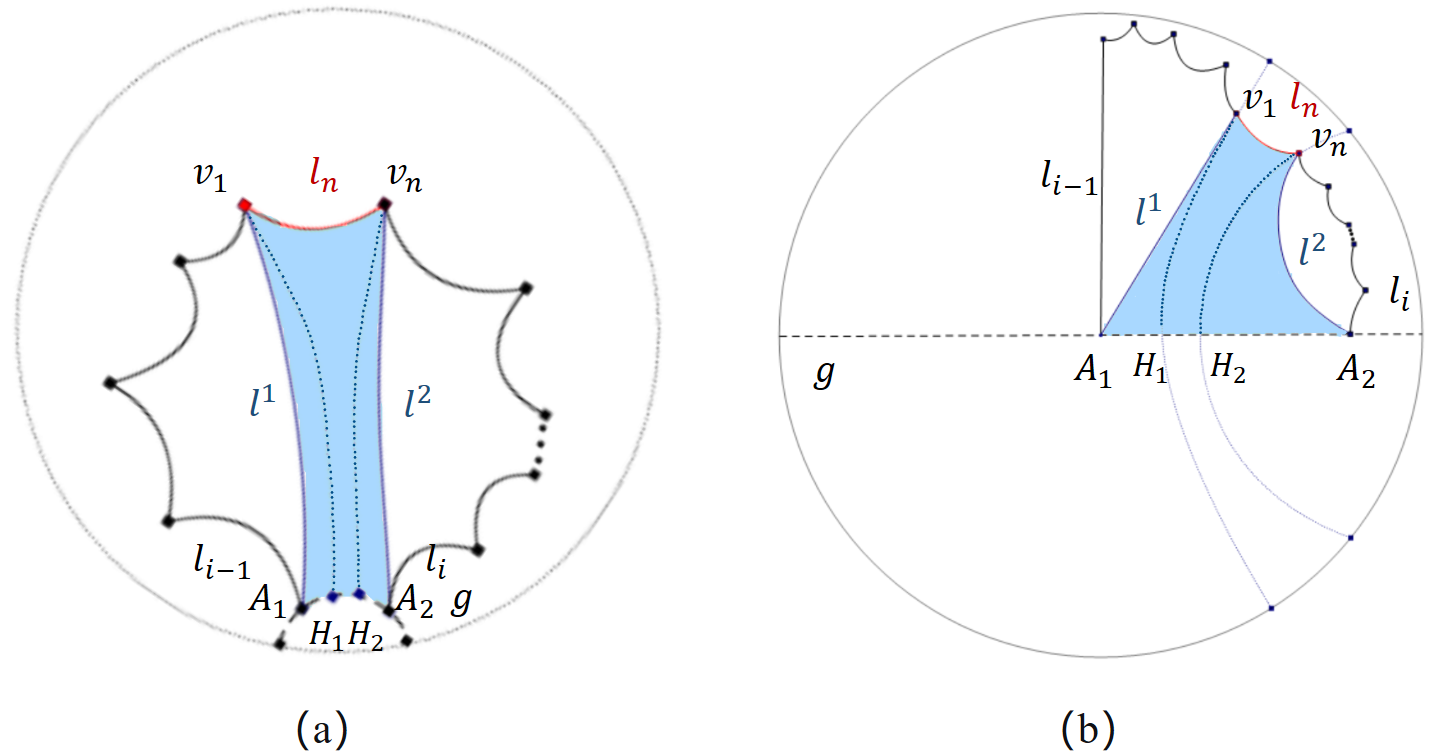} 
\caption{The polygon $P$ with transformation.}
\label{Figu11}

}
\end{figure}

Let \(A_1,A_2\) be the endpoints of the hyper-edge
\(\beta_i\) and \(g\) be the complete geodesic containing \(\beta_i\), oriented
from \(A_1\) to \(A_2\). Denote by \(H_1\) and \(H_2\) the orthogonal
projections of \(v_1\) and \(v_n\) onto \(g\), respectively (see Figure \ref{Figu11} (a)). We denote $\ell(A_j,H_j)$ as the Signed distance from $A_j$ to $H_j$, $j=1,2$. In triangles $\Delta v_1A_1H_1$ and $\Delta v_nA_2H_2$, we have
$$
\tanh \ell(A_1,H_1)=\tanh l^1\cos\alpha^1,\qquad \tanh \ell(A_2,H_2)=\tanh l^2\cos\alpha^2.
$$

Therefore, $u_j=\ell(A_j,H_j),j=1,2$, and $\beta_i-u_1-u_2$ is the signed distance from $H_1$ to $H_2$. By applying a Möbius transformation to the polygon $P$, we map the point \( A_1 \) to the center of the Poincaré disk. Under this mapping, the geodesic \( g \) becomes a straight line in the disk, and the geodesic containing the side \( l_{i-1} \) likewise becomes a straight line which perpendicular to $g$. Furthermore, \( v_1H_1 \) and \( v_1H_2 \) are also two geodesics perpendicular to \( g \). From the convexity of the polygon \( P \), it follows that the four points appear on \( g \) in the order \( A_1, H_1, H_2, A_2 \)(see Figure \ref{Figu11} (b)). Therefore, $\beta_i-u_1-u_2>0$. Consequently, with $l^1, l^2, \alpha^1, \alpha^2$ fixed,
the opposite side length $l_n$ is strictly increasing as a function of the hyper-edge
length $\beta_i$, i.e.
$$
\frac{\partial l_n}{\partial\beta_i}>0.
$$
\end{proof}

\section{Length comparison under deformed generalized circle configurations}

\subsection{Continuity of segment length under deformation}



\begin{lemma}[Continuity of geodesic arc length under deformation]\label{lem-3.7}
  Let $\mathcal{P} = \{C_1, C_2, \dots, C_n\}$ be a generalized circle configuration with circles arranged in the counterclockwise order (see Definition \ref{def-2.3}), and let $P_f$ be the induced hyperbolic polygon. Assume that $\gamma$ is a geodesic on $P_f$ that transversely intersects the boundary of $P_f$ at two points $y_1$ and $y_2$ (see the blue segment on Figure \ref{Fig12} as an example). We assume that $y_i$ lies on the edge $e_{j_i, j_i+1}$ connecting the vertices $v_{j_i}$ and $v_{j_i+1}$ and 
  $t_{j_i}$ is the tangent pint of the circles centered at $v_{j_i}$ and $v_{j_i+1}$, and we further assume that $y_i$ lies on the geodesic arc between 
  $v_{j_i}$ and $t_{j_i}$, 
  where $i=1, 2$. Additionally, $y_i$ does not coincide with the center of the horocycle. Let $\mathcal{P}^{m} = \{C_1^m, C_2^m, \dots, C_n^m\}$, $m=1, 2, \cdots $, be a sequence of generalized circle configurations with the curvature $k_i^m$ of $C_i^m$ converging to the curvature $k_i$ of $C_i$ as $m$ goes to $\infty $ for each $i=1, 2, \cdots, n$, let $P_f^m$ be the induced hyperbolic polygon of $\mathcal{P}^m$ and $y_1^m$ and $y_2^m$ be the deformations of $y_1$ and $y_2$ respectively. That is, $y_i^m$ lies on the same side of $t^m_{j_i}$ as $y_i$ does with respect to $t_{j_i}$, and its distance to $t^m_{j_i}$ equals that from $y_i$ to $t_{j_i}$, where $v^m_{j_i}$ and $t^m_{j_i}$ are defined analogously to $v_{j_i}$ and $t_{j_i}$, where $i = 1, 2$. Let $\gamma_m$ be the geodesic segment connecting $y_1^m$ and $y_2^m$. Then $$\ell(\gamma^m)=d(y_1^m,y_2^m)\rightarrow \ell(\gamma)=d(y_1,y_2) \text{ as }m\rightarrow\infty.$$
\end{lemma}

\begin{figure}[ht]
{
\centering
\includegraphics[height=8cm]{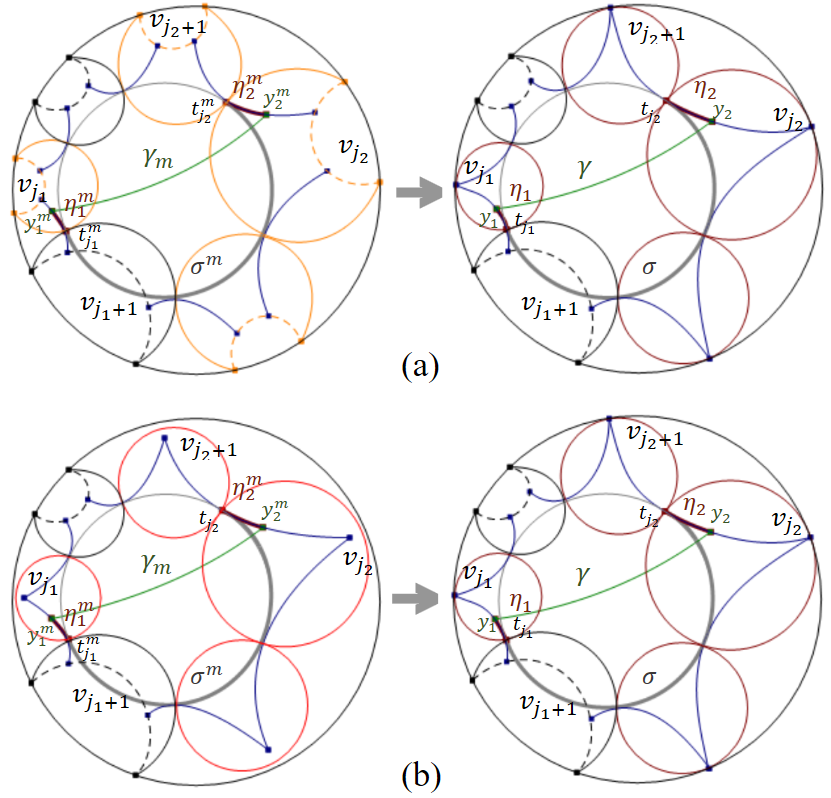} 
\caption{The approximal of horocycle.}
\label{Fig12}

}
\end{figure}

\begin{proof}

Without loss of generality, we assume that $y_i$ (resp. $y^m_i$) lies on the geodesic arc between $v_{j_i}$ and $t_{j_i}$ (resp. ${v}_{j_i}$ and ${t}^m_{j_i}$), where $i=1, 2$ (see Figure \ref{Fig12}). For the sake of using uniform notation to define more items, we let $\mathcal{P}^0=\mathcal{P}$.

For each configuration $\mathcal{P}^m$, we construct a continuous path $\Gamma^m$ connecting $y_1^m$ and $y_2^m$, where $m=0, 1, 2, \cdots $.

Let $\eta_i^m \subset e^m_{j_i, j_i+1}\subset \partial P_f^m$ be the boundary segment from $y_i^m$ to the tangent point $t_{j_i}^m$, where  $i=1, 2$, and let $\sigma^m$ be the arc on the dual circle between the two tangent points $t_{j_1}^m$ and $t_{j_2}^m$ (see Figure \ref{Fig12}). Define $\Gamma^m=\eta_1^m\cup\sigma^m \cup\eta_2^m$.

Upon post-compositions by M\"obius transformations, we may assume that  the dual circles of $\mathcal{P}^m$, $m=0, 1, 2, \cdots$, are arranged at the same position on the hyperbolic plane $\mathbb{H}$ and the tangent points $t_{j_1}^m$, $m=0, 1, 2, \cdots $, are arranged at the same point. 

Each component of $\Gamma^m$ is a compact subset of $\mathbb{H}^2$, where $m=0, 1, 2, \cdots$. Since $\mathbf{k}^m = (k_1^m, k_2^m,\cdots, k_n^m)$ approaches $\mathbf{k} = (k_1, k_2,\cdots, k_n)$ as $m\rightarrow \infty$, from Lemma \ref{lem-3.2} it follows that 
$\sigma^m$ (resp. $\eta_1^m$, $\eta_2^m$) converges to $\sigma^0$ (resp. $\eta_1^0$, $\eta_2^0$) in the Hausdorff topology on $\mathbb{H}^2$.
Thus, $\Gamma^m \to \Gamma$ converges to $\Gamma^0$ in the Hausdorff topology on $\mathbb{H}^2$ as $m\rightarrow \infty$. Then the endpoints $y_1^m$ and $y_2^m$ of $\Gamma^m$ converge to the endpoints 
$y^0_1=y_1$ and $y^0_2=y_2$ of $\Gamma^0$ as $m\rightarrow \infty$. 
Therefore, the hyperbolic distance between $y_1^m$ and $y_2^m$ converges to the distance between $y_1$ and $y_2$ as $m\rightarrow \infty$.

\end{proof}


\subsection{A uniform derivative estimate near horocycles}

\begin{lemma}[Uniform negative derivative estimate]
\label{lem:uniform-negative-derivative}
Let
\[
\mathcal P_{\kappa}=\{C_1,C_2,\ldots,C_n\},
\qquad
\kappa\in K=[k_1,\widetilde k_{1}]\Subset (0,+\infty),
\]
be a $C^1$-family of generalized circle configurations satisfying
$(H1)$--$(H4)$ in Definition~\ref{def-2.3}.  We assume that the
geodesic curvature of $C_1$ is $\kappa$, while the geodesic curvatures of
$C_2,\ldots,C_n$ are independent of $\kappa$.

Let $\gamma_{\kappa}$ be a geodesic segment that meets the hypotheses of Lemma~\ref{lem-3.7}. Thus, $\gamma_{\kappa}$ intersects the induced hyperbolic polygon of $\mathcal{P}_{\kappa}$ transversely at two
points, and these intersection points do not coincide with the ideal
centers of horocycles.  Let $P^*_{\kappa}$ and $P_{*,\kappa}$ be the two
subpolygons cut out by $\gamma_{\kappa}$, where $P^*_{\kappa}$ contains the
vertex corresponding to $C_1$ and $P_{*,\kappa}$ does not.

After a cyclic relabeling, write the part of
$\partial P_{*,\kappa}\setminus\gamma_{\kappa}$ as lying on the generalized
circles
\[
C_{j_1+1},C_{j_1+2},\ldots,C_{j_2},
\]
where the indices are understood as modulo $n$.  Let
\[
Q=\{q_1,\ldots,q_s\}\subset\{j_1+1,\ldots,j_2\}
\]
be the set of indices for which $C_q$ is a horocycle, or equivalently
$k_q=1$.

For each $m\ge1$, let
\[
\mathcal P^m_{\kappa}=\{C^m_1,C^m_2,\ldots,C^m_n\}
\]
be a generalized circle configuration such that
\[
{k_1}^m=\kappa,
\qquad
k^m_j=k_j \quad \text{for } j\notin Q,\ j\neq 1,
\]
and, for every $q\in Q$,
\[
k^m_q\in (1,+\infty),
\qquad
k^m_q\to 1
\quad \text{as } m\to\infty .
\]
In particular, $\mathcal P^m_{*,\kappa}$ has no horocycle on
$\partial P^m_{*,\kappa}\setminus\gamma^m_{\kappa}$, and
\[
\mathcal P^m_{*,\kappa}\longrightarrow \mathcal P_{{\kappa}_1}
\]
uniformly for $k_1\in K$.

Let $\gamma^m_{\kappa}$ be the deformation of $\gamma_{\kappa}$ in
$\mathcal{P}_{\kappa}^m$ in the sense of Lemma~\ref{lem-3.7}, and set
\[
\ell_m({\kappa}):=\ell(\gamma^m_{\kappa}).
\]
Assume that the limiting family is uniformly non-degenerate on $K$; that is,
the endpoints of $\gamma^m_{\kappa}$ remain in a compact subset of
$\mathbb H^2$ after normalization, do not converge to the ideal centers of
horocycles, and the auxiliary triangles and quadrilaterals used in the
decomposition of $P^m_{*,\kappa}$ have non-degenerate limits.

Then there exist constants $m_0\in\mathbb N$ and $c>0$, independent of
$m$ and $\kappa$, such that
\[
\frac{d\ell_m}{d\kappa}(\kappa)\le -c<0,
\qquad
m\ge m_0,\quad \kappa\in K .
\]
If $1\in K$, the derivative at $\kappa=1$ is understood in the corresponding
one-sided sense; equivalently, the estimate holds on
$K\cap(0,1)$ and on $K\cap(1,+\infty)$, and extends to $\kappa=1$ by the
one-sided limits.
\end{lemma}

\begin{proof}
For simplicity, write
\[
k_P^m=k_P^m(\kappa)
\]
for the geodesic curvature of the dual circle of the configuration
$\mathcal P^m_{\kappa}$.  Let
\[
\mathcal J=\{j_1+1,j_1+2,\cdots,j_2\}
\]
be the set of indices corresponding to the generalized circles on
\(\partial P^m_{*,\kappa}\setminus \gamma^m_{\kappa}\), with the indices understood
modulo \(n\).  For \(j\in\mathcal J\), denote by
\[
I_j^m=I_j^m(\kappa)
\]
the corresponding generalized interior angle of \(P^m_{*,k_1}\).  Thus
\(I_j^m=\alpha_j^m\) if \(C_j^m\) is a hyperbolic circle, and
\(I_j^m=\beta_j^m\) if \(C_j^m\) is a hypercicle. Since the subpolygon \(P^m_{*,k_1}\) does not contain the vertex corresponding
to \(C_1\), the parameter \(\kappa\) enters the generalized interior angles
\(I_j^m\) only through the dual curvature \(k_P^m\).  Hence the chain rule gives

\begin{equation}
\label{eq:chain-rule-63}
\frac{d\ell_{m}}{d\kappa}
=
\sum_{j\in\mathcal J}
\frac{\partial \ell_{m}}{\partial I_{j}^{m}}
\frac{\partial I_{j}^{m}}{\partial k_{P}^m}
\frac{\partial k_{P}^m}{\partial \kappa}.
\end{equation}
By Lemma \ref{lem-3.6},
\begin{equation}
\label{eq:length-angle-positive-63}
\frac{\partial \ell_{m}}{\partial I_{j}^{m}}>0.
\end{equation}
By Lemma \ref{lem-3.2},
\[
\frac{\partial k_{P}^m}{\partial \kappa}>0.
\]
More explicitly, by equation (\ref{equa-10}), we have
\begin{equation}
\label{eq:dual-derivative-63}
\frac{\partial k_{P}}{\partial \kappa}
=
\frac{(k_{P}^m)^{2}-1}{(k_{P}^m)^{2}+\kappa^{2}-1}
\left(
\sum_{j=1}^{n}
\frac{k_{P}^m k_{j}^{m}}{(k_{P}^m)^{2}+(k_{j}^{m})^{2}-1}
\right)^{-1}.
\end{equation}
Since $\kappa\in K\Subset(0,+\infty)$ and the limiting family is
non-degenerate, the dual curvatures $k_{P}(\kappa)$ remain in a compact
subinterval of $(1,+\infty)$.  Hence, there exists a constant $a_{0}>0$ such
that
\begin{equation}
\label{eq:dual-derivative-lower-bound-63}
\frac{\partial k_{P}^m}{\partial \kappa}\ge a_{0}>0
\end{equation}
for all sufficiently large $m$ and all $\kappa\in K$.

We now analyze the remaining product
\[
\frac{\partial \ell_{m}}{\partial I_{j}^{m}}
\frac{\partial I_{j}^{m}}{\partial k_{P}^m} .
\]

First let \(j\in\mathcal J\setminus Q\).  Then \(k_j^m=k_j\neq 1\).  If
\(k_j>1\), then \(I_j^m=\alpha_j^m\), and Lemma~\ref{lem-3.1} gives
\begin{equation}
\label{eq:alpha-derivative-away}
\frac{\partial \alpha_j^m}{\partial k_P^m}
=
-\frac{2\sqrt{k_j^2-1}}{(k_P^m)^2+k_j^2-1}<0 .
\end{equation}
If \(0<k_j<1\), then \(I_j^m=\beta_j^m\), and Lemma~\ref{lem-3.1} gives
\begin{equation}
\label{eq:beta-derivative-away}
\frac{\partial \beta_j^m}{\partial k_P^m}
=
-\frac{2\sqrt{1-k_j^2}}{(k_P^m)^2+k_j^2-1}<0 .
\end{equation}
Since \(k_P^m\) ranges in a compact subinterval of \((1,+\infty)\), and since
the family is uniformly non-degenerate, the quantities
\[
\frac{\partial \ell_m}{\partial I_j^m}
\frac{\partial I_j^m}{\partial k_P^m}
\]
are uniformly bounded above by a negative constant for every fixed
\(j\in\mathcal J\setminus Q\).

It remains to consider the indices \(q\in Q\), for which \(C_q\) is a
horocycle in the limiting configuration.  In the present lemma we approximate
these horocycles only by hyperbolic circles.  Thus
\[
k_q^m>1,
\qquad
k_q^m\to1^+.
\]

The important point is that the angle appearing in the chain rule is the whole
angle of the piece of $C_q^m$ contained in
$\partial P^m_{*,\kappa}\setminus\gamma^m_{\kappa}$. We denote this angle by
\[
I_q^m=\alpha_q^m.
\]

Since $\alpha_q^m=I_q^m$ is the angle determined by the dual circle, Lemma~\ref{lem-3.1}
gives
\begin{equation}
\label{eq:cot-Aq-revised}
\cot \frac{\alpha_q^m}{2}
=
\frac{k_P^m}{\sqrt{(k_q^m)^2-1}} .
\end{equation}
Therefore
\begin{equation}
\label{eq:Aq-derivative-revised}
\frac{\partial I_q^m}{\partial k_P^m}
=
\frac{\partial \alpha_q^m}{\partial k_P^m}
=
-
\frac{2\sqrt{(k_q^m)^2-1}}{(k_P^m)^2+(k_q^m)^2-1}
=
-
\frac{2k_P^m\tan(\alpha_q^m/2)}{(k_P^m)^2+(k_q^m)^2-1} .
\end{equation}
In particular, $\alpha_q^m\to0$ uniformly for $\kappa\in K$.

\begin{figure}[ht]
{
\centering
\includegraphics[height=4cm]{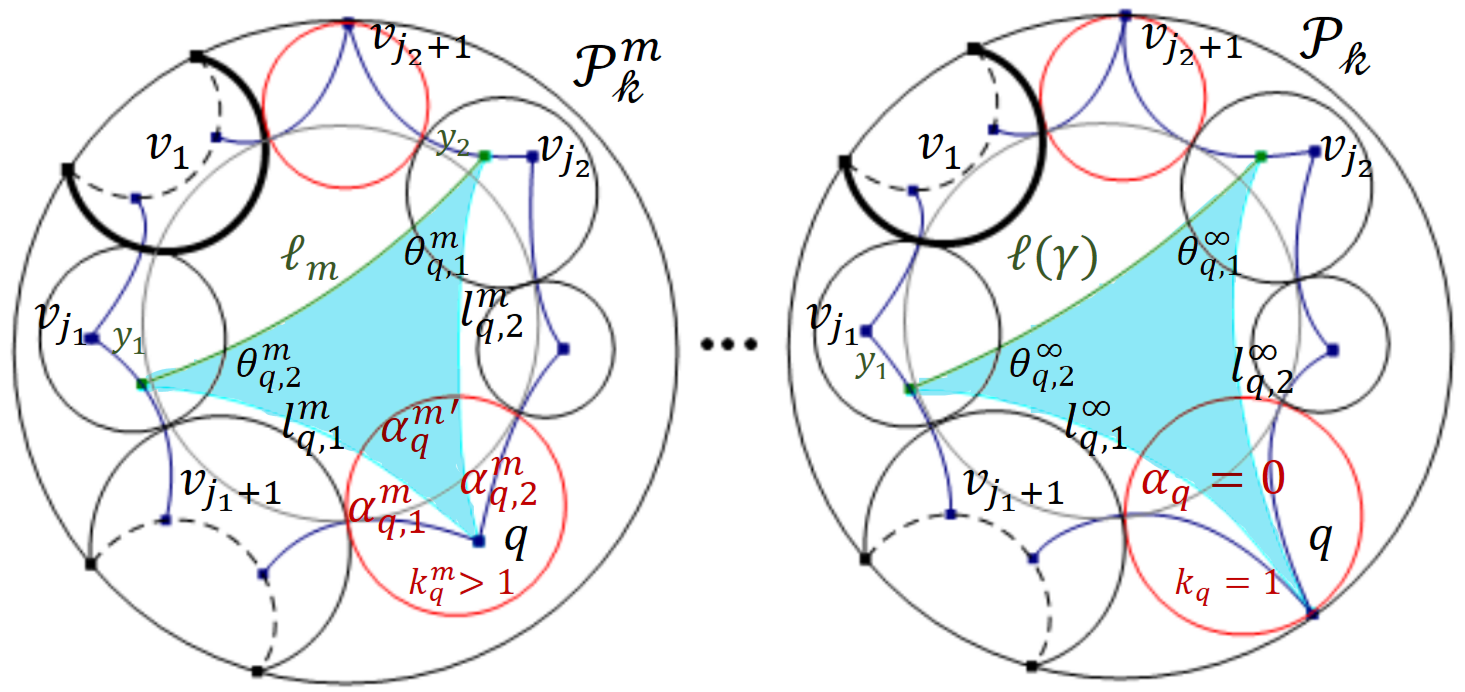} 
\caption{The auxiliary hyperbolic triangle in $\mathcal{P}_{\kappa}^m$ and $\mathcal{P}_{\kappa}$.}
\label{Figuu13}

}
\end{figure}

Next, we estimate $\partial\ell_m/\partial I_q^m$ by considering $I_j^m, j\in \mathcal{J}$, as independent variables of $\ell_m$ and temporarily fixing $I_j^m, j\in \mathcal{J}\setminus \{q\}$.
In the auxiliary triangle at $C_q^m$ (see Figure \ref{Figuu13}), we define $\alpha_q^{m\prime}$ as the angle opposite to the side
$\ell_m$. Let $\alpha_{q,1}^m \text{ and } \alpha_{q,2}^m$ be the two remaining angle portions adjacent
to the endpoints of $\gamma^m_{\kappa}$ on $C_q^m$. Thus, 
\begin{equation}\label{decomposition}
I_q^m=\alpha_{q}^m=\alpha_{q}^{m\prime}+\alpha_{q,1}^m+\alpha_{q,2}^m.\end{equation}

Let $l_{q,1}^m \text{ and }l_{q,2}^m$ be the two sides adjacent to
$\alpha_q^{m\prime}$, and let
$\theta_{q,1}^m \text{ and }\theta_{q,2}^m$ be the other two angles. The opposite side of $\alpha_q^{m\prime}$ is $\ell_m$. 

Applying Lemma \ref{lem-3.6} to the remaining two hyperbolic polygons, we know the lengths of $l_{q,1}^m \text{ and }l_{q,2}^m$ and the values of $\alpha_{q,1}^m \text{ and }\alpha_{q,2}^m$
are also temporarily fixed.
 Thus,
$\partial \alpha_q^{m\prime}/\partial I_q^m=1$.  Combined with the hyperbolic cosine law, we obtain
\begin{equation}
\label{eq:dL-dIq-revised}
\frac{\partial \ell_m}{\partial I_q^m}
=
\frac{\partial \ell_m}{\partial \alpha_q^{m\prime}}
=
\frac{\sinh l_{q,1}^m\sinh l_{q,2}^m\sin\alpha_q^{m\prime}}
{\sinh\ell_m} .
\end{equation}
Using the hyperbolic sine law in the same triangle,
\[
\frac{\sinh \ell_m}{\sin\alpha_q^{m\prime}}
=
\frac{\sinh l_{q,1}^m}{\sin\theta_{q,2}^m}
=
\frac{\sinh l_{q,2}^m}{\sin\theta_{q,1}^m},
\]
we obtain
\begin{equation}
\label{eq:dL-dalpha}
\frac{\partial \ell_m}{\partial I_q^m}
=
\frac{
\sin\theta_{q,1}^m\sin\theta_{q,2}^m\sinh \ell_m
}
{\sin\alpha_q^{m\prime}}.
\end{equation}
Combining \eqref{eq:Aq-derivative-revised} and
\eqref{eq:dL-dalpha}, we obtain the corrected product formula
\begin{equation}
\label{eq:corrected-horocycle-product-revised}
\frac{\partial \ell_m}{\partial I_q^m}
\frac{\partial I_q^m}{\partial k_P^m}
=
-
\frac{
\sin\theta_{q,1}^m\sin\theta_{q,2}^m\sinh\ell_m
}{
\sin\alpha_q^{m\prime}
}
\cdot
\frac{
2k_P^m\tan(\alpha_q^m/2)
}{
(k_P^m)^2+(k_q^m)^2-1
} .
\end{equation}

We now show that the right-hand side of
\eqref{eq:corrected-horocycle-product-revised} is bounded above by a negative
constant, uniformly in $m$ and $\kappa$. Since the limiting family is uniformly
non-degenerate, there exist constants $s_0,L_0>0$ such that, for all sufficiently
large $m$ and all $\kappa\in K$,
\begin{equation}
\label{eq:nondegenerate-lower-bounds-revised}
\sin\theta_{q,1}^m\ge s_0,
\qquad
\sin\theta_{q,2}^m\ge s_0,
\qquad
\sinh\ell_m\ge L_0 .
\end{equation}
Moreover, the dual curvatures $k_P^m$ remain in a compact subinterval
$[p_0,p_1]\Subset(1,+\infty)$, and $k_q^m\to1^+$. Hence, for all sufficiently
large $m$,
\begin{equation}
\label{eq:kp-compact-revised}
0< p_0\le k_P^m\le p_1,
\qquad
(k_P^m)^2+(k_q^m)^2-1\le p_1^2+1 .
\end{equation}
Finally, by \eqref{decomposition},
$0<\alpha_q^{m\prime}\le \alpha_q^m$. Since $\alpha_q^m\to0$ uniformly, for all
sufficiently large $m$ we have $0<\alpha_q^m<\pi/2$ and
$0<\alpha_q^{m\prime}<\pi/2$. Therefore,
\begin{equation}
\label{eq:angle-ratio-lower-revised}
\frac{2\tan(\alpha_q^m/2)}{\sin\alpha_q^{m\prime}}
\ge
\frac{\alpha_q^m}{\alpha_q^{m\prime}}
\ge 1 .
\end{equation}
Using \eqref{eq:nondegenerate-lower-bounds-revised},
\eqref{eq:kp-compact-revised}, and \eqref{eq:angle-ratio-lower-revised} in
\eqref{eq:corrected-horocycle-product-revised}, we get
\begin{equation}
\label{eq:horocycle-negative-uniform-revised}
\frac{\partial \ell_m}{\partial I_q^m}
\frac{\partial I_q^m}{\partial k_P^m}
\le
-
\frac{s_0^2L_0p_0}{p_1^2+1}
=:-b_q<0
\end{equation}
for all sufficiently large $m$ and all $\kappa\in K$. 

If $Q\neq\varnothing$, choose one $q_0\in Q$ and set $b=b_{q_0}$. Since all
terms in the sum
\[
\sum_{j\in\mathcal J}
\frac{\partial \ell_m}{\partial I_j^m}
\frac{\partial I_j^m}{\partial k_P^m}
\]
are non-positive, \eqref{eq:horocycle-negative-uniform-revised} implies
\begin{equation}
\label{eq:sum-negative-revised}
\sum_{j\in\mathcal J}
\frac{\partial \ell_m}{\partial I_j^m}
\frac{\partial I_j^m}{\partial k_P^m}
\le -b<0 .
\end{equation}
If $Q=\varnothing$, then the compact-away-from-$1$ argument above gives the
same conclusion for some index $j_0\in\mathcal J$.

Combining \eqref{eq:chain-rule-63},
\eqref{eq:dual-derivative-63}, and
\eqref{eq:sum-negative-revised}, we obtain
\[
\frac{d\ell_m}{d\kappa}(\kappa)
\le
-a_0 b<0 .
\]
Taking $c=a_0b$ proves that
\[
\frac{d\ell_m}{d\kappa}(\kappa)
\le -c<0,
\qquad
m\ge m_0,\\ \kappa\in K,
\]
for some $m_0\in\mathbb N$. This completes the proof.

\end{proof}

\subsection{Monotonicity of geodesic segments under curvature change}

\begin{lemma}[Monotonicity of geodesic arc length under curvature change]\label{lem-3.8}
Let $\mathcal{P} = \{C_1, \dots, C_n\}$ and $\widetilde{\mathcal{P}} = \{\widetilde{C}_1, \dots, \widetilde{C}_n\}$ be two generalized circle configurations satisfying hypotheses \textup{(H1)}--\textup{(H4)} in Definition \ref{def-2.3}. Let $k_i$ and $\tilde{k}_i$ denote the geodesic curvatures of $C_i$ and $\widetilde{C}_i$, respectively. Assume that $\tilde{k}_i \ge k_i$ for all $i=1,\dots,n$, and there exists at least one index $1\le j\le n$ such that $\tilde{k}_j > k_j$.

Let $P_f$ and $\widetilde{P}_f$ be the induced hyperbolic polygons. Let $\gamma \subset P_f$ and $\tilde{\gamma} \subset \widetilde{P}_f$ be geodesic segments transversely intersecting the respective boundaries at points $y_1, y_2$ and $\tilde{y}_1, \tilde{y}_2$. Assume that $y_i$ and $\tilde{y}_i$ lie on the edges corresponding to the same index, and does not coincide with the center of the horocycle. Let their hyperbolic distances to the adjacent tangent points are preserved; that is, $d(y_i, t_{j_i}) = d(\tilde{y}_i, \widetilde{t}_{j_i})$ for $i=1,2$.

Then the length of $\tilde{\gamma}$ is strictly less than the length of $\gamma$; that is, $\ell(\tilde{\gamma}) < \ell(\gamma)$.
\end{lemma}

\begin{proof}
Without loss of generality, we may assume that there are only one pair of circles in the two circle configurations are given different curvatures, say $C_1$ and $\tilde{C}_1$; that is, we may assume that $\widetilde k_{1} > k_1$ and $\tilde{k}_i = k_i$ for $i=2,\dots, n$.

For the general case where the geodesic curvatures at several vertices vary, the result follows by iterating this argument and invoking the transitivity of the inequality.

\begin{figure}[ht]
{
\centering
\includegraphics[height=4cm]{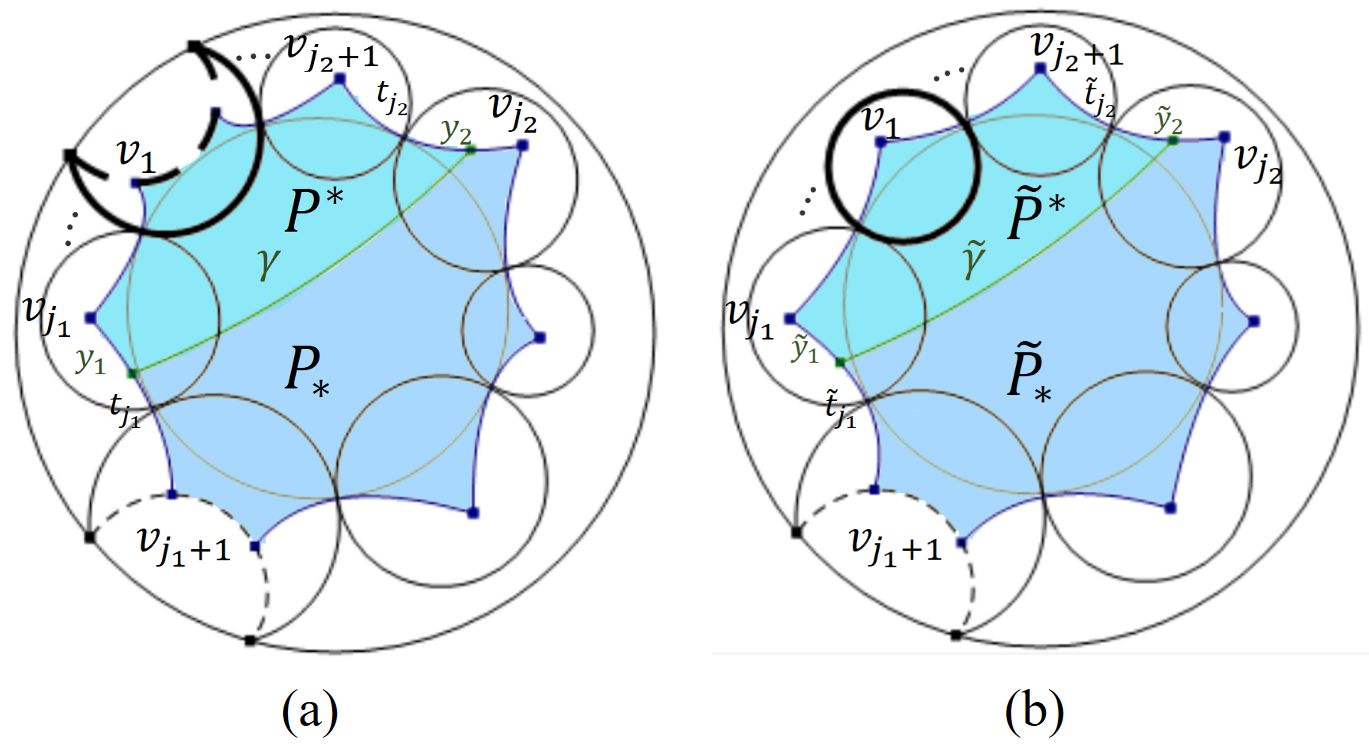} 
\caption{The construction of polygons $P^{*}$ (${P}_{*}$) and $\tilde{P}^{*}$ ($\tilde{P}_{*}$).}
\label{Fig14}

}
\end{figure}

\medskip\noindent\textbf{Step 1: The case where no horocycle appears except possibly $C_1$.}
Firstly, we assume that $k_i=\tilde{k}_i \in (0,1) \cup (1, +\infty)$ for $i=2,\cdots,n$. The geodesic $\gamma$ (resp. $\tilde{\gamma}$) divides $P_f$ (resp. $\tilde{P}_f$) into two compact subpolygons with disjoint interiors. Let $P^{*}$ (resp. $\tilde{P}^{*}$) denote the subpolygon whose boundary consists of $\gamma$ (resp. $\tilde{\gamma}$) together with the portion of the boundary of $P_f$ (resp. $\tilde{P}_f$) that joins $y_1$ to $y_2$ (resp. $\tilde{y}_1$ to $\tilde{y}_2$) and contains the vertex $v_1$ (resp. $\tilde{v}_1$). Moreover, let $P_*$ (resp. $\tilde{P}_*$) denote another subpolygon whose boundary not containing $v_1$. For example, Figure \ref{Fig14} shows $P^*$ and $\partial P^* = \gamma \cup \widehat{y_1 y_2}^{*}$ in the different cases, where $\widehat{y_1 y_2}^{*}$ represents the part of the boundary of $P^*$ connecting $y_1$ to $y_2$ and through the vertices of $P_f$ in the counterclockwise order (which containing $v_1$). Similarly, $\partial P_* = \gamma \cup \widehat{y_1 y_2}_{*}$, where $\widehat{y_1 y_2}_{*}$ represents the part of the boundary of $P_*$ connecting $y_1$ to $y_2$ and through the vertices of $P_f$ in the counterclockwise order (not containing $v_1$). Pictures hold for show $\tilde{P}^*$ and $\tilde{P}_*$.

Consider the dual circle $C_P$ (resp. $\widetilde{C}_P$) associated with the configuration $\mathcal{P}$ (resp. $\widetilde{\mathcal{P}}$). By Lemma~\ref{lem-3.2}, the geodesic curvature $k_P$ of the dual circle is a strictly increasing function of each $k_i$. Since $\widetilde k_{1} > k_1$, it follows that 
\[
k_{\tilde{P}}> k_P.
\]

Then by Lemma~\ref{lem-3.1}, if the curvature $k_P$ of the dual circle increases, then each generalized interior angle $I_i$ of $P_f$ strictly decreases. Thus, $\tilde{I}_i < I_i$ for each $i=2, \cdots, n$.

Since only $C_1$ varies, it follows that each edge of $P_*$ except $\gamma$ has length equal to the corresponding edge of $\tilde{P}_{*}$. From the above step, we have known that  the interior angle between any two edges of $\tilde{P}_*$ (not including $\gamma$) is strictly smaller than the corresponding interior angle of ${P}_*$. 
Using Lemma~\ref{lem-3.6}, we conclude 
\[
\ell(\tilde{\gamma}) < \ell(\gamma).
\]

\begin{figure}[ht]
{
\centering
\includegraphics[height=4cm]{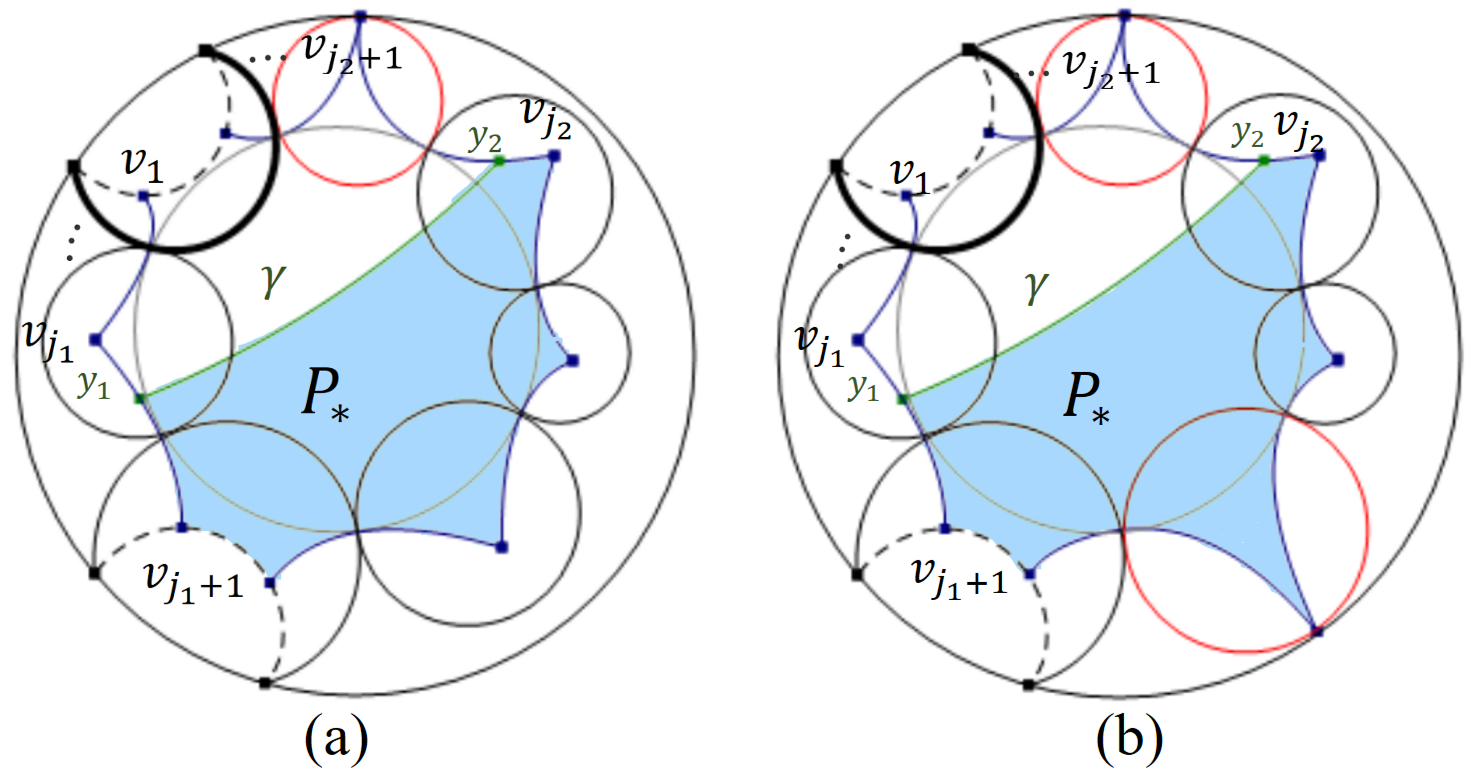} 
\caption{The cases with horocycles.}
\label{Fig15}
}
\end{figure}

\medskip\noindent\textbf{Step 2: The case when all the horocycles lie on the same side of $\gamma$ as $v_1$.}

We now consider configurations containing horocycles. We define the subpolygons $P_*$ and $\widetilde{P}_*$ as in Step 1. In this case, the subpolygon $P_*$ \textit{does not contain} any horocycle center. (see Figure~\ref{Fig15} (a)).

Since $P_*$ does not contain $v_1$, it contains no horocycle center, and the monotonicity follows by the same argument as in Step 1. By Lemma \ref{lem-3.2}, Lemma \ref{lem-3.1} and Lemma \ref{lem-3.6}, we have 
\[\ell(\tilde{\gamma})<\ell(\gamma).\]

\medskip
\noindent\textbf{Step 3: The general case with horocycles on both sides.}
Assume now that horocycles occur on both sides of $\gamma$.  

We define
$P_{*}$ and $\widetilde P_{*}$ as subpolygons not containing the
vertex corresponding to the varying generalized circle $C_{1}$ (see Figure \ref{Fig15} (b)).

Denote by $K=[k_{1},\widetilde k_{1}]$. Clearly, $K\subset (0,+\infty)$.

Approximate each horocycle contained in $\partial P_{*}$ by a sequence of hyperbolic circles with curvature $k_{j}^{m}$ that 
satisfies $k_{j}^{m}\to 1$ as $m\to\infty$. For each $m$, let
\[
\mathcal{P}_{\kappa}^{(m)},\qquad \kappa\in K,
\]
be the resulting generalized circle configuration in which the curvature of
$C_{1}$ is $\kappa$, while all other curvatures are fixed.  Denote by
$\gamma_{\kappa}^{m}$ the corresponding geodesic segment and put
\[
\ell_{m}(k_1):=\ell(\gamma_{\kappa}^{m}).
\]

By Lemma~\ref{lem:uniform-negative-derivative}, there exist constants
$m_{0}\in\mathbb N$ and $c>0$ such that
\[
\frac{d\ell_{m}}{d\kappa}(\kappa)\le -c<0
\text{ for any }
m\ge m_{0}, \text{ and any }\kappa\in K .
\]
Therefore, for $m\ge m_{0}$, the fundamental theorem of calculus gives
\[
\ell_{m}(\widetilde k_{1})-\ell_{m}(k_{1})
=
\int_{k_{1}}^{\widetilde k_{1}}
\frac{d\ell_{m}}{d\kappa}(\kappa)\,d\kappa
\le
-c(\widetilde k_{1}-k_{1})<0 .
\]
By Lemma \ref{lem-3.7}, the geodesic lengths converge when the approximating
generalized circles tend to the corresponding horocycles.  Hence, letting
$m\to\infty$, we obtain the following:
\[
\ell(\widetilde\gamma)-\ell(\gamma)
=
\lim_{m\to\infty}
\bigl(\ell_{m}(\widetilde k_{1})-\ell_{m}(k_{1})\bigr)
\le
-c(\widetilde k_{1}-k_{1})<0 .
\]
Thus,
\[
\ell(\widetilde\gamma)<\ell(\gamma),
\]
which means that the strict monotonicity of the geodesic length is preserved in the
horocyclic limiting case.

\medskip\noindent\textbf{Step 4: Multiple varying curvatures.}

Recall the hypothesis of this lemma: $\tilde{k}_i \ge k_i$ for all $i=1, \dots, n$, and there exists at least one index $m \in \{1, \dots, n\}$ such that $\tilde{k}_m > k_m$.

To rigorously establish the strict decrease in geodesic length when multiple curvatures vary, we construct a finite sequence of intermediate generalized circle configurations $\mathcal{P}_0, \mathcal{P}_1, \dots, \mathcal{P}_n$. 

Let $\mathcal{P}_0 = \mathcal{P}$ be the initial configuration with the curvature vector $\mathbf{k}^{(0)} = (k_1, k_2, \dots, k_n)$. 
For each $j \in \{1, 2, \dots, n\}$, we define the intermediate configuration $\mathcal{P}_j$ associated with the curvature vector $\mathbf{k}^{(j)}$, where the first $j$ components are updated to the new curvatures, and the remaining components remain at their initial values:
$$\mathbf{k}^{(j)} = (\widetilde k_{1}, \dots, \widetilde{k}_j, k_{j+1}, \dots, k_n).$$
Let $\gamma_j$ denote the corresponding geodesic arc in the induced polygon for the configuration $\mathcal{P}_j$. Note that $\mathcal{P}_n = \tilde{\mathcal{P}}$ is the final configuration, so $\gamma_n = \tilde{\gamma}$ and $\gamma_0 = \gamma$.

We analyze the transition from $\mathcal{P}_{j-1}$ to $\mathcal{P}_j$ for each $j = 1, \dots, n$. In this step, only the $j$-th curvature is modified (from $k_j$ to $\tilde{k}_j$), while all other $n-1$ curvatures remain strictly fixed.
\begin{itemize}
    \item If $\tilde{k}_j = k_j$, the curvature vector does not change ($\mathbf{k}^{(j)} = \mathbf{k}^{(j-1)}$). The geometric configuration is identical, thus $\ell(\gamma_j) = \ell(\gamma_{j-1})$.
    \item If $\tilde{k}_j > k_j$, exactly one curvature strictly increases while all others are fixed. This exactly satisfies the conditions analyzed in Steps 1 through 3. Based on the conclusions established in the preceding steps, this strict increase induces a strict decrease in the geodesic length, resulting in $\ell(\gamma_j) < \ell(\gamma_{j-1})$.
\end{itemize}

Since $\tilde{k}_i \ge k_i$ for all $i$, we have a non-increasing sequence of geodesic lengths:
$$\ell(\gamma_n) \le \ell(\gamma_{n-1}) \le \cdots \le \ell(\gamma_1) \le \ell(\gamma_0).$$

Furthermore, the hypothesis guaranties that there is at least one index $m$ where the curvature strictly increases ($\tilde{k}_m > k_m$). Therefore, at least one of the inequalities in the sequence above is strict (i.e., $\ell(\gamma_m) < \ell(\gamma_{m-1})$). 

By the transitivity of inequalities, the presence of at least one strict inequality in the non-increasing chain ensures that the final length is strictly less than the initial length; that is,
$$\ell(\gamma_n) < \ell(\gamma_0).$$
This means
$$\ell(\tilde{\gamma}) < \ell(\gamma).$$
\end{proof}

\begin{remark}\label{outside the induced hyperbolic polygon}
Under the same notation as in Lemma \ref{lem-3.8}, suppose that for $i=1$ or $2$, the radius of the generalized circle $\widetilde{C}_{j_i}$ becomes smaller than the fixed distance from $\tilde{y}_i$ to the adjacent tangent point. Then $\tilde{y}_i$ lies outside the induced hyperbolic polygon $\widetilde{P}_f$, specifically on the extension of the edge $\tilde{e}_{j_i}$ beyond ${v}_{j_i}$ (or ${v}_{j_i+1}$). Even in this case, the geodesic $\tilde{\gamma}$ connecting $\tilde{y}_1$ and $\tilde{y}_{2}$ satisfies 
$$\ell(\tilde{\gamma}) < \ell(\gamma).$$

\begin{figure}[ht]
{
\centering
\includegraphics[height=4cm]{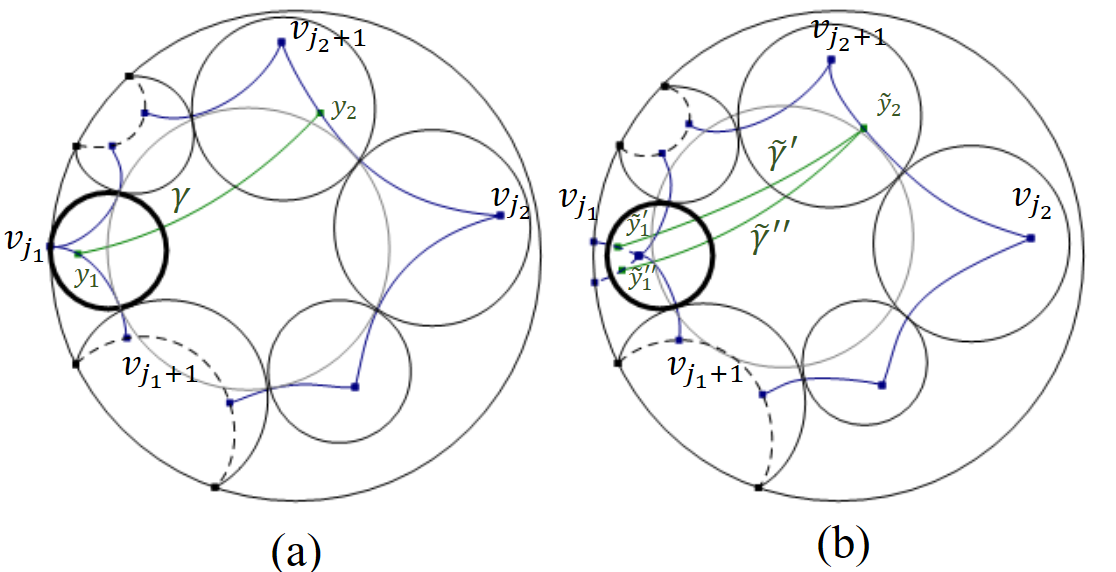} 
\caption{The case outside the induced hyperbolic polygon.}
\label{Fig16}

}
\end{figure}

\end{remark}

\section{Generalized discrete boundary value theorem based on polyhedral decomposition}

\subsection{Limit behavior}
In this section, we consider the limit behavior of the total geodesic curvature vector as the geodesic curvatures of some generalized circles go to $0$ or $\infty$, and the size of the set $\mathcal{T}$ used in Theorem \ref{theorem-1.1}.

Set $\textbf{c}=\left(c_1, \cdots, c_n\right) \in[0,+\infty]^n$ as a non-negative vector (possible infinity). For a generalized circle configuration $\{C_i\}_{i=1}^n$
associated with a topological polygon $f$ with a given geodesic curvature vector $\textbf{k}$, we consider the limiting behavior of the total geodesic curvature of each side of the interstice $\Omega_f$ bounded by this circle configuration as 
$$
\textbf{k}=(k_1,\cdots,k_n) \rightarrow \left(c_1, \cdots, c_n\right)=\textbf{c}.
$$
Denote by $T_{j,P}$ the total geodesic curvature of the side of $\Omega_f$ on the generalized circle centered at $v_j$. The following two lemmas are proved in \cite{MR4842734}. 
\begin{lemma}[\cite{MR4842734} Lemma 3.2]\label{lem-4.1}
If $c_j=0$ for some $1 \leq j \leq n$, then $$\lim _{\textbf{k}\rightarrow \textbf{c}} T_{j, P}(\textbf{k})=0.$$
\end{lemma}








\begin{lemma}[\cite{MR4842734} Lemma 3.3]\label{lem-4.2}
Let $I=\{j:\;c_j=\infty\}$ and $|I|$ denote the number of elements of $I$. Then the following two properties hold.  

(i) If $|I|<n-2$, then $\lim _{\textbf{k}\rightarrow \textbf{c}} T_{j, P}(\textbf{k})=\pi$ for each $j\in I$; 

(ii) If $n-2\leq|I|\leq n$, then $\lim _{\textbf{k}\rightarrow \textbf{c}} \sum_{j=1}^nT_{j, P}(\textbf{k})=(n-2)\pi$.
\end{lemma}
The following figure illustrates the two conclusions given in the previous lemma \ref{lem-4.2}.
\begin{figure}[ht]
{
\centering
\includegraphics[height=3cm]{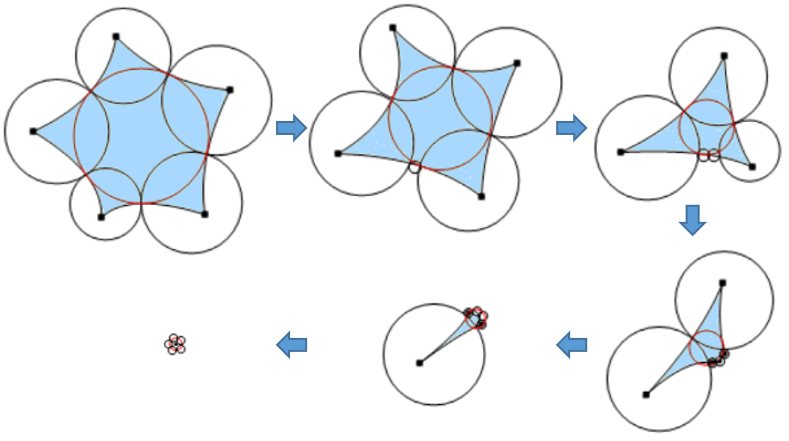} 
\caption{From a circle configuration to a similar one in a small scale.}
\label{Fig17}
}
\end{figure}






\subsection{Discrete boundary value theorem}

\begin{theorem}\label{theorem-5.1}
Let $S_{g,n}$ be a compact oriented topological surface with finitely many genus $g\geq 0$ and finitely many holes $n\geq 1$ and let $\mathcal{D}$ be a polygonal cellular decomposition of $S_{g,n}$. Let $G=\{V,E,F\}$ be the $1$-skeleton of $\mathcal{D}$, where $V=V^{\partial}\cup V^{\circ}$. Assume that  $\mathcal{P}$ is a generalized circle on $S_{g,n}$ with $G$ as a contact graph. Denote by $\textbf{T}(v_i)$ the total geodesic curvature of a conical generalized circle centered at an interior vertex $v_i\in V^{\circ}$. Then
$$
\textbf{T}:V^{\circ}\longrightarrow \mathbb{R}_{+}:
v_i\longmapsto \textbf{T}(v_i)
$$
belongs to $\mathcal{T}$, where $\mathcal{T}$ is defined by Condition (\ref{equa-1}).
\end{theorem}

\begin{proof}
For each face \(f\in F\), let \(\Omega_f\) be the interstice
corresponding to \(f\). For \(v\in V(f)\), denote by \(T_{v,f}\) the
total geodesic curvature of the side of \(\partial\Omega_f\) which lies
on the generalized circle centered at \(v\). Then, for every interior
vertex \(v\in V^\circ\),
\[
    \textbf{T}(v)=\sum_{\substack{f\in F\\ v\in V(f)}} T_{v,f}.
\]

We first record a local estimate. For every face \(f\in F\) and every
\(v\in V(f)\),
\begin{equation}\label{equa-15}
    0<T_{v,f}<\pi .
\end{equation}

Indeed, let \(k_v\) be the geodesic curvature of the generalized circle
centered at \(v\), and let \(k_{P_f}>1\) be the geodesic curvature of
the dual circle associated with the generalized circle configuration over
\(f\). By the elementary right-quadrilateral formulas for the three types
of generalized circles, we have
\[
T_{v,f}=
\begin{cases}
\displaystyle
\frac{2k_v}{\sqrt{k_v^2-1}}\,
\text{arccot}\left(\frac{k_{P_f}}{\sqrt{k_v^2-1}}\right),
& k_v>1,\\[1.2em]
\displaystyle
\frac{2}{k_{P_f}},
& k_v=1,\\[1.2em]
\displaystyle
\frac{2k_v}{\sqrt{1-k_v^2}}\,
\operatorname{arccoth}\!\left(\frac{k_{P_f}}{\sqrt{1-k_v^2}}\right),
& 0<k_v<1.
\end{cases}
\]
These quantities are positive. Moreover, since \(k_{P_f}>1\), each of
them is strictly less than \(\pi\). More explicitly, in the circle case,
writing \(k_v=\coth r\), we obtain
\[
T_{v,f}
<
2\cosh r\cdot\text{arccot}(\sinh r)
<\pi,
\]
where the last inequality follows from
\[
\arctan x<\frac{\pi x}{2\sqrt{1+x^2}},\qquad x>0.
\]
In the hypercycle case, writing \(k_v=\tanh r\), we obtain
\[
T_{v,f}
<
2\sinh r\,\operatorname{arccoth}(\cosh r)
<2<\pi,
\]
where we use
\[
\operatorname{arccoth} x<\frac{1}{\sqrt{x^2-1}},\qquad x>1.
\]
The horocycle case gives \(T_{v,f}=2/k_{P_f}<2<\pi\). Thus (\ref{equa-15})
holds.

Next, let \(f\in F\) and put \(n_f=N(f)=\#V(f)\). Applying the
Gauss--Bonnet formula to the interstice \(\Omega_f\), whose boundary is
formed by the arcs of the generalized circles corresponding to the
vertices of \(f\), gives
\[
    \sum_{v\in V(f)} T_{v,f}
    =
    \pi(n_f-2)-\operatorname{Area}(\Omega_f).
\]
Since \(\operatorname{Area}(\Omega_f)>0\), we have
\begin{equation}\label{equa-16}
    \sum_{v\in V(f)} T_{v,f}
    <
    \pi\bigl(N(f)-2\bigr).
\end{equation}

Now let \(Q\subset V^\circ\) be any nonempty subset. We need to prove
the defining inequality of \(\mathcal T\). Using the decomposition of
\(\textbf{T}(v)\) into face-wise contributions, we have
\[
\begin{aligned}
    \sum_{v\in Q} T(v)
    &=
    \sum_{v\in Q}
    \sum_{\substack{f\in F\\ v\in V(f)}} T_{v,f}      \\
    &=
    \sum_{f\in F_Q}
    \sum_{v\in Q\cap V(f)} T_{v,f}.
\end{aligned}
\]
Fix \(f\in F_Q\). By (\ref{equa-15}),
\begin{equation}\label{equa-17}
    \sum_{v\in Q\cap V(f)} T_{v,f}
    <
    \pi\,\#\bigl(Q\cap V(f)\bigr)
    =
    \pi N(f,Q).
\end{equation}
On the other hand, since all \(T_{v,f}\) are positive, (\ref{equa-16}) implies
\begin{equation}\label{equa-18}
    \sum_{v\in Q\cap V(f)} T_{v,f}
    \le
    \sum_{v\in V(f)} T_{v,f}
    <
    \pi\bigl(N(f)-2\bigr).
\end{equation}
Combining (\ref{equa-17}) and (\ref{equa-18}), we obtain
\[
    \sum_{v\in Q\cap V(f)} T_{v,f}
    <
    \pi\min\{N(f,Q),\,N(f)-2\}.
\]
Summing this inequality over all \(f\in F_Q\), we get
\[
\begin{aligned}
    \sum_{v\in Q} T(v)
    &=
    \sum_{f\in F_Q}
    \sum_{v\in Q\cap V(f)} T_{v,f}        \\
    &<
    \sum_{f\in F_Q}
    \pi\min\{N(f,Q),\,N(f)-2\}.
\end{aligned}
\]
This is precisely Condition \((1)\). Therefore \(T\in\mathcal T\).
\end{proof}

\begin{theorem}\label{theorem-5.2}
    Let $S_{g,n}$ and $G$ be defined as above. Given $\hat{\textbf{k}}:V^{\partial}\rightarrow \mathbb{R}_{+}$ and $\hat{\textbf{T}}:V^{\circ}\rightarrow \mathbb{R}_{+}$. If $\hat{\textbf{T}}\in \mathcal{T}$, there exists a unique generalized circle packing on $S_{g,n}$, up to isometry,
with geodesic curvature vector \(\textbf{k}:V\to\mathbb R_+\) satisfying:
    \begin{itemize}
    \item[(i)] For any boundary vertex $w\in V^{\partial}$, the generalized circle satisfies $\textbf{k}(w)=\hat{\textbf{k}}(w)$;
    \item [(ii)] For any interior vertex $v_0\in V^{\circ}$, the total geodesic curvature $\textbf{T}_{\textbf{k}}(v_0)$ satisfies $\textbf{T}_{\textbf{k}}(v_0)=\hat{\textbf{T}}(v_0)$.
    \end{itemize}
\end{theorem}
\begin{proof}
We apply the Perron method to prove this result. This idea comes from \cite{MR1115988}. Each geodesic curvature vector $\textbf{k}:V\rightarrow\mathbb{R}_{+}$ determines a unique generalized circle packing on $S_{g,n}$ with the contact graph $G$, which we briefly call a generalized circle packing $\textbf{k}$ for $G$. Let 
\begin{equation}\label{equa-34}
\mathcal{K}_1=\left\{\textbf{k}:V\rightarrow\mathbb{R}_{+}\begin{array}{|l}\textbf{k}\text{ is a generalized circle packing for $G$ with}\\
\text{$\textbf{T}_{\textbf{k}}(v)\leq\hat{\textbf{T}}(v) \text{ for any } v\in V^{\circ}$ and prescribed}\\
\text{boundary values $\textbf{k}(w)\le \hat{\textbf{k}}(w)\text{ for any } w\in V^{\partial}$}\end{array}\right\}.
\end{equation}

\medskip
\noindent\textbf{Step one:} We show that $\mathcal{K}_1$ is not empty.

Let $\textbf{k}: V\rightarrow \mathbb{R}_+$ be a geodesic curvature vector such that on the boundary vertices $\textbf{k}$ takes the same values as $\hat{\textbf{k}}$. By Lemma \ref{lem-4.1}, we may choose the values of $\textbf{k}$ at the interior vertices small enough so that all total geodesic curvatures of the corresponding hyperbolic arcs are less than or equal to $\frac{m}{d}$, where $m=\min_{v\in V^{\circ}}{\hat{\textbf{T}}}(v)$ and $d$ is the maximum number of edges with an end point at a vertex $v\in V^{\circ}$. Therefore, $\textbf{k}\in\mathcal{K}_1$; that is, the set of $\mathcal{K}_1$ is not empty. 

\medskip
\noindent\textbf{Step two:} $\mathcal{K}_1$ has the following so-called net property; that is, for $\textbf{k}_1^{\prime},\textbf{k}_1^{\prime\prime}\in\mathcal{K}_1$, the new geodesic curvature vector 
$\textbf{k}_1: V\rightarrow \mathbb{R}_+$ defined by 
\begin{equation}\label{equa-38}
\textbf{k}_1(v)=\max\{{\textbf{k}_1^{\prime}}(v),{\textbf{k}_1^{\prime\prime}}(v)\}, v\in V, 
\end{equation}
belongs to $\mathcal{K}_1$ as well. 

Let $v_0\in V^{\circ}$. Suppose that $\mathbf{k}_1^{\prime}(v_0)\geq \mathbf{k}_1^{\prime\prime}(v_0)$. Then $\textbf{k}(v_0)=\textbf{k}_1^{\prime}(v_0)$. Denote by $V(v_0)$ the collection of the neighboring vertices of $v_0$. Then $\textbf{k}(w)\geq \textbf{k}_1^{\prime}(w)$, $\forall \; w\in V(v_0)$. By the monotonicity property (Lemma \ref{lem-3.3} (b)), the total geodesic curvature decreases as the geodesic curvatures of the neighboring circles increase. Therefore, $\textbf{T}_{\textbf{k}}(v_0)\leq \textbf{T}_{\textbf{k}_1^{\prime}}(v_0)\leq \hat{\textbf{T}}(v_0),\forall \; v_0\in V^{\circ}$. Thus, $\mathbf{k}_1\in \mathcal{K}_1$.

\medskip
\noindent\textbf{Step three:} There is an uniform upper bound for all $\textbf{k}\in \mathcal{K}_1$. 

Define $\tau_1:V\rightarrow \mathbb{R}_+\cup \{+\infty\}$ by letting 
\begin{equation}\label{equa-36}
\tau_1(v)=\sup_{\textbf{k}\in \mathcal{K}_1}\textbf{k}(v) \text{ for any }v\in V.\end{equation}
To prove the non-emptiness of $\mathcal{K}_1$, we have chosen $\textbf{k}$ to have the same values as $\hat{\textbf{k}}$ on the boundary vertices. Thus, $\tau_1$ takes the same values of $\hat{\textbf{k}}$ at the boundary vertices.
It suffices to show that the value of $\tau_1$ at each interior vertex is finite. Suppose not, assume that there is a nonempty subset $\tilde{V}\subset V^{\circ}$ such that the value of $\tau_1$ at each vertex in $\tilde{V}$ is $+\infty $. Using the net property, one can obtain a sequence $\{\textbf{k}_m\}_{m=1}^{\infty}\in \mathcal{K}_1$ such that $\lim_{m\rightarrow \infty}\textbf{k}_m=\tau_1$. Since $\textbf{k}_m\in \mathcal{K}_1$, $\textbf{T}_{\textbf{k}_m}(v)\leq \hat{\textbf{T}}(v)$ for all $v\in V^{\circ}$.
On the other hand, applying Lemma \ref{lem-4.2} we obtain
$$\lim_{m\rightarrow \infty}\sum_{v\in \tilde{V}} \textbf{T}_{\textbf{k}_m}(v)=\sum_{f\in F_{\tilde{V}}}\pi\min\{N(f,\tilde{V}),N(f)-2\}.$$
Thus, if $m$ is large enough, then
$$\sum_{v\in \tilde{V}} \textbf{T}_{\textbf{k}_m}(v)>\sum_{v\in \tilde{V}} \hat{\textbf{T}}(v),$$
which is a contradiction to $\textbf{k}_m\in \mathcal{K}_1$. Therefore, the value of $\tau_1$ at each interior vertex is finite. 
In summary, $\tau_1: V \rightarrow \mathbb{R}_+$ is a well-defined geodesic vector which takes the same values as $\hat{\textbf{k}}$ on the boundary vertices. 

\medskip
\noindent\textbf{Step four:} We prove that $\tau_1$ is the target generalized circle packing; that is, it is a solution of the discrete boundary value problem given in Theorem \ref{theorem-1.1}.

By the continuity of the total geodesic curvature depending on $\textbf{k}$, we know $\tau_1\in \mathcal{K}_1$. It remains to show that $T_{\tau_1}(v)=\hat{\textbf{T}}(v)$ at each $v\in V^{\circ}$. Suppose not, there exists $v\in V^{\circ}$ such that $T_{\tau_1}(v)<\hat{\textbf{T}}(v)$. By Lemma \ref{lem-3.3-1} (a) and (b), we increase $\tau_1(v)$ to a value such that $T_{\tau_1}(v)=\hat{\textbf{T}}(v)$ and keep the same value of $\tau_1$ at other vertices $v\in V^{\circ}$ to obtain a new $\tau_1\in \mathcal{K}_1$. This produces a contradiction to the definition of $\tau_1(v)$. Therefore, $T_{\tau_1}(v)=\hat{\textbf{T}}(v)$ for any $v\in V^{\circ}$.

\medskip
\noindent\textbf{Step five:} We show that $\tau_1$ is the unique solution of the discrete boundary value problem. 

Assume that $\tau_1'$ is another solution. Let $\Omega_{\tau_1}$ and $\Omega_{\tau_1'}$ be the union of the interstices of the circle packings determined by $\tau$ and $\tau_1'$ respectively. We show $\tau_1'=\tau_1$ by comparing the areas of $\Omega_{\tau_1}$ and $\Omega_{\tau_1'}$. 

\begin{figure}[ht]
{
\centering
\includegraphics[height=5cm]{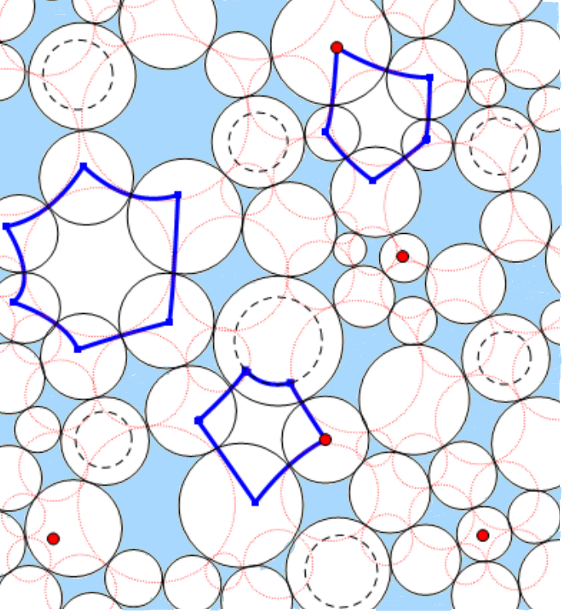} 
\caption{The interstice of a generalized circle packing.}
\label{Fig18}

}
\end{figure}

By the Gauss-Bonnet theorem, 
$$
Area_{\tau_1}(\Omega_{\tau_1})=\pi\sum_{f\in F} (N(f)-2)-\sum_{v_i\in V^{\circ}}\hat{\textbf{T}}(v_i)-\sum_{w_j\in V^{\partial}}\textbf{T}_{\tau_1}(w_j),
$$
where $\textbf{T}_{\tau_1}(w_j)$ stands for the total geodesic curvature of the part of the circle centered at the boundary vertex $w_j$ as an edge of an interstice. 

similarly, 
$$
Area_{\tau_1^{\prime}}(\Omega_{\tau_1'})=\pi\sum_{f\in F} (N(f)-2)-\sum_{v_i\in V^{\circ}}\hat{\textbf{T}}(v_i)-\sum_{w_j\in V^{\partial}}\textbf{T}_{\tau_1^{\prime}}(w_j).
$$

Clearly, $\tau_1'\in \mathcal{K}_1$. Thus,  $\tau_1(v)\geq \tau_1^{\prime}(v)$ for any $v\in V^{\circ}$. By Lemma \ref{lem-3.3-1} (b), $\textbf{T}_{\tau_1}(w_j)\leq \textbf{T}_{\tau_1^{\prime}}(w_j)$ for any $w_j\in V^{\partial}$. Therefore, the previous two equations imply that $Area_{\tau_1}(\Omega_{\tau_1})\geq Area_{\tau_1^{\prime}}(\Omega_{\tau_1'})$. On the other hand, by the monotonicity of the area (Lemma \ref{lem-3.3-1} (d)), $Area_{\tau_1}(\Omega)\leq Area_{\tau_1^{\prime}}(\Omega)$. Thus, $Area_{\tau_1}(\Omega)= Area_{\tau_1^{\prime}}(\Omega)$. Furthermore, it follows from $\tau_1(v)\geq \tau_1^{\prime}(v)$ for any $v\in V$ and the strict monotonicity of the area (Lemma \ref{lem-3.3-1} (d)) that $\tau_1=\tau_1'$.
\end{proof}

\begin{remark}\label{mathcalK-1'} To prove Theorem \ref{theorem-5.2}, one may replace $\mathcal{K}_1$ by 
\begin{equation}\label{equa-34'}
\mathcal{K}_1'=\left\{\textbf{k}:V\rightarrow\mathbb{R}_{+}\begin{array}{|l}\textbf{k}\text{ is a generalized circle packing for $G$ with}\\
\text{$\textbf{T}_{\textbf{k}}(v)\leq\hat{\textbf{T}}(v) \text{ for any } v\in V^{\circ}$ and prescribed}\\
\text{boundary values $\textbf{k}(w)=\hat{\textbf{k}}(w)\text{ for any } w\in V^{\partial}$}\end{array}\right\}.
\end{equation}
The reason we use $\mathcal{K}_1$ instead of $\mathcal{K}_1'$ in the proof is that we will apply $\mathcal{K}_1$ to show the Schwarz-Pick property I (Theorem \ref{theorem-6.1}).
\end{remark}

\begin{remark}\label{mathcalK-2'} Let
\begin{equation}\label{equa-35}
\mathcal{K}_2=\left\{\textbf{k}:V\rightarrow\mathbb{R}_{+}\begin{array}{|l}\textbf{k}\text{ is a generalized circle packing for $G$ with}\\
\text{$\textbf{T}(v)\geq\hat{\textbf{T}}(v),\;\forall \;v\in V^{\circ}$ and prescribed}\\
\text{boundary values $\textbf{k}(w)\ge \hat{\textbf{k}}(w),w\in V^{\partial}$}\end{array}\right\}
\end{equation}
and 
\begin{equation}\label{equa-35'}
\mathcal{K}_2'=\left\{\textbf{k}:V\rightarrow\mathbb{R}_{+}\begin{array}{|l}\textbf{k}\text{ is a generalized circle packing for $G$ with}\\
\text{$\textbf{T}(v)\geq\hat{\textbf{T}}(v),\;\forall \;v\in V^{\circ}$ and prescribed}\\
\text{boundary values $\textbf{k}(w)= \hat{\textbf{k}}(w),w\in V^{\partial}$}\end{array}\right\}.
\end{equation}
Given $\textbf{k}_2^{\prime},\textbf{k}_2^{\prime\prime}\in\mathcal{K}_2$ (resp. $\mathcal{K}_2'$),  define  
$\textbf{k}_2: V\rightarrow \mathbb{R}_+$ as 
\begin{equation}\label{equa-38'}
\textbf{k}_2(v)=\min\{{\textbf{k}_2^{\prime}}(v),{\textbf{k}_2^{\prime\prime}}(v)\} \;\text{ for any } v\in V,
\end{equation}
and $\tau_2:V\rightarrow \mathbb{R}_+\cup \{+\infty\}$ as
\begin{equation}\label{equa-36'}
\tau_2(v)=\inf_{\textbf{k}\in \mathcal{K}_2 \text{ (resp. $\mathcal{K}_2'$) }} \textbf{k}(v) \;\text{ for any }v\in V.\end{equation}
Then we can utilize $\mathcal{K}_2$ (resp. $\mathcal{K}_2'$)  and the steps of the previous proof to give an alternative proof of Theorem \ref{theorem-5.2}.
\end{remark}

\section{Convergence of Thurston's algorithm}\label{sec-5.2}
Let $S_{g,n}$ be a compact oriented topological surface with finitely many genus $g\geq 0$ and finitely many holes $n\geq 1$ and let $\mathcal{D}$ be a polygonal cellular decomposition of $S_{g,n}$. Our Theorem \ref{theorem-1.1} develop a necessary and sufficient condition for the discrete boundary value problem to have a unique solution. 
In this subsection, we show the convergence of Thurston's algorithm to this unique solution; that is, we prove Theorem \ref{theorem-1.3}. 

\begin{proof}[Proof of Theorem \ref{theorem-1.3}] Let $V^{\partial}$ and $V^{\circ}$ denote the sets of the boundary vertices and interior vertices of $\mathcal{D}$, respectively. For any initial vector $(k_v^0)_{v\in V^{\circ}}$, there is a generalized circle packing $\mathcal{P}^0$ such that the geodesic curvature of the generalized circle centered at the boundary vertex $w$ is $\hat{k}(w)$ for every $w\in V^{\partial}$ and the geodesic curvature of the generalized circle centered at the interior vertex $v$ is $k^0_v$ for every $v\in V^{\circ}$. Denote by $\textbf{k}^0$ the geodesic curvature vector of $\mathcal{P}^0$.

Let $\{\textbf{k}^m\}_{m=0}^{\infty}$ be the sequence of geodesic curvature vectors constructed through the Thurston algorithm in Section \ref{Thurston algorithm} and let $\mathcal{P}^m$ be the generalized circle packing determined by $\textbf{k}^m$. Since the values of $\textbf{k}^m$ maintain the same values as ${\textbf{k}}^*$ on the boundary vertices, we only need to consider the geodesic curvatures $({k_1}^m,k^m_2,\cdots,k^m_{|V^{\circ}|})$ of the generalized circles centered at the interior vertices.

Let 
$$
\textbf{s}^m=\langle s_1^m,s_2^m,\cdots,s_{|V^{\circ}|}^m\rangle=\langle \ln k_1^m,\ln k^m_2, \cdots, \ln k^m_{|V^{\circ}|}\rangle.
$$
The process of the Thurston algorithm from an input to an output can be expressed by the following total geodesic curvature adjusting map $\Psi$:
\begin{equation}\label{the adjusting map for the algorithm}
  \Psi:\mathbb{R}^{|V^{\circ}|}\longrightarrow \mathbb{R}^{|V^{\circ}|}:
\textbf{s}^0 \longmapsto\textbf{s}^1  
\end{equation}
Then $\Psi^m(\textbf{s}^0)=\textbf{s}^m$. 
Let ${\textbf{k}}^*$ be a solution of the discrete boundary value problem (Theorem \ref{theorem-5.2}), and let
$${\textbf{s}}^{*}=\langle {s}_1^*,{s}_2^*,\cdots,{s}_{|V^{\circ}|}^*\rangle=\langle \ln {k}_1^*,\ln {k}_2^*, \cdots, \ln {k}_{|V^{\circ}|}^*\rangle.$$
Clearly, ${\textbf{s}}^{*}$ is a fixed point of the adjusting map. We show that $\Psi$ is a contraction map on any compact box $X$ centered at ${\textbf{s}}^{*}$ in $\mathbb{R}^{|V^{\circ}|}$, which implies that the fixed point of $\Psi$ is unique.

Let $\textbf{T}^m=\langle T_1^m,T_2^m,\cdots,T^m_{|V^{\circ}|}\rangle$ be the vector of total geodesic curvatures of the generalized circles of $\mathcal{P}^m$ centered at the interior vertices. Now we view $\textbf{T}^m$ as a vector function $\mathscr T$ of $\textbf{s}$; that is, 
\begin{equation}
  \mathscr T:\mathbb{R}^{|V^{\circ}|}\longrightarrow \left(\mathbb{R}_{+}\right)^{|V^{\circ}|}:
\langle s_1,s_2,\cdots,s_{|V^{\circ}|} \rangle\longmapsto \langle \mathscr T_1,\mathscr T_2,\cdots,\mathscr T_{|V^{\circ}|}\rangle.  
\end{equation}

Given $v_i\in V^{\circ}$, denote by $j\sim i$ if there is an edge between $v_j$ and $v_i$. Applying Lemma \ref{lem-3.3-1}, we obtain the following monotone results. 
\begin{proposition}\label{prop5.2}
\begin{itemize}
\item [(1)] $\sum_{j=1}^{|V^{\circ}|}\frac{\partial \mathscr T_i}{\partial s_j}=\frac{\partial \mathscr T_i}{\partial s_i}+\sum_{j\sim i}\frac{\partial \mathscr T_i}{\partial s_j}>0$,
    \item [(2)] $\frac{\partial \mathscr T_i}{\partial s_j}=\frac{\partial \mathscr T_j}{\partial s_i}<0, j\sim i$,
    \item [(3)] $\frac{\partial \mathscr T_i}{\partial s_i}>0$,
    \item [(4)] $\sum_{v_j\in V(f)}\frac{\partial Area (\Omega_{f})}{\partial s_j}<0, \forall \; f\in F$.
\end{itemize}
\end{proposition}

By the Thurston algorithm, we know that the total geodesic curvature at an interior vertex $v_i$ is adjusted to $\hat{\textbf{T}}_i$, that is, for each $1\le i\le |V^{\circ}|$, $$\mathscr T_i(s_1,\cdots,s_{i-1},\Psi_i(\textbf{s}),s_{i+1}\cdots,s_{|V^{\circ}|})=\hat{\textbf{T}}_i.$$
By the implicit differentiation of $\mathscr T_i$, we obtain:
$$
\frac{\partial \Psi_i(\textbf{s})}{\partial s_j} =\begin{cases}
  -\frac{\frac{\partial \mathscr T_{i}}{\partial s_j}}{\frac{\partial \mathscr T_{i}}{\partial \Psi_i(\textbf{s})}}, & j\neq i  \\
  0,&j=i
\end{cases}
$$
By Proposition \ref{prop5.2} (3), we see that 
$$
\sum_{j\neq i}\left|\frac{\partial \mathscr T_{i}}{\partial s_j}\right|<\left|\frac{\partial \mathscr T_{i}}{\partial\Psi_i(\textbf{s})}\right|.
$$
Thus,
$$
\sum_{j=1}^{|V^{\circ}|}\left|\frac{\partial\Psi_i(\textbf{s})}{\partial s_j}\right|=\frac{\sum_{j\neq i}\left|\frac{\partial \mathscr T_{i}}{\partial s_j}\right|}{\left|\frac{\partial \mathscr T_{i}}{\partial\Psi_i(\textbf{s})}\right|}<1.
$$
Given $\textbf{s}, \textbf{s}^{\prime}\in \mathbb{R}^{|V^{\circ}|}$, we apply the mean value theorem for a vector function to obtain 
$$
\Psi_i(\textbf{s})-\Psi_i\left(\textbf{s}^{\prime}\right)=\int_0^1 \nabla \Psi_i\left(\textbf{s}^{\prime}+t\left(\textbf{s}-\textbf{s}^{\prime}\right)\right) \cdot\left(\textbf{s}-\textbf{s}^{\prime}\right) d t. 
$$
Thus,
$$
\begin{aligned}
 \left|\Psi_i(\textbf{s})-\Psi_i\left(\textbf{s}^{\prime}\right)\right| \leq &\int_0^1\left|\nabla \Psi_i\left(\textbf{s}^{\prime}+t\left(\textbf{s}-\textbf{s}^{\prime}\right)\right)\left(\textbf{s}-\textbf{s}^{\prime}\right)\right| d t \\
\leq & \int_{0}^1\left( \sum_{j=1}^{|V^{\circ}|}\left|\frac{\partial \Psi_i}{\partial s_j}\left( \textbf{s}^{\prime}+t\left(\textbf{s}-\textbf{s}^{\prime}\right)\right)\right| \cdot\left|s_j-s_j^{\prime}\right| \right)d t\\
= & \int_{0}^1\left( \sum_{j\neq i}\left|\frac{\partial \Psi_i}{\partial s_j}\left( \textbf{s}^{\prime}+t\left(\textbf{s}-\textbf{s}^{\prime}\right)\right)\right| \cdot\left|s_j-s_j^{\prime}\right| \right)d t.
\end{aligned}
$$
Let \( \|\textbf{s} - \textbf{s}'\|_{\infty} = \max_j |s_j - s_j'| \). Define $ \textbf{s}_t = \textbf{s}^{\prime}+t\left(\textbf{s}-\textbf{s}^{\prime} \right)$. Then
$$
\left|\Psi_i(\textbf{s})-\Psi_i\left(\textbf{s}^{\prime}\right)\right| \leq \sup _{0\leq t\leq 1} \left\{\sum_{j\neq i}\left|\frac{\partial \Psi_i}{\partial s_j}(\textbf{s}_t)\right|\right\}  \cdot\left\|\textbf{s}-\textbf{s}^{\prime}\right\|_{\infty}.
$$
Therefore,
\begin{equation}\label{equ-15}
\left\|\Psi(\textbf{s})-\Psi\left(\textbf{s}^{\prime}\right)\right\|_{\infty}\leq \lambda(\textbf{s},\textbf{s}^{\prime})\cdot\left\|\textbf{s}-\textbf{s}^{\prime}\right\|_{\infty},\quad \textbf{s}\neq \textbf{s}^{\prime}
\end{equation}
where
$$
\lambda(\textbf{s},\textbf{s}^{\prime})= \max_{i}\sup _{0\leq t\leq 1} \left\{\sum_{j\neq i}\left|\frac{\partial \Psi_i}{\partial s_j}(\textbf{s}_t)\right|\right\} <1.
$$

Now let $B({\textbf{s}}^{*},\rho)$ be the box consisting of the vectors $\textbf{s}$ of $\mathbb{R}^{|V^{\circ}|}$ such that $||\textbf{s}-{\textbf{s}}^{*}||_{\infty}\le \rho$, where $\rho>0$. Let
$$
\lambda :=\sup_{\textbf{s},\textbf{s}^{\prime}\in B({\textbf{s}}^{*},\rho),\textbf{s}\neq\textbf{s}^{\prime}}\lambda(\textbf{s},\textbf{s}^{\prime}).
$$

By expression (\ref{equ-15}), we have
$$
\left\|\Psi(\textbf{s})-{\textbf{s}}^{*}\right\|_{\infty}=\left\|\Psi(\textbf{s})-\Psi({\textbf{s}}^{*})\right\|_{\infty}\leq \lambda \left\|\textbf{s}-{\textbf{s}}^{*}\right\|_{\infty}<\rho.
$$

Thus, $\Psi$ maps $B({\textbf{s}}^{*},\rho)$ into itself. We show that $\Psi$ is a contraction on this box. 

Taking the partial derivative of the map $\mathscr T(s_1,\cdots,s_{i-1},\Psi_i(\textbf{s}),s_i,\cdots,s_{|V^{\circ}|})$ with respect to $\Psi_i(\textbf{s})$ and applying Proposition \ref{prop5.2} (2) and (3), we obtain
$$
\sum_{j=1}^{|V^{\circ}|}\frac{\mathscr T_j}{\partial \Psi_i(\textbf{s})}=\frac{\partial \mathscr T_i}{\partial\Psi_i(\textbf{s})}+\sum_{j\sim i}\frac{\partial \mathscr T_j}{\partial \Psi_i(\textbf{s})}$$ $$
=\frac{\partial \mathscr T_i}{\partial\Psi_i(\textbf{s})}+\sum_{j\sim i}\frac{\partial \mathscr T_{i}}{\partial \Psi_i(\textbf{s})}
=\frac{\partial \mathscr T_i}{\partial\Psi_i(\textbf{s})}+\sum_{j\neq i}\frac{\partial \mathscr T_{i}}{\partial s_j}>0.$$
Given any $\textbf{s}\in B({\textbf{s}}^{*},\rho)$, there is a positive real number $\epsilon$ such that $$\frac{\partial \mathscr T_i}{\partial\Psi_i(\textbf{s})}+\sum_{j\neq i}\frac{\partial \mathscr T_i}{\partial s_j}>\epsilon,\quad \forall\;v_i\in V^{\circ}, 1\leq i\leq |V^{\circ}|.$$
Moreover, there is $M>0$, which is determined by $B({\textbf{s}}^{*},\rho)$, such that
$$
\left|\frac{\partial \mathscr T_{i}}{\partial\Psi_i(\textbf{s})}\right|<M,\quad \forall\;v_i\in V^{\circ}.
$$
Thus, given any $\textbf{s},\textbf{s}^{\prime}\in B({\textbf{s}}^{*},\rho),\textbf{s}\neq \textbf{s}^{\prime}$, 
$$
\lambda =\sup_{\textbf{s},\textbf{s}^{\prime}\in B({\textbf{s}}^{*},\rho),\textbf{s}\neq\textbf{s}^{\prime}}\lambda(\textbf{s},\textbf{s}^{\prime})=\sup _{\textbf{s}\in B({\textbf{s}}^{*},\rho)} \left\{\max_{i}\sum_{j\neq i}\left|\frac{\partial \Psi_i}{\partial s_j}(\textbf{s})\right|\right\} $$
$$=\sup_{\textbf{s}\in B({\textbf{s}}^{*},\rho)}\left\{\max_{i}\frac{\sum_{j\neq i}\left|\frac{\partial \mathscr T_{i}}{\partial s_j}\right|}{\left|\frac{\partial \mathscr T_{i}}{\partial\Psi_i(\textbf{s})}\right|}\right\}
\leq \sup_{\textbf{s}\in B({\textbf{s}}^{*},\rho)}\left\{\max_{i}\frac{\left|\frac{\partial \mathscr T_{i}}{\partial\Psi_i(\textbf{s})}\right|-\epsilon}{\left|\frac{\partial \mathscr T_{i}}{\partial\Psi_i(\textbf{s})}\right|}\right\}
\leq 1-\frac{\epsilon}{M}
<1.$$
Thus, \(\Psi\) is a contraction mapping. It follows that the fixed point ${\textbf{s}}^{*}$ is unique and any $\textbf{s}^0\in (-\infty,+\infty)^{|V^{\circ}|}$,  $$\lim_{m\rightarrow+\infty}\Psi^m(\textbf{s}^0)={\textbf{s}}^{*}.$$
\end{proof}

\begin{remark} If the initial vector $\textbf{k}^0\in \mathcal{K}_1'$ (see Remark \ref{mathcalK-1'}), then the sequence $\{\textbf{k}^m\}_{m=0}^{\infty}$ of the Thurston algorithm stays in $\mathcal{K}_1'$ and increases monotonically to ${\textbf{k}}^*$; if $\textbf{k}^0\in \mathcal{K}_2'$ (see Remark \ref{mathcalK-2'}), then $\{\textbf{k}^m\}_{m=0}^{\infty}$ stays in $\mathcal{K}_2'$ and decreases monotonically to ${\textbf{k}}^*$. 
\end{remark}

\section{Discrete Schwarz-Pick Lemma and Maximum Modulus Principle}
In this section, we develop the so-called discrete Schwarz-Pick lemma and maximum modulus principle for generalized circle packings with the contact graph $G$ as the $1$-skeleton of a polygonal cellular decomposition $\mathcal{D}$ of a compact oriented surface $\mathcal{S}_{g, n}$ of genus $g\ge 0$ and with $n$ holes. Let $\textbf{k, k}^{\prime} : V \rightarrow(0, \infty)^{|V|}$ be two such generalized circle packings. Denote by $V^{\circ}$ and $V^{\partial}$ the sets of interior and boundary vertices, respectively, where $V=V^{\partial}\cup V^{\circ}$. 

\begin{theorem}[Schwarz-Pick I]\label{theorem-6.1} 
Assume that $\hat{\textbf{T}}\in \mathcal{T}$ and two generalized circle packings $\textbf{k}$ and $\textbf{k}^{\prime}$ associated with a graph $G$ satisfy

(1) $\textbf{k}(w) \leq\textbf{k}^{\prime}(w)$ for any $w \in \mathrm{V}^{\partial}$ and 

(2) $\textbf{T}_{\textbf{k}}(v)=\textbf{T}_{\textbf{k}^{\prime}}(v)=\hat{\textbf{T}}(v)$ for any $v\in V^{\circ}$. 

Then 

(a) $\textbf{k} \leq \textbf{k}^{\prime}$, which means ${\textbf{k}}(v) \leq {\textbf{k}}^{\prime}(v)$ at every vertex $v$ of $G$, and 

(b) if the equality holds at some vertex $v\in V^{\circ}$, then $\textbf{k}=\textbf{k}^{\prime}$. 
\end{theorem}

\begin{proof}
 Let
\begin{equation}\label{equa42}
\mathcal{K}=\{\tilde{\textbf{k}}:V\rightarrow\mathbb{R}_{+}|\begin{array}{l}\textbf{T}_{\tilde{\textbf{k}}}(v)\leq\hat{\textbf{T}}(v), \forall \;v\in V^{\circ},\text{ and}\\
\tilde{\textbf{k}}(w)\le \textbf{k}^{\prime}(w),\forall \;w\in V^{\partial}\end{array}\}.
\end{equation}
In the course of proving the existence of a solution to the discrete boundary value theorem, we have already known
\begin{equation}\label{equa43}
\textbf{k}^{\prime}=\sup\mathcal{K},
\end{equation}
which means that $\textbf{k}'(v)=\sup_{\tilde{\textbf{k}}\in \mathcal{K}} \tilde{\textbf{k}}(v)$ for any $v \in V$. 

Clearly, $\textbf{k}\in\mathcal{K}$. Thus, ${\textbf{k}}(v) \leq {\textbf{k}}^{\prime}(v)$ at every vertex $v$ of $G$.

To prove (b), we assume that $\textbf{k}(v_0)=\textbf{k}^{\prime}(v_0)$ at some $v_0\in V^{\circ}$. By Lemma \ref{lem-3.3} (b), the total geodesic curvature of the generalized circle centered at $v_0$ is strictly decreasing with the geodesic curvatures of the neighboring vertices. Since $\textbf{T}_{\textbf{k}}(v_0)=\textbf{T}_{\textbf{k}^{\prime}}(v_0)$ and $\textbf{k}(v)\leq\textbf{k}^{\prime}(v),\forall \;v\in V$, it follows that $\textbf{k}(v')=\textbf{k}^{\prime}(v')$ at any neighboring vertex $v'$ of $v$. Now we apply the same argument to conclude that $\textbf{k}(v'')=\textbf{k}^{\prime}(v'')$ at any neighboring vertex $v''$ of $v'$. Continuing this process, we know that $\textbf{k}(v)=\textbf{k}^{\prime}(v),\forall \;v\in V$; that is, $\textbf{k}=\textbf{k}^{\prime}$.

\end{proof}

\begin{theorem}[Schwarz-Pick II]\label{theorem-6.3} Let $\textbf{k}$ and $\textbf{k}^{\prime}$ be the same generalized circle packings considered in Theorem \ref{theorem-6.1}. Assume that $\text{Area}_{\textbf{k}}(\Omega_f)$ (resp. $\text{Area}_{\textbf{k}^{\prime}}(\Omega_f)$) stands for the hyperbolic area of the interstice corresponding to a face $f$ on $\tilde{\mathcal{S}}(\textbf{k})$ (resp. $\tilde{\mathcal{S}}(\textbf{k}^{\prime})$). Then 

(a) $\text{Area}_{\textbf{k}}(\Omega_f) \geq \text{Area}_{\textbf{k}^{\prime}}(\Omega_f)$ for any face $f$ of $G$;

(b) if the equality holds at some interstice $\Omega_f$, where the face $f$ has at least one interior vertex, then $\textbf{k}=\textbf{k}^{\prime}$. 
\end{theorem}

\begin{proof}
According to Lemma \ref{lem-3.3}(c), $\text{Area}(\Omega_f)$ is strictly decreasing with geodesic curvatures at the vertices of the face $f$. By Theorem \ref{theorem-6.1} (a), $\textbf{k}(v)\leq\textbf{k}^{\prime}(v),v\in V$. It follows that $\text{Area}_{\textbf{k}}(\Omega_f) \geq \text{Area}_{\textbf{k}^{\prime}}(\Omega_f)$.

To prove (b), we assume that $\text{Area}_{\textbf{k}}(\Omega_{f_0}) =\text{Area}_{\textbf{k}^{\prime}}(\Omega_{f_0})$ for some $f_0\in F$. If there is a vertex $v\in V(f_0)$ such that $\textbf{k}(v)<\textbf{k}^{\prime}(v)$, then $\text{Area}_{\textbf{k}}(\Omega_{f_0})>\text{Area}_{\textbf{k}^{\prime}}(\Omega_{f_0})$ by Lemma \ref{lem-3.3}(c) again, which is a contradiction. So we obtain $\textbf{k}(v)=\textbf{k}^{\prime}(v)$ at any $v\in V(f_0)$. Based on the graph $G$ in our consideration, one of the vertex $v$ of the face $f_0$ is an interior vertex. By Theorem \ref{theorem-6.1} (b), we conclude that $\textbf{k}=\textbf{k}^{\prime}$.
\end{proof}

\begin{theorem}[Schwarz-Pick III]\label{theorem-6.4} 
 Let $\textbf{k}$ and $\textbf{k}^{\prime}$ be the same generalized circle packings considered in Theorem \ref{theorem-6.1}. Assume that \(\ell_{\textbf{k}_i}(v, f)\) (resp. \(\ell_{\textbf{k}^{\prime}_i}(v, f)\)) stands for the length of the edge (see Figure \ref{Fig19}) of the interstice $\Omega_f$ (corresponding to the face \(f\))) that lies on the generalized circle centered at $v$ on $\tilde{\mathcal{S}}(\textbf{k})$ (resp. $\tilde{\mathcal{S}}(\textbf{k}^{\prime})$).
Then

(a) $\ell_{\textbf{k}}(v,f)\geq\ell_{\textbf{k}^{\prime}}(v,f)$ for any face $f$ and any vertex $v$ of $f$;

(b) if the equality holds for a pair $(v,f)$, where $f$ is a face and $v\in V^{\circ}$ is a vertex of $f$, then $\textbf{k}=\textbf{k}^{\prime}$. 
\end{theorem}

\begin{proof}

\begin{figure}[ht]
{
\centering
\includegraphics[height=4cm]{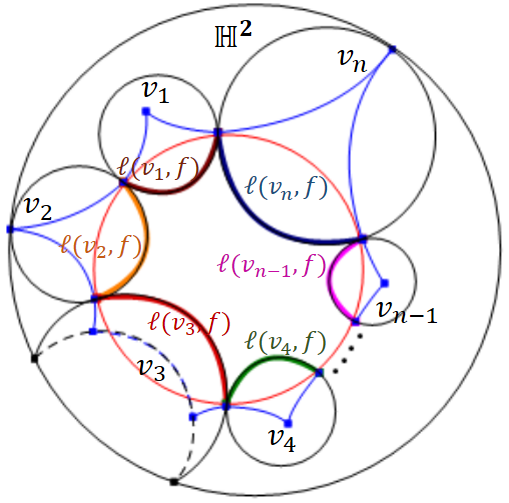} 
\caption{The demonstration of $\ell(v,f)$.}
\label{Fig19}

}
\end{figure}

By the proof of Theorem 1.6 in \cite{2024arXiv241106274H}, we can see that $$\frac{\partial\ell(v,f)}{\partial k_P}=-\frac{2}{k^2+k_P^2-1}<0$$  and 
\begin{equation}
\frac{\partial \ell(v,f)}{\partial k}=\begin{cases}
 -\frac{2k}{(1-k^2)^{\frac{3}{2}}}(\frac{k_P\sqrt{1-k^2}}{k_P^2+k^2-1}-arctanh \frac{\sqrt{1-k^2}}{k_P})<0,\quad 0<k<1\\
 -\frac{4}{3}\frac{1}{k_P3}<0,\quad k=1\\
 -\frac{2k}{(k^2-1)^{\frac{3}{2}}}(\frac{k_P\sqrt{k^2-1}}{k_P^2+k^2-1}-arctan \frac{\sqrt{k^2-1}}{k_P})<0,\quad k>1.
\end{cases}
\end{equation}
Furthermore, $\frac{\partial \ell(v,f)}{\partial k}$ is a continuous function. Therefore, $\ell(v,f)$ decreases as $k$ or $k_P$ increases. By Lemma \ref{lem-3.2}, we know that the geodesic curvature of the dual circle $k_P$ increases as $k_j$ increases, where $k_j\in V(P)$. Since $\textbf{k}\leq\textbf{k}^{\prime}$, it follows that $\ell_{\textbf{k}}(v,f)\geq\ell_{\textbf{k}^{\prime}}(v,f)$ for any $f\in F$ and any $v\in V(f)$.

Suppose that there exist $f_0\in F$ and $v_0\in V(f_0)$ such that $\ell_{\textbf{k}}(v_0,f_0)=\ell_{\textbf{k}^{\prime}}(v_0,f_0)$. We may assume that $v_0$ is an interior vertex. Because $\textbf{k}\leq\textbf{k}^{\prime}$, $k_{P_0}\leq k^{\prime}_{P_0}$. Using the strict decreasing monotonicity of the function $\ell_{\textbf{k}}(v_0,f_0)$ on $k_{v_0}$ and $k_{f_0}$,  we know $k_{v_0}< k^{\prime}_{v_0}$ and $k_{P_0}\leq k^{\prime}_{P_0}$ imply $\ell_{\textbf{k}}(v_0,f_0)>\ell_{\textbf{k}^{\prime}}(v_0,f_0)$. This is a contradiction. Thus, $k_{v_0}=k^{\prime}_{v_0}$. By Theorem \ref{theorem-6.1}, we conclude that $\textbf{k}=\textbf{k}^{\prime}$.
\end{proof}

\begin{theorem}[Schwarz-Pick IV]\label{theorem-6.2}
Let $\textbf{k}$ and $\textbf{k}^{\prime}$ be the same generalized circle packings considered in Theorem \ref{theorem-6.1}. Assume that $\rho_k$ (resp. $\rho_{\textbf{k}^{\prime}}$) is the distance function on the metric surface $\tilde{\mathcal{S}}(\textbf{k})$ (resp. $\tilde{\mathcal{S}}(\textbf{k}^{\prime})$). Then 

(a) $\rho_\textbf{k}(v_1, v_2) \geq \rho_{\textbf{k}^{\prime}}(v_1, v_2)$, $\forall \;v_1,v_2\in \tilde{V}$, where $\tilde{V}=\{v\in V:\;\textbf{k}(v)>1\}$;

(b) if the equality holds at two vertices $u$ and $v$ with at least one of them in $V^{o}$, then $\textbf{k}=\textbf{k}^{\prime}$.
\end{theorem}

\begin{proof}
 We first prove (a). By Theorem \ref{theorem-6.1}, $\textbf{k}(v)\le\textbf{k}'(v)$ for any $v\in V$. Using the transitivity of inequalities, it is sufficient to develop the monotonicity of distance for two generalized hyperbolic conic circle packings determined by $\textbf{k}$ and $\textbf{k}'$ for which the values of $\textbf{k}$ and $\textbf{k}'$ differ only at one vertex $v\in V$. So, we assume that $\textbf{k}(v)<\textbf{k}'(v)$ and $\textbf{k}(v')=\textbf{k}'(v')$ for any $v'\in V\setminus\{v\}$.
 
 Let $v_1,v_2\in \tilde{V}$, where $\tilde{V}=\{v\in V:\;\textbf{k}(v)>1\}$. Let $\gamma$ be a piecewise geodesic curve on the surface $\tilde{S}(\textbf{k})$ that connects $v_1$ and $v_2$ and realizes the hyperbolic distance between them. Then $\gamma$ passes through a chain of polygons $\{P_1,P_2,\cdots,P_s\}$ that correspond to a chain of faces $\{f_1,f_2,\cdots,f_s\}$ in $G$. They divide $\gamma$ into finitely many geodesic segments $\{\gamma _j\} _{j=1}^{m}$, where the end points of $\gamma _j$ lie on the boundary of $P _j$. 
 
 The curve $\gamma$ does not pass through the center of any hypercycle or the center of any horocycle on the circle packing determined by $\textbf{k}$. Suppose this is not true, We consider two cases. 
 If $\gamma$ passes through the center of some horocircle, then the length of $\ell(\gamma)$ is $\infty$ and hence the distance between $v_1$ and $v_2$ on $\tilde{\mathcal{S}}(\textbf{k})$ is equal to $\infty $. If $\gamma$ intersects with the center of some hypercycle $C$, there are two situations to consider. One situation is that $\gamma_0$ is decomposed into a union of  three geodesic segments $\gamma_{01},\gamma_{02}$ and $\gamma_{03}$ and the other is that $\gamma_0$ intersects the center of some hypercircle as a point (see Figure \ref{Fig20}). By doing a perturbation or replacing some pieces of $\gamma_0$ intersecting the center of the hypercircle by other geodesic segments to form another piecewise geodesic segment with the same end points as $\gamma_0$ and a shorter length than $\gamma_0$. Therefore, $\gamma$ does not realize the hyperbolic distance between $v_1$ and $v_2$ on the surface $\tilde{\mathcal{S}}(\textbf{k})$.

 \begin{figure}[ht]
{
\centering
\includegraphics[height=4cm]{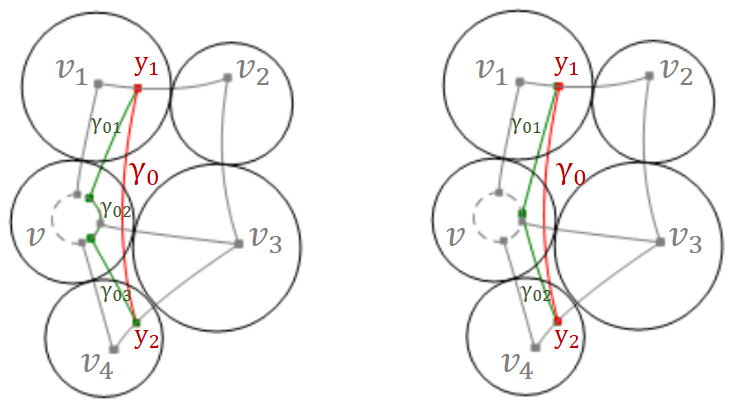} 
\caption{Illustrations when $\gamma_0$ intersects the center of a hypercycle.}
\label{Fig20}

}
\end{figure}
 
We divide $\{\gamma_j\}_{j=1}^{m}$ into two groups. A geodesic segment $\gamma_j$ is said to be ``bad" if it passes through the star of $v$ while it is said to be ``good" if it does not pass through the star. Next, we construct the corresponding geodesic segments $\gamma^{\prime}_j$ on the generalized hyperbolic conic circle packing determined by $\textbf{k}^{\prime}$. Keep the end points of $\gamma_j$ still and connect them by the geodesic segment on $\textbf{k}^{\prime}$. Then, by joining $\{\gamma^{\prime}_j\}_{j=1}^{m}$ in the same order as $\{\gamma_j\}_{j=1}^{m}$ are joined to form $\gamma$, we obtain a piecewise geodesic curve $\gamma^{\prime}$ on $\textbf{k}^{\prime}$. Clearly, the distance between $v_1$ and $v_2$ on the surface determined by $\textbf{k}'$ satisfies $\rho_{\textbf{k}^{\prime}}(v_1,v_2)\leq\ell^{\prime}(\gamma ')$. It is obvious that the length of each ``good" segment satisfies $\ell(\gamma_j)=\ell^{\prime}(\gamma_j')$. Then we only need to consider the ``bad" segments. There are three cases to be considered.

\textbf{Case one:} Suppose that the geodesic curvature of a circle centered at one end point of $\gamma$ is changing; that is, $v_1=v$. Assume that $\gamma_1$ connects $v$ to a point $p$ on the edge of $P_1$ opposite to $v$, where $k_{v}\in (1,\infty)$. Suppose that the polygon $P$ has $n$ vertices $v,u_1,\cdots,u_{n-1}$, where $k_{u_i}\in (0,+\infty),i=1,\cdots,n-1$. Let $r_v,r_{u_1},\cdots,r_{u_{n-1}}$ denote the radii of the generalized circles centered at $v,u_1,\cdots,u_{n-1}$ respectively. Let $\ell_p$ be the length of $\gamma_1$. Next, we show that $\ell_p$ is deceasing with $k_v$.

The polygon $P_1$ is divided by $\gamma_1$ into two polygons $D_1$ and $D_2$. Let $t$ be the distance between $p$ and its neighboring vertex in $D_1$. Since only the geodesic curvature of the circle centered at the vertex $v$ is changing, the length $\ell(u_i,u_{i+1}),i=1,\cdots,n-2,$ and $t$ remain unchanged. 

Denote by $x_1$ and $x_2$ the tangent points of the circle centered at $v$ with the generalized circles centered at $u_1$ and $u_{n-1}$, respectively. Passing through $x_1$ and $x_2$, there are two geodesics perpendicular to the edges $(v,u_1)$ and $(v,u_{n-1})$, respectively. On the induced hyperbolic polygon of the corresponding circle configuration, these two geodesics intersect at the center of the dual circle $C_{P_1}$, intersecting the boundary of $P_1$ at $q_1$ and $q_2$, respectively.

\begin{figure}[ht]
{
\centering
\includegraphics[height=5cm]{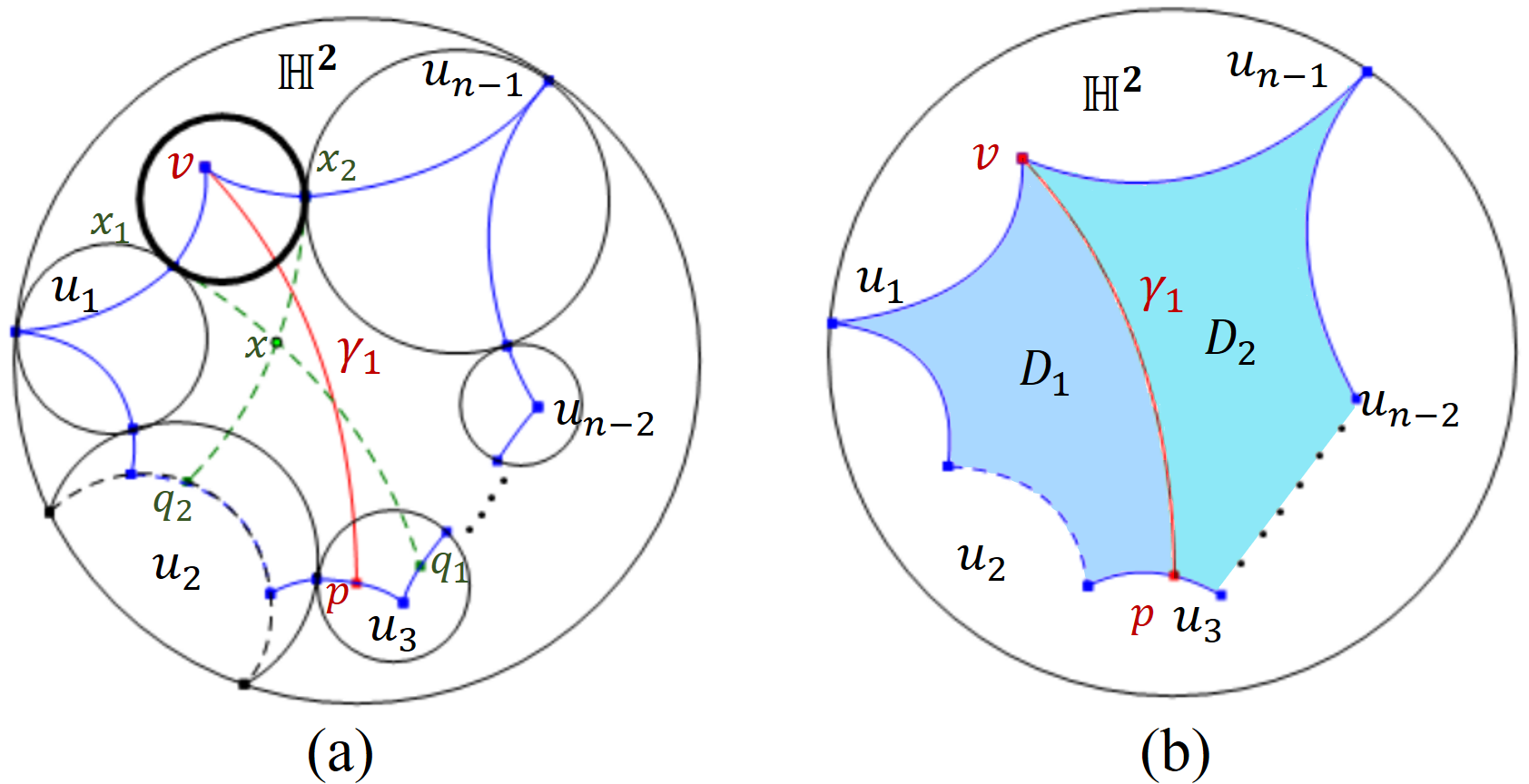} 
\caption{Illustration of the case one.}
\label{Fig21}
}
\end{figure}
Define $D$ as follows:
$$
D=\begin{cases}
    D_1 \text{ if $p$ is on the left of $q_1$}\\
    D_2 \text{ if $p$ is on the right of $q_1$ or exactly on $q_1$}.
    
\end{cases}\\
$$
Consider the hyperbolic polygon $D$. Under the assumption that $\boldsymbol{k}^{\prime}(v)>\boldsymbol{k}(v)$ and 
$\boldsymbol{k}^{\prime}(v')=\boldsymbol{k}(v)$ for all vertices $v'\in V\setminus \{v\}$, we know that $r^{\prime}(v)<r(v)$ and $r^{\prime}(v')=r(v')$ for all $v'\in V\setminus \{v\}$. Furthermore, the hyperbolic triangle $\Delta vx_1p$ (respectively, $\Delta vx_2p$) is obtuse when $D=D_1$ (respectively, $D=D_2$). Consequently, we obtain the length inequality $\ell'(\gamma^{\prime}_1)< \ell(\gamma_1)$.

\textbf{Case two:} Suppose that the geodesic curvature $k_v\in (0,+\infty)$ of the generalized circle centered at a vertex $v\in \cup_{i=1}^{n-1} V(f_i)$ is changing, where $V(f_i)$ denotes the collection of the vertices of the face $f_i\in F$. The segment $\gamma_j$ lies in the star of the vertex $v$ as follows. 

Let $u_1,\cdots,u_n$ be the boundary vertices of the generalized star of $v$. Let $x$ and $y$ be the boundary points of $\gamma_j$. If $\gamma_j$ passes through the center of a hyperbolic conical circle, then we cut $\gamma_j$ into two geodesic arcs $\gamma_{j_1}$ and $\gamma_{j_2}$. We can handle this situation as the same as we have done in Case one. By the discussions in the third paragraph of this proof (partially illustrated on Figure \ref{Fig20}),  we know that $\gamma_j$ does not pass through the center of any horocycle or the center of any hypercircle. In this case, $\gamma_j$ divides the star of the vertex $v$ into two components $D_1$ and $D_2$, and at least one of them is a hyperbolic polygon (see $D_1$ on Figure \ref{Fig22} (b)).

We divide $\gamma_j$ into geodesic segments on the induced hyperbolic polygons that intersect it. For example, on Figure \ref{Fig22}(a) $\gamma_j$ is divided into two pieces $\gamma_{j1}$ and $\gamma_{j2}$. 
Now we keep the distances from $y_1$ to $u_4$, from $y_3$ to $u_1$, and from $y_2$ to $u_9$ the same on $\tilde{\mathcal{S}}(\textbf{k})$ and $\tilde{\mathcal{S}}(\textbf{k}^{\prime})$. By Lemma \ref{lem-3.8} and Remark \ref{outside the induced hyperbolic polygon}, we obtain $$
\ell '(\gamma_{js}^{\prime})<\ell(\gamma_{js}).
$$

Since $\gamma_j^{\prime}$ constitutes a geodesic connecting $y_1$ and $y_2$ through the generalized star, its length does not exceed the total length of the concatenation of the sub-arcs $\gamma_{js}$. Consequently,
\[
\ell '(\gamma_j^{\prime})\leq\sum_s \ell '(\gamma^{\prime}_{js}) <\sum_s \ell(\gamma_{js}) = \ell(\gamma_j).
\]

\begin{figure}[ht]
{
\centering
\includegraphics[height=5cm]{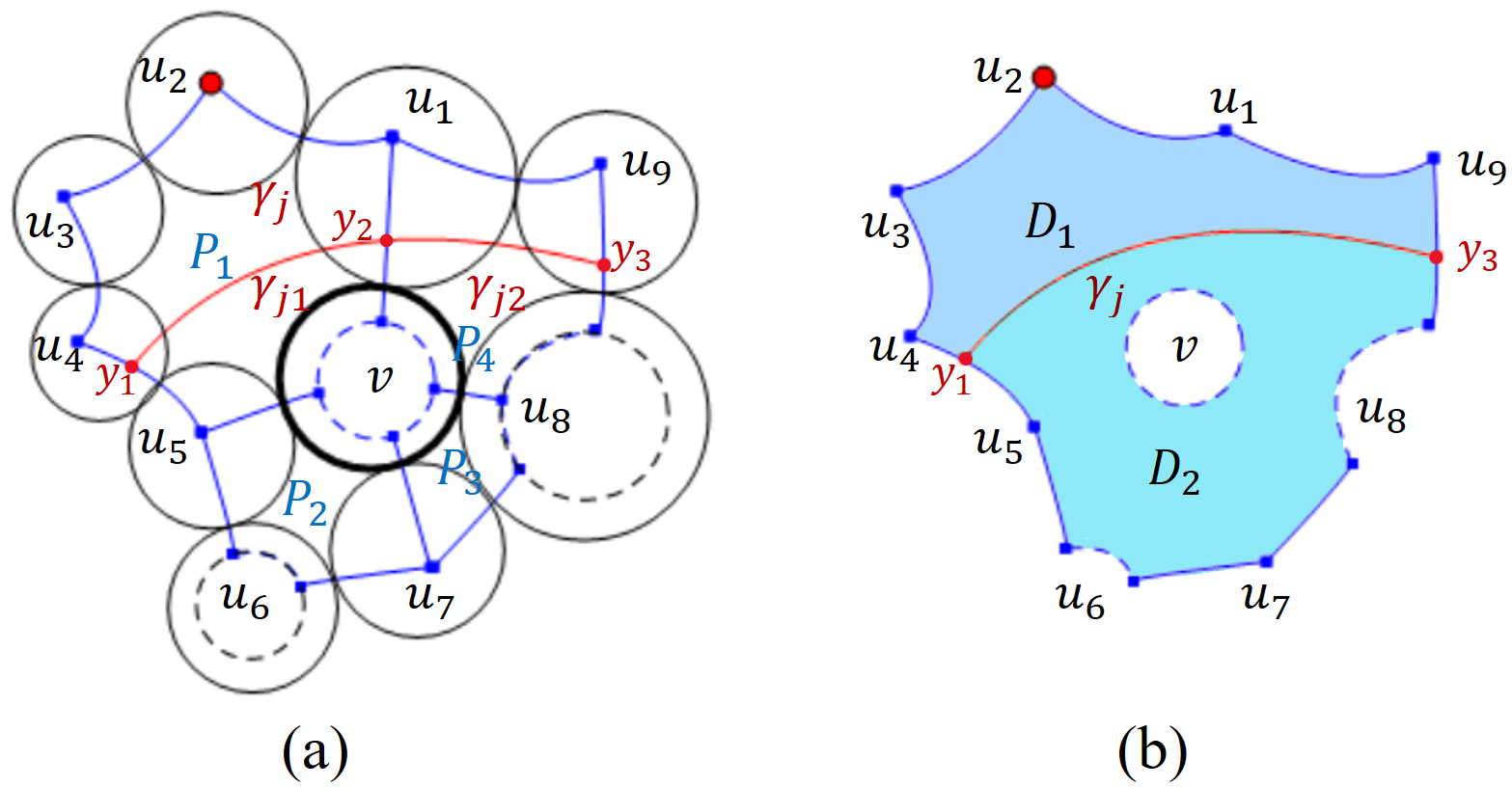} 
\caption{Illustration of the case two.}
\label{Fig22}

}
\end{figure}

\textbf{Case three:}
The segment $\gamma_j$ does not lie in the polygonal region $\Omega$ completely; that is, $\gamma_j$ has $2t$ $(t\geq2)$ intersections with the star of the vertex $v$. Then we can cut $\gamma_j$ into geodesic segments $\{\gamma_{j1},\cdots,\gamma_{jt}\}$ based on these intersections. 

\begin{figure}[ht]
{
\centering
\includegraphics[height=5cm]{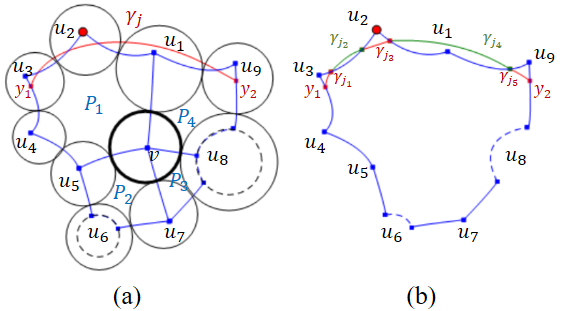} 
\caption{IIlustration of the case three.}
\label{Fig23}

}
\end{figure}

Let $\mathcal{X}$ be the set of the geodesic segments that lie in the star of the vertex $v$ and $\mathcal{Y}$ be the set of the geodesic segments that stay out of the star of the vertex $v$. Each geodesic segment in $\mathcal{Y}$ has the same length on $\tilde{\mathcal{S}}(\textbf{k})$ and $\tilde{\mathcal{S}}(\textbf{k}^{\prime})$. From the previous discussion, we see that for any $\gamma_{0}^{\prime}\in \mathcal{X}$, there exists a geodesic segment $\gamma_{0}^{\prime}$ in $\textbf{k}^{\prime}$ such that $\ell '(\gamma_{0}^{\prime})<\ell(\gamma_{0})$.

Therefore, by the comparisons obtained in the above three cases, we can see that when $\mathbf{k}(v)<\mathbf{k}'(v)$ and $\mathbf{k}(v')=\mathbf{k}'(v')$ at any other vertex $v'$,
$\rho_{\textbf{k}^{\prime}}(v_1, v_2)<\rho_\textbf{k}(v_1, v_2)$ if the shortest piecewise geodesic segment connecting $v_1$ and $v_2$ on $\tilde{\mathcal{S}}(\textbf{k})$ intersecting the star of the vertex $v$; otherwise, $\rho_{\textbf{k}^{\prime}}(v_1, v_2)=\rho_\textbf{k}(v_1, v_2)$.

Assume that $\textbf{k}(w)<\textbf{k}'(w)$ at some $w\in V^\partial$. By Theorem \ref{theorem-6.1}, $\textbf{k}(v)<\textbf{k}'(v)$ at any $\in V^\circ$. It follows that for any two $v_1,v_2\in \tilde{V}$ and $v_1\in V^\circ$, the shortest piecewise geodesic segment connecting $v_1$ and $v_2$ much intersect the star of some vertex $v$ at which $\mathbf{k}(v)<\mathbf{k}'(v)$. Thus, $\rho_{\textbf{k}^{\prime}}(v_1,v_2)<\rho_{\textbf{k}}(v_1,v_2)$. Therefore, if $\rho_{\textbf{k}}(v_1,v_2)=\rho_{\textbf{k}^{\prime}}(v_1,v_2)$ for some $v_1,v_2\in \tilde{V}$ and $v_1\in V^\circ$, then $\textbf{k}(w)=\textbf{k}'(w)$ for any $w\in V^\partial$.  By the uniqueness of the solution to our discrete boundary value problem, we know
$\textbf{k}=\textbf{k}^{\prime}$. This proves (b).
\end{proof}



\begin{theorem}[Maximum Modulus Principle]\label{theorem-6.5}
 Let $\textbf{k}$ and $\textbf{k}^{\prime}$ be the same generalized circle packings considered in Theorem \ref{theorem-6.1}. Define
\begin{equation}\label{equa44}
\phi: V\rightarrow \mathbb{R}_+: v\mapsto \phi(v)=\frac{\textbf{k}^{\prime}_v}{\textbf{k}_v}.
\end{equation} Let $M$ be the maximal value of $\phi$. Then

(a) $M\ge 1$;

(b) if $M>1$, then it is only attained at some boundary vertex $w_j\in V^{\partial}$;

(c) if $M$ is attained at some interior vertex, then $M=1$ and hence $\textbf{k}=\textbf{k}^{\prime}$.
\end{theorem}
\begin{proof} By Theorem \ref{theorem-6.1}, we only need to prove (b).
Suppose that $M>1$ and there exists a vertex $v_0 \in V^{\circ}$ such that $\frac{\tilde{k}_{v_0}}{k_{v_0}}=M$. Then 
$\frac{\tilde{k}_v}{k_v}\le M$ for any $v \in V$.

Denote the neighboring vertices of $v_0$ by $w_1, \cdots ,w_n$ and let $s_v=\ln k_v$. Then $\tilde{s}_{v_0}-s_{v_0}=\ln M>0$ and $\tilde{s}_{v_0}-s_{v_0} \geq \tilde{s}_{w_j}-s_{w_j}\ge 0$ for each $j \in\{1, \cdots n\}$. 

Let $s_0(t)=(1-t) s_{v_0}+t \tilde{s}_{v_0}$ and $s_j(t)=(1-t) s_{w_j}+t \tilde{s}_{w_j}$ each $j \in\{1, \cdots n\}$. Define 
$$
s(t)=\left(s_0(t), s_1(t), \cdots s_n(t)\right) \text{ and }
T_0(t)=T_0(s(t)).$$
Then $T_0(0)=T_0(1)$. By Rolle's theorem, there exists $\xi \in(0,1)$ such that $T_0^{\prime}(\xi)=0$; that is,
$$
I=\left.\sum_{j=1}^n \frac{\partial T_0}{\partial s_j} \cdot \frac{d s_j}{d t}\right|_{t=\xi}+\left.\frac{\partial T_0}{\partial s_0} \cdot \frac{d s_0}{d t}\right|_{t=\xi}=0.$$
On the other hand, applying Proposition \ref{prop5.2} we obtain
$$
I=\left.\sum_{j=1}^n \frac{\partial T_0}{\partial s_j}\right|_{s_j=s_j(\xi)} \cdot (\tilde{s}_{w_j}-s_{w_j})+\left.\frac{\partial T_0}{\partial s_0} \right|_{s_0=s_0(\xi)}\cdot(\tilde{s}_{v_0}-s_{v_0})$$
$$
\geq[\left.\sum_{j=0}^n \frac{\partial T_0}{\partial s_j}\right|_{s_j=s_j(\xi)}] \cdot (\tilde{s}_{v_0}-s_{v_0})>0 .
$$
This is a contradiction. Therefore, if $M>1$, then $M$ cannot be reached at any interior vertex. 
\end{proof}

\section{Appendix}\label{appendix}

\subsection{Proof of Formula (\ref{equ-6})}\label{sec-7.1}

In this subsection, we derive the formula (\ref{equ-6}). 
At first, by the proof of Lemma 2.1 in \cite{MR4842734}, we obtain 
$$\cot \frac{\theta_i}{2}=\begin{cases} \tanh r_i\sinh r_P=\frac{k_i}{\sqrt{k_P^2-1}},\quad 0<k_i<1\\
\sinh r_P=\frac{1}{\sqrt{k_P^2-1}},\quad k_i=1\\
\operatorname{coth} r_i \sinh r_P=\frac{k_i}{\sqrt{k_P^2-1}},\quad k_i>1.
\end{cases}
$$
For a generalized circle configuration $C_i$, $i=1, 2, \cdots, n$, with geodesic curvatures $k_i \;(>0)$, $i=1, 2, \cdots, n$,  respectively,
$$
\left\{\begin{array}{l}
\cot \frac{\theta_1}{2}=\frac{k_1}{\sqrt{k_P^2-1}} \\
\cot \frac{\theta_2}{2}=\frac{k_2}{\sqrt{k_P^2-1}} \\
\vdots\\
\cot \frac{\theta_n}{2}=\frac{k_n}{\sqrt{k_P^2-1}} \\
\theta_1+\theta_2+\cdots+\theta_n=2\pi .
\end{array}\right.
$$
Then
$$
\begin{aligned}
\cot\frac{\theta_n}{2}&=\cot(\pi-\frac{\theta_1}{2}-\cdots-\frac{\theta_{n-1}}{2})\\&=\frac{\cos \left(\pi-\frac{\theta_1}{2}-\cdots-\frac{\theta_{n-1}}{2}\right)}{\sin \left(\pi-\frac{\theta_1}{2}-\cdots-\frac{\theta_{n-1}}{2}\right)}\\
& =\frac{-\sum_{|U| \text{ even}}(-1)^{\frac{|U|}{2}}(\prod _{i\in U}\sin\frac{\theta_i}{2})(\prod_{i\notin U}\cos\frac{\theta_i}{2})}{\sum_{|V| \text{ odd}}(-1)^{\frac{|V|-1}{2}}(\prod_{i\in V}\sin\frac{\theta_i}{2})(\prod_{i\notin V}\cos\frac{\theta_i}{2})}\\
&=\frac{-\sum_{|U| \text{ even}}(-1)^{\frac{|U|}{2}}\prod_{i\notin U}\cot\frac{\theta_i}{2}}{\sum_{|V| \text{ odd}}(-1)^{\frac{|V|-1}{2}}\prod_{i\notin V}\cot\frac{\theta_i}{2}},
\end{aligned}
$$
where $U,V\subset\{1,\cdots,n-1\}$, $|U|$ is even, and $|V|$ is odd. 
It follows that
$$
\frac{-\sum_{m=0}^{[(n-1)/2]}(-1)^m\cdot\frac{\sum_{1\leq i_1<i_2<\cdots<i_{n-1-2m}\leq n-1}k_{i_1}k_{i_2}\cdot k_{i_{n-1-2m}}}{(\sqrt{k_P^2-1})^{n-1-2m}}}{\sum_{m=0}^{[(n-2)/2]}(-1)^m\cdot\frac{\sum_{1\leq i_1<i_2<\cdots<i_{n-2-2m}\leq n-1}k_{i_1}k_{i_2}\cdot k_{i_{n-2-2m}}}{(\sqrt{k_P^2-1})^{n-2-2m}}}=\frac{k_n}{\sqrt{k_P^2-1}},
$$
which can be rewritten as
$$\sum_{m=0}^{[\frac{n-1}2{}]}(-1)^mS_{n-1-2m}(k_P^{2}-1)^{m}=0,$$
where 
$$S_0=1 \text{ and }S_{t}=\sum_{1\leq i_1<\cdots <i_{t}\leq n} k_{i_1}\cdots k_{i_{t}} \text{ for } t=1,\cdots,n.$$

So, when $n=3$, the equation is written as $S_2(k_P^2-1)^0-S_0(k_P^2-1)=0$; that is,
$$
k_P^2=k_1 k_2+k_1 k_3+k_2 k_3+1,
$$
which is obtained in \cite{MR4878836}.

When $n=4$, the equation is written as $S_3(k_P^2-1)^0+S_1(k_P^2-1)=0$; that is
$$
k_P^2=\frac{S_3}{S_1}+1=\frac{k_1 k_2 k_3+k_1 k_2 k_4+k_1 k_3 k_4+k_2 k_3 k_4}{k_1+k_2+k_3+k_4}+1.
$$
Let $w=k_P^2-1$ and 
$$
f_n(w)=\sum_{m=0}^{[\frac{n-1}{2}]}(-1)^mS_{n-1-2m}w^{m}.
$$

Now we show that for $\forall \; n\geq 3$ and $\forall \; k_1,k_2,\cdots,k_n\in(0,+\infty)$, the polynomial $f_n(w)$ has a real positive root. From the above discussions, this is true for $n=3, 4$. It remains to prove that it is true for $n\geq 5$. We divide the proof into the following three cases. 

\textbf{(i)} Assume that $n=4j+3$ and $n=4j+4$, $j\geq 1$. Then
$$
f_n(w)=\sum_{m=0}^{2j+1}(-1)^mS_{n-1-2m}w^m.
$$
Clearly, $f_n(0)=S_{n-1}>0$. Since the leading coefficient of $f_n(w)$ is 
$$(-1)^{2j+1}S_{n-4j-3}=\begin{cases}
    -S_0<0,\quad n=4j+3\\
    -S_1<0,\quad n=4j+4 ,
\end{cases}$$
it follows that 
$$
\lim_{w\rightarrow -\infty}f_n(w)=+\infty.
$$
By the intermediate value theorem, the continuous function $f_n(w)$ has a positive real root.

\textbf{(ii)} Assume that $n=4j+1$, $j\geq 1$. Then
$$
\begin{aligned}
f_n(w)=&\sum_{m=0}^{2j}(-1)^mS_{4j-2m}w^m\\
=&S_{4j}-S_{4j-2}w+S_{4j-4}w^2-S_{4j-6}w^3+\cdots\\
&+S_4w^{2j-2}-S_2w^{2j-1}+S_0w^{2j}.
\end{aligned}
$$
Clearly, $f_n(0)=S_{n-1}>0$. We show $f_n(\frac{S_2}{2S_0})<0$. Clearly,
$$
\begin{aligned}
f_n(\frac{S_2}{2S_0})=&S_{4j}-\frac{S_{4j-2}S_2}{2S_0}+\frac{S_{4j-4}S_2^2}{(2S_0)^2}-\frac{S_{4j-6}S_2^3}{(2S_0)^3}+\cdots\\
&+\frac{S_{4}S_2^{2j-2}}{(2S_0)^{2j-2}}-\frac{S_{2}S_2^{2j-1}}{(2S_0)^{2j-1}}+\frac{S_{0}S_2^{2j}}{(2S_0)^{2j}}\\
=&-\frac{1}{(2S_0)^{2j}}[(2S_0)^{2j-1}(S_{4j-2}S_2-2S_{4j})\\
&+(2S_0)^{2j-3}S_2^2(S_{4j-6}S_2-2S_{4j-4})+\cdots\\
&+(2S_0)^{2}S_2^{2j-4}(S_{6}S_2-2S_{8})\\
&+(2S_0)S_2^{2j-2}(S_2^{2}-4S_{4})].
\end{aligned}
$$
Since 
\begin{equation}\label{equ-7}
\begin{aligned}
S_tS_2-4S_{t+2}=&(\sum_{1\leq i_1<\cdots <i_{t}\leq n} k_{i_1}\cdots k_{i_{t}})\cdot (\sum_{1\leq i_1\leq i_2\leq n}k_{i_1}k_{i_2})\\
&-4\sum_{1\leq i_1<\cdots <i_{t+2}\leq n} k_{i_1}\cdots k_{i_{t+2}}\\
>&0 \text{ when } t\geq 2,
\end{aligned}
\end{equation}
we obtain $f_n(\frac{S_2}{2S_0})<0$. By the intermediate value theorem, there exists $w_0\in(0,\frac{S_2}{2})$ such that $f_n(w_0)=0$.

\textbf{(iii)} Assume that $n=4j+2$, $j\geq 1$. Then
$$
\begin{aligned}
f_n(w)=&\sum_{m=0}^{2j}(-1)^mS_{4j+1-2m}w^m\\
=&S_{4j+1}-S_{4j-1}w+S_{4j-3}w^2-S_{4j-5}w^3+\cdots\\
&+S_5w^{2j-2}-S_3w^{2j-1}+S_1w^{2j}.
\end{aligned}
$$
Clearly, $f_n(0)=S_{n-1}>0$. We show that $f_n(\frac{S_3}{2S_1})<0$. Clearly,

$$
\begin{aligned}
f_n(\frac{S_3}{2S_1})=&S_{4j+1}-\frac{S_{4j-1}S_3}{2S_1}+\frac{S_{4j-3}S_3^2}{(2S_1)^2}-\frac{S_{4j-5}S_3^3}{(2S_1)^3}+\cdots\\
&+\frac{S_{5}S_3^{2j-2}}{(2S_1)^{2j-2}}-\frac{S_{3}S_3^{2j-1}}{(2S_1)^{2j-1}}+\frac{S_{1}S_3^{2j}}{(2S_1)^{2j}}\\
=&-\frac{1}{(2S_1)^{2j}}[(2S_1)^{2j-1}(S_{4j-1}S_3-2S_1S_{4j+1})\\
&+(2S_1)^{2j-3}S_3^2(S_{4j-5}S_3-2S_1S_{4j-3})+\cdots\\
&+(2S_1)^{2}S_3^{2j-4}(S_{7}S_3-2S_1S_{9})\\
&+(2S_1)S_3^{2j-2}(S_3^{2}-4S_1S_{5})].
\end{aligned}
$$
Using the inequality (\ref{equ-7}), we obtain $f_n(\frac{S_3}{2S_1})<0$. Again, by the intermediate value theorem, there exists $w_0\in(0,\frac{S_2}{2})$ such that $f_n(w_0)=0$.

\subsection{The edge-angle relationship for a hyperbolic polygon induced by a generalized circle configuration}\label{sec-7.2}

\begin{figure}[ht]
{
\centering
\includegraphics[height=4cm]{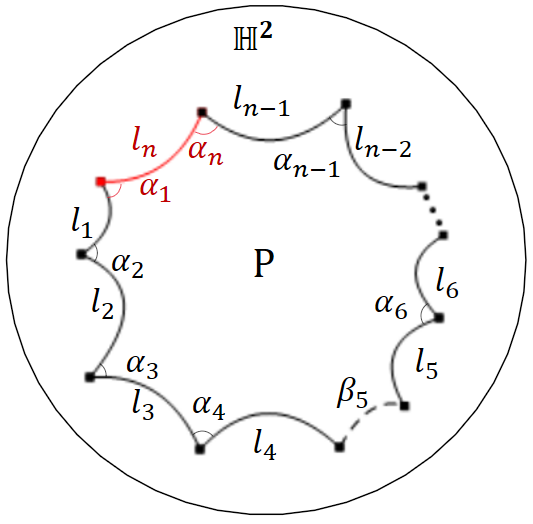} 
\caption{An example of generalized hyperbolic polygon.}
\label{Fig24}

}
\end{figure}

Let

\begin{equation}\label{equa-25}
\mathcal{A}_{l_i}=\mathcal{H}_{l_i}=\left(\begin{array}{ccc}
\cosh l_i & 0 & \sinh l_i \\
0 & 1 & 0 \\
\sinh l_i & 0 & \cosh l_i
\end{array}\right),
\end{equation}

and 

\begin{equation}\label{equa-26}
\mathcal{A}_{\alpha_i}=
\mathcal{R}_{\pi-\alpha_i},\quad \mathcal{A}_{\beta_j}=\mathcal{R}_{\frac{\pi}{2}}\mathcal{H}_{\beta_j}\mathcal{R}_{\frac{\pi}{2}}.
\end{equation}

Here,

\begin{equation}\label{equa-27}
\mathcal{R}_{\pi-\alpha}=\left(\begin{array}{ccc}
\cos(\pi- \alpha) & -\sin(\pi-\alpha) & 0 \\
\sin (\pi-\alpha) & \cos(\pi- \alpha) & 0 \\
0 & 0 & 1
\end{array}\right).
\end{equation}

For the polygon in Figure \ref{Fig24}, we put $v_1$ to $(0,0,1)$ at Minkowski model by Möbius transformation such that the unit vector at the point $v_1$ tangent to the edge $v_1v_2$ is $\textbf{n}=(1,0,0)$. Transformation $\mathcal{A}_{l_1}$ means a translation along with the edge $v_1 v_2$ for length $l_1$. Consequently, it moves the unit vector $\textbf{n}$ at the point $v_1$ to the vector $\textbf{n}\cdot\mathcal{A}_{l_1}$ at the point $v_2$. Transformation $\mathcal{A}_{\alpha_2}$ means a counterclockwise rotation centered on $v_2$ for angle $\pi-\alpha_2$. Consequently, the vector $\textbf{n}\cdot\mathcal{A}_{l_1}\mathcal{A}_{\alpha_2}$ is tangent to edge $v_2v_3$. And so on, we are back to the starting point $v_1$ when we are done with the last rotation $\mathcal{A}_{\alpha_1}$, that is, $\textbf{n}\cdot\mathcal{A}_{l_1}\mathcal{A}_{\alpha_2}\cdots \mathcal{A}_{l_n}\mathcal{A}_{\alpha_1}=\textbf{n}$. In fact, the composite map of this is an identity. Therefore, we obtain
\begin{equation}\label{equa-28}
\mathcal{A}_{l_1}\mathcal{A}_{\alpha_2}\mathcal{A}_{l_2}\mathcal{A}_{\alpha_3}\cdots \mathcal{A}_{l_{n-1}}\mathcal{A}_{\alpha_n}\mathcal{A}_{l_{n}}\mathcal{A}_{\alpha_1}=id,
\end{equation}
that is,
$$
\begin{aligned}
&\left(\begin{array}{ccc}
\cosh l_1 & 0 & \sinh l_1 \\
0 & 1 & 0 \\
\sinh l_1 & 0 & \cosh l_1
\end{array}\right)
\left(\begin{array}{ccc}
-\cos \alpha_{2} & -\sin \alpha_{2} & 0 \\
\sin \alpha_{2} & -\cos \alpha_{2} & 0 \\
0 & 0 & 1
\end{array}\right)\\
&\cdots
\left(\begin{array}{ccc}
\cosh l_n & 0 & \sinh l_n \\
0 & 1 & 0 \\
\sinh l_n & 0 & \cosh l_n
\end{array}\right)
\left(\begin{array}{ccc}
-\cos \alpha_{1} & -\sin \alpha_{1} & 0 \\
\sin \alpha_{1} & -\cos \alpha_{1} & 0 \\
0 & 0 & 1
\end{array}\right)=i d.
\end{aligned}
$$
This is a system of equations consisting of nine ternary linear equations, from which the three unknown variables $l_{n},\alpha_{1}$ and $\alpha_{n}$ are determined.

\bigskip
\textbf{Acknowledgments:}
Guangming Hu is supported by the NSF of China (Grant No.12101275) and by the Natural Science Research Start-up Foundation of Nanjing University of Posts and Telecommunications to Recruit Talents (Grant No. NY224040). Jun Hu is supported by a fellowship leave award from the City University of New York for the academic year 2025-26, and the Research Foundation of CUNY (PSC-CUNY No. 68140-00 56 and No. GR-00017614). Yi Qi is supported by the NSF of China (Grant No. 12271017). 

\bigskip
\textbf{Data availability statement:}
Data sharing is not applicable to this work as no new data is created or analyzed in the study.

\bibliographystyle{plain}
\bibliography{refs}

\noindent Guangming Hu, Email: 20230210@njupt.edu.cn\\
\noindent{College of Science, Nanjing University of Posts and Telecommunications, Nanjing, 210003, P.R. China} 
\\

\noindent Jun Hu, Email: junhu@brooklyn.cuny.edu or JHu1@gc.cuny.edu\\
\noindent{Department of Mathematics, Brooklyn College of CUNY, Brooklyn, NY 11210, USA and Ph.D. Program in Mathematics, Graduate Center of CUNY, 365 Fifth Avenue, New York, NY 10016, USA} 
\\

\noindent Lishan Li, Email: lishan-li@buaa.edu.cn\\
\noindent{School of Mathematical Sciences, Beihang University, Beijing, 102206, P. R. China}
\\

\noindent Yi Qi, Email: yiqi@buaa.edu.cn\\
\noindent{School of Mathematical Sciences, Beihang University, Beijing, 102206, P.R. China}

\end{document}